\documentclass[a4paper,11pt,reqno,noindent]{amsart}
\usepackage[centertags]{amsmath}
\usepackage{mathtools} 
\usepackage{amsfonts,amssymb,amsthm,dsfont,cases,amscd,esint,enumerate}
\usepackage[T1]{fontenc}
\usepackage[english,frenchb]{babel}
\usepackage[applemac]{inputenc}
\usepackage{newlfont}
\usepackage{color}
\usepackage[body={15cm,21.5cm},centering]{geometry} 
\usepackage{fancyhdr}
\usepackage{enumerate}
\usepackage{tikz}
\usetikzlibrary{trees}

\DeclarePairedDelimiter\ceil{\lceil}{\rceil}

 \definecolor{dgreen}{rgb}{0,0.5,0}

\theoremstyle{plain}
\newtheorem{theor10}{Theorem}
\newenvironment{theor1}
  {\pushQED{\qed}\begin{theor10}}
  {\popQED\end{theor10}}
\newtheorem{prop10}{Proposition}

\newtheorem{cor10}{Corollary}
\newenvironment{cor1}
  {\pushQED{\qed}\begin{cor10}}
  {\popQED\end{cor10}}
\newtheorem{lem0}{Lemma}[section]
\newenvironment{lem}
  {\pushQED{\qed}\begin{lem0}}
  {\popQED\end{lem0}}
\theoremstyle{plain}
\newtheorem{theor0}[lem0]{Theorem}

\newtheorem{prop0}[lem0]{Proposition}
\newenvironment{prop}
  {\pushQED{\qed}\begin{prop0}}
  {\popQED\end{prop0}}
\newtheorem{cor0}[lem0]{Corollary}
\newenvironment{cor}
  {\pushQED{\qed}\begin{cor0}}
  {\popQED\end{cor0}}
\theoremstyle{definition}
\newtheorem{defin0}[lem0]{Definition}
\newenvironment{defin}
  {\pushQED{\qed}\begin{defin0}}
  {\popQED\end{defin0}}
\newtheorem{rems0}[lem0]{Remarks}
\newenvironment{rems}
  {\pushQED{\qed}\begin{rems0}}
  {\popQED\end{rems0}}
\newtheorem{rem0}[lem0]{Remark}
\newenvironment{rem}
  {\pushQED{\qed}\begin{rem0}}
  {\popQED\end{rem0}}
\newtheorem{ex0}[lem0]{Example}

\newtheorem{exs0}[lem0]{Examples}
\newenvironment{exs}
  {\pushQED{\qed}\begin{exs0}}
  {\popQED\end{exs0}}
  \newtheorem{notation0}[lem0]{Notation}

\numberwithin{equation}{section}
\mathchardef\emptyset="001F

\newcommand{\cprime}{$'$}
\newcommand{\Aa}{a}
\newcommand{\N}{\mathbb N}
\newcommand{\e}{\varepsilon}

\newcommand{\R}{\mathbb R}
\newcommand{\Z}{\mathbb Z}
\newcommand{\C}{\mathbb C}
\newcommand{\T}{\mathbb T}

\newcommand{\Kc}{\mathcal K}

\newcommand{\Vc}{\mathcal V}

\newcommand{\Oc}{\mathcal O}

\newcommand{\df}{\mathfrak d}
\newcommand{\Tc}{\mathcal T}

\newcommand{\F}{\mathcal F}
\newcommand{\Lc}{\mathcal L}
\newcommand{\Rc}{\mathcal R}
\newcommand{\Ic}{\mathcal I}

\newcommand{\sinc}{\operatorname{sinc}}

\newcommand{\Sc}{\mathcal S}
\newcommand{\Id}{\operatorname{Id}}
\newcommand{\p}{\mathbb{P}}

\newcommand{\loc}{{\operatorname{loc}}}

\newcommand{\Ld}{\operatorname{L}}

\newcommand{\supp}{\operatorname{supp}}

\newcommand{\step}[1]{\noindent \textit{Step} #1.}

\newcommand{\pr}[1]{\mathbb{P}\left[{#1}\right]}

\newcommand{\expec}[1]{\mathbb{E}\left[{#1}\right]}

\newcommand{\expecm}[1]{\mathbb{E}\big[{#1}\big]}

\newcommand{\vertiii}[1]{{\left\vert\kern-0.25ex\left\vert\kern-0.25ex\left\vert #1 
    \right\vert\kern-0.25ex\right\vert\kern-0.25ex\right\vert}}

\usepackage{color}
\usepackage[colorlinks,citecolor=black,urlcolor=black]{hyperref}

\title[Approximate Floquet-Bloch theory for linear waves]{Approximate normal forms via Floquet-Bloch theory: \\ Nehoro\v{s}ev stability for linear waves\\in quasiperiodic media}

\author[M. Duerinckx]{Mitia Duerinckx}
\author[A. Gloria]{Antoine Gloria}
\author[C. Shirley]{Christopher Shirley}
\address[Mitia Duerinckx]{Universit\'e Paris-Saclay, CNRS, Laboratoire de Math\'ematiques d'Orsay, 91400~Orsay, France \& Universit\'e Libre de Bruxelles, D\'epartement de Math\'ematique, 1050~Brussels, Belgium}
\email{mitia.duerinckx@u-psud.fr}
\address[Antoine Gloria]{Sorbonne Universit\'e, CNRS, Laboratoire Jacques-Louis Lions, 75005~Paris, France \& Institut Universitaire de France (IUF) \& Universit\'e Libre de Bruxelles, D\'epartement de Math\'ematique, 1050~Brussels, Belgium}
\email{gloria@ljll.math.upmc.fr}
\address[Christopher Shirley]{Universit\'e Paris-Saclay, CNRS, Laboratoire de Math\'ematiques d'Orsay, 91400~Orsay, France}
\email{christopher.shirley@universite-paris-saclay.fr}

\begin{document}
\selectlanguage{english}
\maketitle

\begin{abstract}
We study the long-time behavior of the Schrödinger flow in a heterogeneous potential $\lambda V$ with small intensity $0<\lambda\ll 1$ (or alternatively at high frequencies).
The main new ingredient, which we introduce in the general setting of a stationary ergodic potential,
is an approximate
stationary Floquet-Bloch theory that is used to put the perturbed Schr\"odinger operator into approximate normal form. We apply this approach to quasiperiodic potentials and establish a Nehoro\v{s}ev-type stability result. In particular, this ensures asymptotic ballistic transport up to a stretched exponential timescale $\exp(\lambda^{-\frac1{s}})$ for some $s>0$. More precisely, the approximate normal form leads to an accurate long-time description of the Schrödinger flow as an effective unitary correction of the free flow.
The approach is robust and generically applies to linear waves.
For classical waves, for instance, this allows to extend diffractive geometric optics to quasiperiodically perturbed media.
\end{abstract}

\setcounter{tocdepth}{1}
\tableofcontents

\section{Introduction and statement of the main results}\label{chap:intro}

\subsection{General overview}
{ 
Consider the perturbed Schrödinger operator $\Lc_\lambda:=\Lc_0+ \lambda V$ on $\Ld^2(\R^d)$, where $\Lc_0:=-\triangle$ is the Laplacian on the ambient space $\R^d$, where $V:\R^d\to\R$ denotes a quasiperiodic potential (cf.~(QP) in Section~\ref{sec:res-quant}), and where the coupling constant $\lambda\ge0$ measures its intensity, and
let $u_\lambda$ denote the corresponding Schrödinger flow
\[i\partial_tu_\lambda=\Lc_\lambda u_\lambda,\qquad u_\lambda|_{t=0}=u^\circ_\lambda\]
with initial condition $u^\circ_\lambda$.
This well-travelled model describes the motion of an electron in a quasiperiodic scenery.
We are interested in the dynamical properties of the flow on long timescales.

\smallskip

A first approach to such questions focuses on the spectrum of $\Lc_\lambda$.
In case of a periodic potential $V$, a complete answer follows from the classical Floquet-Bloch theory (e.g.~\cite{Reed-Simon-78,Kuchment-16}): the operator has absolutely continuous spectrum and the flow satisfies ballistic transport on all times for all values of $\lambda$ (cf.~\cite{Asch-Knauf}), that is, for all $t>0$
$$
\int_{\R^d} (\frac{|x|}{t})^2 |u^t_\lambda(x)|^2dx \,\gtrsim \int_{\R^d} |x|^2|u^\circ_\lambda(x)|^2dx.
$$
The theory relies on the fact that periodic Schrödinger operators have compact resolvent on the torus (via the Rellich theorem). 
For a quasiperiodic potential as considered here, the story is radically different: quasiperiodic Schrödinger operators are degenerate elliptic operators on the corresponding higher-dimensional torus and compactness fails.
This is not only a technical matter: the spectrum of such operators indeed drastically differs from their periodic counterparts since the quasiperiodic potential may lead to localization (that is, point spectrum).
The situation is particularly well understood in dimension $d=1$ (under suitable assumptions): in the large disorder regime $\lambda\gg1$ the spectrum is pure point with dynamical localization at low frequencies~\cite{FSW-90} and is absolutely continuous with ballistic transport at high frequencies~\cite{E-92}, while in the small disorder regime $\lambda \ll 1$ the spectrum is solely absolutely continuous~\cite{E-92,Zhao-16,Zhao-17}.
In higher dimension $d>1$, although the situation is expected to be similar, the available results are much sparser
--- except in the setting of discrete Schrödinger operators on $\Z^d$~\cite{BGS-02,B-05,B-07,Krueger}.
In the continuum setting, 
the whole question remains open.
Recently, Karpeshina and coauthors~\cite{Karpeshina-Lee-10,Karpeshina-Shterenberg-14} proved the existence of a semi-axis of absolutely continuous spectrum at high frequencies in dimension $d=2$ for a specific class of quasiperiodic potentials (more precisely, with the notation of (QP) below: $M=4$ and $\mathcal F\tilde V$ compactly supported), which was supplemented with a corresponding ballistic transport result~\cite{Karpeshina-Lee-Shterenberg-Stolz-15}.
In contrast to the known situation in dimension $d=1$, this result does however not rule out the presence of point spectrum, nor of other types of transport: the authors identify a subspace of initial data for which ballistic transport holds for all times; the picture remains largely incomplete.
These non-perturbative spectral questions are  difficult, and rigorous approaches are based on argument from multi-scale analysis.

\smallskip

In the present contribution, we take a different path and rather investigate the question of long-time behavior of linear waves in its own right,
regardless of spectral aspects. This perspective turns out to be more flexible (one may indeed look for a statement valid on long but not infinitely long timescales in some perturbative regime)
and robust --- of course,  it only gives partial answers to  the above conjectures.
More precisely, based on a perturbative use of formal Rayleigh-Schr\"odinger series, we introduce a general method in form of an approximate Floquet-Bloch theory, which allows to put the perturbed Schrödinger operator $\Lc_\lambda$ into an approximate normal form.
This leads to a perturbative ballistic transport result up to some stretched exponential timescale both in the small disorder regime or at high frequencies.
Roughly speaking, by a scaling argument, treating a small potential $\lambda V$ with $\lambda\ll1$ and any fixed initial data is expected to be equivalent to treating a fixed potential $V$ and initial data localized at high frequencies $\frac1\lambda\gg1$ (in the same spirit as Aubry duality for quasi-Mathieu operators, e.g.~\cite[Section~6.2]{Damanik-17}).
Our techniques allow to treat both regimes at once (up to minor modifications) and we obtain ballistic transport up to a stretched exponential timescale $\exp(\lambda^{-s})$ for some exponent $s>0$ (depending on the regularity and the quasiperiodic structure of the potential~$V$).
Whereas ballistic transport up to arbitrary algebraic timescale $\lambda^{-n}$ would only require a relatively soft analysis, there is a major gap to reach stretched exponential timescales: a finer analysis of the formal Rayleigh-Schr\"odinger series is needed.

\smallskip

Let us emphasize that this perturbative approach is robust and can deal with a large variety of (regular and Diophantine) quasiperiodic potentials in any dimension $d\ge1$. For general potentials and/or dimension $d\ge3$, this provides the sharpest results available on the long-time behavior of quantum waves in quasiperiodic scenery. Note that in the (very little understood) setting when the potential is not regular or not Diophantine, this approach still applies, and yields ballistic transport up to some algebraic timescale (limitation due to the worse behavior of perturbation series).
In addition, these techniques
are easy to adapt to linear wave equations in general, as we illustrate by treating both quantum and classical waves (classical wave equations with higher-order elliptic operators or other types of linear waves like Maxwell's equations can be treated similarly). For classical waves, this leads us to revisit diffractive geometric optics in quasiperiodically perturbed media~\cite{APR1,APR2}. The approach extends to the discrete setting mutadis mutandis

\smallskip

The rest of this introduction is organized as follows:
In Section~\ref{sec:res-quant} we state our main result on quantum waves in form of an approximate normal form decomposition up to a stretched exponential timescale (Theorem~\ref{th:main-quasi}), while ballistic transport follows as a corollary
(Corollary~\ref{th:main-quasi-tsp}).
To keep this article short we mainly focus on the small disorder regime; the corresponding ballistic transport result at high frequencies is briefly displayed as Corollary~\ref{th:main-quasi-tsp-2}.
In Section~\ref{sec:res-class}, we turn to the case of classical waves (Theorem~\ref{thm:class-wave-qp}).
In that setting, we further specialize the result to initial data that are peaked in frequency, which leads to an effective PDE for classical waves in quasiperiodic disorder in form of a diffractive correction to geometric optics (Corollary~\ref{cor:class-wave-qp}). This extends previous periodic results~\cite{APR1,APR2} to the quasiperiodic setting and is of independent interest.
A general presentation of the method is postponed to Section~\ref{sec:method}.
}

\subsubsection*{Notation}
\begin{enumerate}[$\bullet$]
\item Throughout the article, $d$ denotes the space dimension, we denote by $C\ge1$ any constant that only depends on the dimension and on controlled quantities, the value of which may change from line to line. We use the notation $\lesssim$ (resp.~$\gtrsim$) for $\le C\times$ (resp.~$\ge\frac1C\times$) up to such a multiplicative constant $C$. We write $\simeq$ when both $\lesssim$ and $\gtrsim$ hold. We also use the notation $a=O(b)$ for $a\lesssim b$. We write $\ll$ for $\le\frac1C\times$ for some large enough $C$.
We add subscripts to $C,\lesssim,\gtrsim,\simeq,O(\cdot),\ll$ in order to indicate dependence on other parameters. The notation $a_\lambda=o(b_\lambda)$ stands for $a_\lambda/b_\lambda\to0$ as $\lambda\downarrow0$.
\item We denote by $\hat f(k):=\F[f](k):=\int_{\R^d} e^{-ik\cdot x}f(x)\,dx$ the usual Fourier transform of a smooth function $f$ on $\R^d$. The inverse Fourier transform is then given by
$f(x)=\F^{-1}[\hat f\,](x)=\int_{\R^d} e^{ik\cdot x}\hat f(k)d^*k$
in terms of the rescaled Lebesgue measure $d^*k:=(2\pi)^{-d}dk$.
Likewise, for $M\in \N$, when dealing with periodic functions on the torus $\T^M=[0,2\pi)^M$, we denote by $\hat f(k):=\F[f](k):=\int_{\T^M} e^{-ik\cdot x}f(x)dx$ the associated Fourier coefficients on $\Z^M$.
\item The ball centered at $x$ and of radius $r$ in dimension $n$ is denoted by $B^n(x,r)$, by $B^n_r$ if $x=0$, and by $B^n(x)$ if $r=1$. We drop the superscript $n$ whenever $n=d$.
\item $\Sc(\R^d)$ denotes the Schwartz class (we always implicitly consider complex-valued maps when discussing the Schrödinger equation).
$\N$ denotes the set of natural numbers (including $0$). For all $m\in \N$ and $b\in \N^m$, we set $|b|=\sum_i b_i$.
We write $\langle x\rangle:=(1+|x|^2)^{1/2}$.
\item LHS and RHS stand for ``left-hand side'' and ``right-hand side'', respectively.
\end{enumerate}

\subsection{Quantum waves in a quasiperiodic potential}\label{sec:res-quant}

We consider a quasiperiodic potential $V$ satisfying the following standard smoothness and Diophantine properties.
\begin{enumerate}[\quad(H)]
\item[(QP)] \emph{Quasiperiodic setting:}
$$V(x):=\tilde V(F^Tx),$$
for $M\ge d$, some (winding) matrix $F\in\R^{d\times M}$ (the transpose of which is denoted by~$F^T$), and some lifted map $\tilde V\in C(\T^M)$
on the $M$-dimension torus $\T=[0,2\pi)^d$.
Assume that the winding matrix $F\in\R^{d\times M}$ satisfies a Diophantine condition, that is, for some $r_0>0$,
\begin{align}\label{eq:cond-Dioph-Th1}
|F\xi|\ge\frac1C|\xi|^{-r_0}\qquad\text{for all $\xi\in\Z^M\setminus\{0\}$},
\end{align}
and that the lifted map $\tilde V$ is Gevrey-regular, that is, for some $\alpha>0$, 
\begin{align}\label{eq:cond-Gevrey-Th1}
\|\mathds{1}_{|\cdot|> K}\F \tilde V \|_{\Ld^1} \le \exp(-K^\alpha)\qquad\text{for all $K\ge0$}.
\end{align}
\end{enumerate}

In this setting, our main result shows that for all smooth initial data the Schr\"odinger flow in the small disorder regime remains close to an effective unitary correction of the free flow up to some stretched exponential timescale. The general strategy of the proof is described in Section~\ref{sec:strat} in the perspective of normal forms.

\begin{theor1}[Nehoro\v{s}ev-type stability in the small disorder regime]\label{th:main-quasi}
Consider the quasiperiodic setting~\emph{(QP)}, let $u^\circ\in\Sc(\R^d)$, and consider the Schrödinger flow
\begin{align}\label{eq:schr-V}
i\partial_tu_\lambda=(-\triangle+\lambda V) u_\lambda,\qquad u_\lambda|_{t=0}=u^\circ.
\end{align}
For all $0<\lambda\ll1$, it satisfies for any $\gamma<\frac1{2d+1}$,
\begin{gather}\label{eq:concl-1(i)}
\sup_{0\le\lambda'\le\lambda}\,\sup_{0\le t\le T(\lambda)}\big\|u_{\lambda'}^t-U_{\lambda'}^{\ell(\lambda);t}\big\|_{\Ld^2}\,\lesssim_{u^\circ}\,\lambda^{\gamma},\\
T(\lambda):=\exp\big(\lambda^{-\frac1s}\big),
\qquad
\ell(\lambda):=\ceil{\lambda^{-\frac1s}},
\qquad
s:=s(M,r_0,\alpha,\gamma,d) >0,\nonumber
\end{gather}
where the effective flow $U_\lambda^{\ell;t}$ is defined by
\begin{align}\label{eq:def-Uell-appr}
U_{\lambda}^{\ell;t}(x) := \int_{\R^d} e^{ik\cdot x} e^{-it(|k|^2+\tilde \kappa_{k,\lambda}^\ell)} \hat u^\circ(k)\,d^*k,
\end{align}
in terms of some explicitly given symbol $\tilde \kappa_{k,\lambda}^\ell:=\lambda\sum_{n=0}^\ell\lambda^n\tilde\nu_k^n$ (cf.~Definition~\ref{def:taylor-waves}). 
(When restricting~\eqref{eq:concl-1(i)} to a shorter timescale $T'(\lambda)\le T(\lambda)$, only a smaller number of phase corrections is needed and we can choose $\ell'(\lambda):=\lceil\tfrac1{|\!\log\lambda|}\log(T'(\lambda))\rceil\le\ell(\lambda)$.)
\end{theor1}

\medskip

In particular, the above stability result leads to ballistic transport up to the stretched exponential timescale $T(\lambda)=\exp\big(\lambda^{-\frac1s}\big)$.
Although ballistic transport is expected to hold for all times~\cite{E-92,Zhao-16,Zhao-17,Karpeshina-Lee-Shterenberg-Stolz-15}, this result is new and stands out by its generality as it is established in any dimension and under a mere Diophantine condition.
More precisely, for $m\ge1$, we define the rescaled moments of the flow,
\begin{align}\label{eq:def-moment-Mm}
M_m^t(u_\lambda)\,:=\,\big\|\big(\tfrac{|\cdot|}{t}\big)^mu_\lambda^t\big\|_{\Ld^2},
\end{align}
which can be viewed as measuring the asymptotic ballistic velocity of the wavefunction.
We then show that this velocity $M_m^t(u_\lambda)$ remains close to the velocity $M_m^t(u_0)$ of the free flow $u_0=e^{it\triangle}u^\circ$, which can be explicitly computed and obviously remains of order $1$.

\begin{cor1}[Asymptotic ballistic transport in the small disorder regime]\label{th:main-quasi-tsp}
Consider the quasiperiodic setting~\emph{(QP)}, let $u^\circ\in\Sc(\R^d)$, and consider the Schrödinger flow $u_\lambda$ as in~\eqref{eq:schr-V}.
Then there holds
\begin{align*} 
\lim_{t,\lambda}\big|M_m^t(u_\lambda) - M_m^t(u_0)\big|\,=\,0
\end{align*}
as $t\uparrow\infty$ and $\lambda\downarrow0$ in the regime $t\le \exp(\lambda^{-\frac1s})$
for some $s=s(M,m,r_0,\alpha,d)>0$, where $u_0^t=e^{it\triangle}u^\circ$ denotes the free flow.
\end{cor1}

The proofs of Theorem~\ref{th:main-quasi} and Corollary~\ref{th:main-quasi-tsp} are the object of Section~\ref{chap:pr-taylorbloch} and are based on preliminary estimates established in Section~\ref{chap:QP} on formal Rayleigh-Schrödinger series.
Note that Section~\ref{chap:pr-taylorbloch} is presented in such a way that it may be easily adapted to situations where the estimates of Section~\ref{chap:QP} differ
from the ones considered here.

\begin{rems}\mbox{}\label{rem:QP-res}
\begin{enumerate}[$\bullet$]
\item
\emph{Regularity and algebraic assumptions:}\\
As clear from the proof, the largest timescale allowed by this approach both depends on the regularity of $\tilde V$ and on the algebraic properties of $F$, which is not surprising in view of the known results in 1D (e.g.~\cite{Damanik-17}).
Stretched exponential timescales are obtained only for Gevrey-regular $\tilde V$ and Diophantine $F$ as in~(QP).
On the one hand, if $F$ is Diophantine but if we decrease the regularity of $\tilde V$, the allowed timescale is shortened: for algebraically decaying $\F \tilde V$ the same results hold up to some corresponding algrebraic timescale.
On the other hand, if $\tilde V$ is a trigonometric polynomial (that is, $\F\tilde V$ compactly supported) but if $F$ is only irrational (that is, $F\xi\ne0$ for all $\xi\in\Z^M\setminus\{0\}$), then the results hold up to any algebraic timescale (cf.\@ Remark~\ref{rem:per++}).
\item\emph{Peaked initial data in Fourier space:}\\
Provided $u^\circ$ has a fixed compact support in Fourier space, the error estimates in Theorem~\ref{th:main-quasi} depend on $u^\circ$ via its $\Ld^2(\R^d)$-norm only,
so that they hold uniformly for initial data that are peaked in Fourier space, e.g.\@ of the form $u^\circ_\e(x)=\e^{d/2}e^{ik_0\cdot x}g(\e x)$ with $0<\e\ll1$, $\hat g\in C^\infty_c(\R^d)$, and $k_0$ in the non-resonant set $\Oc$ (cf.~\eqref{def:non-resonant}).
Choosing a scaling relation $\e=\lambda^\beta$ for some $\beta>0$, injecting this in the formula for $U_{\lambda}^\ell$, and Taylor expanding $k\mapsto\tilde\kappa_{k,\lambda}^\ell$ around $k_0$, we are led to an effective PDE.
We refer to~Corollary~\ref{cor:class-wave-qp} below for the corresponding result in case of classical waves, for which we provide details.
\item \emph{Periodic setting:}\\
For a periodic disorder $V$,
perturbation series are indeed summable, so the error estimates are the best possible: we may take $\ell\uparrow\infty$ in~\eqref{eq:concl-1(i)} and the results hold for all $t\ge0$ and $\lambda\ll1$ without any restriction, thus recovering the classical results of~\cite{Asch-Knauf}.
The proof follows from the additional periodic estimates in Remarks~\ref{rem:per+} and~\ref{rem:per++}.
\qedhere
\end{enumerate}
\end{rems}

\medskip

As emphasized in the introduction, all our results hold both in the small disorder regime or at high frequencies.
To illustrate  this, we display a corresponding 
asymptotic ballistic transport result at high frequencies, in a form to be compared to~\cite[Theorem~1.1]{Karpeshina-Lee-Shterenberg-Stolz-15}.
Whereas for small disorder a simple approximation argument allows to consider any smooth initial data, the situation is more delicate at high frequencies: although for initial data of the specific form $u_\lambda^\circ=\lambda^{-\frac d2}u^\circ(\lambda^{-1}\cdot)$ with $u^\circ\in\Sc(\R^d)$ and $\int_{\R^d} u^\circ=0$ we could directly adapt all previous results by scaling, some additional care is needed to treat more general initial data. As a consequence, the result below, which contains a ballistic lower bound, yields no characterization of the asymptotic velocity. The proof is displayed in Section~\ref{sec:pr-cor2}.
\begin{cor1}[Existence of asymptotic ballistic transport at high frequencies]\label{th:main-quasi-tsp-2}
Consider the quasiperiodic setting~\emph{(QP)} and let $\gamma<\frac1{2d+1}$.
There exists a subset $G\subset\R^d$ that is extensive in the sense that for all $n\ge0$,
\[\big|\big(B_{2^n}\setminus B_{2^{n-1}}\big)\setminus G\big|\lesssim \big|B_{2^n}\setminus B_{2^{n-1}}\big|^{1-2\gamma},\]
with the following property: For any sequence $(u_\lambda^\circ)_\lambda\subset\Sc(\R^d)^d$ such that $\hat u_\lambda^\circ$ is supported in $\R^d\setminus B_{\lambda^{-1}}$ 
and satisfies
\begin{itemize}
\item[(i)] the unit bound $\|u_\lambda^\circ\|_{\Ld^2}\le 1$,
\item[(ii)] the  localization in frequencies (upper bound) at scale $\lambda^{-1}$: $\|\langle\cdot\rangle^3\hat u_\lambda^\circ\|_{\Ld^2}\le \lambda^{-3}$,
\item[(iii)] the  localization in frequencies (lower bound) on $G$: $\|\langle\cdot\rangle \mathds1_G \hat u_\lambda^\circ\|_{\Ld^2}\ge \frac12 \lambda^{-1}$,
\end{itemize}
the corresponding Schrödinger flow $u_\lambda$ solution of 
\[i\partial_tu_\lambda=(-\triangle+V) u_\lambda,\qquad u_\lambda|_{t=0}=u_\lambda^\circ,\]
satisfies for all $0<\lambda\le\lambda_0$ and $\|\langle\cdot\rangle^3u_\lambda^\circ\|_{\Ld^2}^\frac13\le t\le \exp(\lambda^{-\frac1s})$,
\[\int_{\R^d} \big( \tfrac{\lambda |x|}{t}\big)^2 |u_\lambda^t(x)|^2dx \ge \tfrac1{4}  ,\]
where $\lambda_0:=\lambda_0(L,M,r_0,\alpha,\gamma,d)>0$ and $s:=s(M,r_0,\alpha,\gamma,d)$.
\end{cor1}
Let us comment on assumptions (ii) \& (iii), and on the ballistic lower bound. Assumption (ii) ensures that the frequencies of $u_\lambda^\circ$ are at most of order $\lambda^{-1}$ (measured in a weighted $\Ld^2$-norm in Fourier space). Assumption (iii) ensures that the projection of $\hat u^\circ_\lambda$ on the extensive non-resonant set $G$ has enough mass at frequencies of order $\lambda^{-1}$ (measured in a weighted $\Ld^2$-norm in Fourier space).
In the case of the free flow (for which one may take $G=\R^d$), assumptions (i)---(iii) would ensure the same ballistic lower bound as above (which shows its optimality), albeit for infinite times.

\subsection{Classical waves in quasiperiodic media}\label{sec:res-class}

We now consider  the corresponding perturbed classical wave flow
\begin{align}\label{eq:class-a}
\partial_{tt}^2 u_\lambda=\nabla \cdot (\Id+\lambda a) \nabla u_\lambda,\qquad u_\lambda|_{t=0}=u^\circ, \qquad \partial_t u_\lambda|_{t=0}=v^\circ,
\end{align}
where the matrix field $a$ is quasiperiodic in the following sense.
\begin{enumerate}[\quad(H)]
\item[(QP$'$)] \emph{Quasiperiodic setting:}
$$a(x):=\tilde a(F^Tx),$$
for $M\ge d$, some (winding) matrix $F\in\R^{d\times M}$ (the transpose of which is denoted by~$F^T$), and some lifted map $\tilde a\in C(\T^M)$.
Assume that the winding matrix $F\in\R^{d\times M}$ satisfies the Diophantine condition~\eqref{eq:cond-Dioph-Th1} for some $r_0>0$, and that the lifted map $\tilde a$ is Gevrey-regular in the sense of~\eqref{eq:cond-Gevrey-Th1} for some $\alpha>0$.
\end{enumerate}
For classical waves (as opposed to quantum waves), a large body of literature is devoted to the identification of effective equations to describe the wave flow on long timescales in perturbed media, e.g.~\cite{APR1,APR2} where diffractive corrections to geometric optics are obtained in case of a periodic perturbation of a periodic medium.
In the following, we extend such results to quasiperiodic perturbations.
We start with a result for fixed initial data; the proof follows that of Theorem~\ref{th:main-quasi} and the main modifications are indicated in Section~\ref{app:gen}.

\begin{theor1}[Nehoro\v{s}ev-type stability for classical waves]\label{thm:class-wave-qp}
Consider the quasiperiodic setting~\emph{(QP$'$)}, let $(u^\circ,v^\circ)\in\Sc(\R^d)^2$, assume that the Fourier transform $\hat v^\circ$ of $v^\circ$ is supported away from $0$, and consider the classical wave flow $u_\lambda$ as in~\eqref{eq:class-a}.
For all $0<\lambda\ll1$, it satisfies the estimate~\eqref{eq:concl-1(i)} of Theorem~\ref{th:main-quasi}
in terms of the effective flow
\begin{equation*}
U_{\lambda}^{\ell;t}(x):= \int_{\R^d}e^{ik\cdot x}\bigg(\cos\Big(t\sqrt{|k|^2+\tilde \kappa_{k,\lambda}^{\ell}}\Big)\,\hat u^\circ(k) +t\sinc\Big(t\sqrt{|k|^2+\tilde \kappa_{k,\lambda}^{\ell}}\Big)\,\hat v^\circ(k)\bigg)d^*k,
\end{equation*}
with $\sinc x:=\frac{\sin x}x$, for some explicitly given symbol $\tilde \kappa_{k,\lambda}^\ell$ (cf.~Section~\ref{sec:7.3}).
\qedhere
\end{theor1}

Next, we specialize this result to the case when the initial data are peaked in Fourier space, and replace equation~\eqref{eq:class-a} by
\begin{gather}\label{eq:class-a-wc}
\partial_{tt}^2 u_{\lambda,\e}=\nabla \cdot (\Id+\lambda a) \nabla u_{\lambda,\e},
\\
u_{\lambda,\e}(x)|_{t=0}=\e^{\frac d2} e^{ik_0 \cdot x} u^\circ(\e x), \qquad \partial_t u_{\lambda,\e}(x)|_{t=0}=\e^\frac d2 e^{ik_0 \cdot x}v^\circ(\e x),\nonumber
\end{gather}
for some non-resonant $k_0 \in \Oc\setminus\{0\}$ (cf.~\eqref{def:non-resonant}). (Note that the case $k_0=0$ is of a very different nature as it is about zooming at the bottom of the spectrum.)
In this setting, for $\lambda,\e\downarrow0$, the effective flow defined in Theorem~\ref{thm:class-wave-qp} solves an explicit PDE.
As in~\cite{APR1,APR2}, the result is naturally expressed by factorizing out the free flow and the group velocity $\pm\frac{k_0}{|k_0|}$.
\begin{cor1}[Diffractive geometric optics]\label{cor:class-wave-qp}
Under the assumptions of Theorem~\ref{thm:class-wave-qp}, consider the classical wave flow $u_{\lambda}$ as in~\eqref{eq:class-a-wc} with non-resonant $k_0\in \Oc\setminus\{0\}$ (cf.~\eqref{def:non-resonant}) and with the scaling relation $\e=\lambda^\beta$ for some $\beta>0$.
Further assume that $(u^\circ,v^\circ)$ has compactly supported Fourier transform and that the lifted map $\tilde a$ is a trigonometric polynomial.
For all $t\ge0$, $0<\lambda\ll1$, and $\ell\ge1$, we have for some $0<\gamma=\gamma(\beta,\ell)\le 1$,
\[\|u_\lambda^t-\tilde U_{\lambda}^{\ell;t}\|_{\Ld^2}\,\lesssim_{k_0,\ell,u^\circ,v^\circ}\,\lambda^\gamma(1+\lambda^\ell t),\]
where the effective flow is given by $\tilde U_{\lambda}^{\ell;t}:=\frac12(\tilde U_{\lambda,+}^{\ell;t}+\tilde U_{\lambda,-}^{\ell;t})$ with
\[\tilde U_{\lambda,\pm}^{\ell;t}(x)\,:=\,\frac{\e^{\frac{d}2}}2 e^{ik_0\cdot (x\mp t\frac{k_0}{|k_0|})}A_{\lambda,\pm}^{\ell;t}\big(\e (x\mp t\tfrac{k_0}{|k_0|})\big),\]
where the amplitudes $A_{\lambda,\pm}^{\ell}$ solve the effective diffractive PDEs
\begin{multline}\label{e.diffrac-corr-QP}
\bigg(i\partial_t \pm\e^2\frac{|k_0|^2\triangle-(k_0\cdot\nabla)^2}{2|k_0|^3}\mp\lambda\frac{\nu_{k_0}^0}{2|k_0|}\bigg)A_{\lambda,\pm}^{\ell}\\
\,=\,\pm \sum_{0\le m\le\ell\atop0\le p\le \lfloor\ell/\alpha\rfloor}\mathds1_{m+p\ge2\atop(m,p)\ne(0,2)}\lambda^{m}\e^{p}C_{m,p}(k_0)\odot(-i\nabla)^{\otimes p}A_{\lambda,\pm}^{\ell},
\end{multline}
with initial data $A_{\lambda,\e,\pm}^{\ell}|_{t=0}=u^\circ\pm\frac{iv^\circ}{|k_0|}$ and with $\odot$ denoting the total contraction of tensors, where $C_{m,p}(k_0)$ are explicitly given $p$-th order tensors for all $m,p\ge0$ (cf.~Section~\ref{sec:7.5}).
\end{cor1}

For the specific choice $\lambda=\e^2$, this result extends~\cite{APR1,APR2} to the case of quasiperiodic perturbations (up to changing variables $(\e x,\e t)\leadsto (x',t')$).
In terms of diffractive geometric optics, it reads as follows:
\begin{enumerate}[$\bullet$]
\item For times $t\ll \e^{-2}\wedge \lambda^{-1}$, equation~\eqref{e.diffrac-corr-QP} for the amplitudes reduces to $i\partial_tA_{\lambda,\pm}^{\ell}=0$ up to negligible terms. Hence the flow $u_{\lambda}^{t}(x)$ simply remains close to
\[\frac{\e^\frac d2}2\sum_\pm e^{ik_0\cdot (x\mp t\frac{k_0}{|k_0|})} \Big(u^\circ\pm\frac{v^\circ}{i|k_0|}\Big)\big(\lambda^\beta (x\mp\tfrac{k_0}{|k_0|} t)\big),\]
that is, the geometric optic approximation with group velocity $\pm\frac{k_0}{|k_0|}$.
\item For times $t\gtrsim \e^{-2}\wedge \lambda^{-1}$, the transported profiles $u^\circ\pm\frac{v^\circ}{i|k_0|}$ are further deformed and their spread is described by the Schrödinger equation~\eqref{e.diffrac-corr-QP}.
This is known as diffractive geometric optics.
The first diffractive correction due to the background appears on times $t\gtrsim\e^{-2}$, while the first diffractive correction due to the disorder $a$ appears on times $t\gtrsim\lambda^{-1}$ and is nonzero whenever $\nu_{k_0}^0:=k_0 \cdot(\int_{\T^M}\tilde a) k_0\neq 0$.
\end{enumerate}

\section{General method and relation to the literature}\label{sec:method}

{
In this section, we start with the description of  the general approach in the perspective of normal forms, we draw a comparison to the recent literature on the topic, and we give a spectral interpretation of the method (also briefly commenting on the corresponding random setting~\cite{DS-2}).
We focus on the small disorder regime.

\subsection{General method: approximate normal forms}\label{sec:strat}

Since the Fourier transform $\F$ diagonalizes $\Lc_0:=-\triangle$ and since the perturbation $\lambda V$ has small intensity $\lambda\ll1$, we may look for a small deformation $\F_\lambda$ of the Fourier transform that diagonalizes the perturbed operator $\Lc_\lambda$. If it exists, $\F_\lambda$ is known as a {\it Bloch wave transform}.
In other words, this amounts to writing $\Lc_\lambda$ in {\it normal form}: decomposing the $\F_\lambda$-symbol of $\Lc_\lambda$ as $k\mapsto|k|^2+\kappa_{k,\lambda}$ (that is, as a perturbation of the $\F$-symbol $k\mapsto|k|^2$ of $\Lc_0$), we would find
\begin{align}\label{eq:Lnormaldec}
\Lc_\lambda\Tc_\lambda = \Tc_\lambda(\Lc_0+\Kc_\lambda),
\end{align}
in terms of the transform $\Tc_\lambda:=\F_\lambda^{-1}\F$ and of the pseudo-differential operator $\Kc_\lambda:=\kappa_{-i\nabla,\lambda}$ (a multiplication operator in Fourier space).
Alternatively, at the level of the Schrödinger flow, this yields
\begin{align}\label{eq:Lnormaldec-flow}
u_\lambda^t=e^{-it\Lc_\lambda}u^\circ=\Tc_\lambda e^{-it(\Lc_0+\Kc_\lambda)}\Tc_\lambda^{-1}u^\circ.
\end{align}
Since $\F_\lambda$ is a small deformation of $\F$, the transform $\Tc_\lambda$ should be close to identity, in which case the flow $u_\lambda$ would be close to the effective flow $U_\lambda^t:=e^{-it(\Lc_0+\Kc_\lambda)}u^\circ$.
The operator~$\Kc_\lambda$ is viewed as an effective unitary correction of the free flow due to the perturbation $\lambda V$ on long timescales.

We describe below the strategy  in the general setting of a stationary ergodic potential $V$, constructed on some probability space $(\Omega,\p)$; see Section~\ref{chap:stat-BF} for precise definitions.
As shown in Examples~\ref{ex:stat}, this encompasses both the case of a stationary random potential, and the case of a periodic or quasiperiodic potential as considered here:
\begin{enumerate}[$\bullet$]
\item In the periodic setting, $\Omega$ reduces to the torus $\T^d$ endowed with the Lebesgue measure and a realization $\omega\in\Omega$ corresponds to a translation of the periodic potential $W(x,\omega):=V(\omega+x)$.
\item In the quasiperiodic setting~(QP), $\Omega$ coincides with the underlying high-dimensional torus~$\T^M$ and $W(x,\omega):=\tilde V(\omega+F^Tx)$.
\end{enumerate}
On the one hand, this unified setting leads to a fruitful comparison between the periodic, quasiperiodic, and random cases (cf.~Section~\ref{sec:spectral}). On the other hand, it suggests to exploit averaging wrt the translation $\omega$, which is a crucial ingredient in our approach. In the periodic and quasiperiodic settings we use Sobolev embeddings to get rid of averaging a posteriori and to set $\omega=0$ in the final results.
Our strategy splits into five steps.

\begin{enumerate}[(S1)]
\item \emph{Stationary Floquet-Bloch fibration.}\\
Rather than considering the Schrödinger operator $\Lc_\lambda$ on $\Ld^2(\R^d)$ for fixed $\omega\in\Omega$, we exploit averaging wrt the realization and view $\Lc_\lambda$ as an operator on the extended Hilbert space $\Ld^2(\R^d\times\Omega)$.
This new operator can be partially diagonalized via Fourier, which leads to a natural decomposition as a direct integral of simpler fibered operators $\Lc_{k,\lambda}$ on the elementary space $\Ld^2(\Omega)$, for $k\in\R^d$.
The (centered) fibered operators take the form
\[\qquad\Lc_{k,\lambda}:=e^{-ik\cdot x}(-\triangle+\lambda V)e^{ik\cdot x}-|k|^2=\Lc_{k,0}+\lambda V,\qquad \Lc_{k,0}:=-\triangle-2ik\cdot\nabla,\]
and act on $\Ld^2(\Omega)$ viewed as the space of stationary fields.
In particular, the Schrödinger flow $u_\lambda$ can be decomposed as
\begin{align}\label{eq:flow-fibration}
\qquad
u_\lambda^t(x)=\int_{\R^d} \hat u^\circ(k)\,e^{ik\cdot x-it|k|^2}\,\big(e^{-it\Lc_{k,\lambda}}1\big)(x,\omega)\,d^*k,
\end{align}
in terms of the fibered evolutions $e^{-it\Lc_{k,\lambda}}1$ on $\Ld^2(\Omega)$.
We refer to this decomposition as the stationary Floquet-Bloch fibration since it extends the well-known corresponding construction in the periodic setting~\cite{Reed-Simon-78,Kuchment-16}.
\item \emph{Fibered perturbation problem.}\\
By Step~(S1), the Schrödinger operator $\Lc_\lambda$ on $\Ld^2(\R^d\times\Omega)$ splits into a family of fibered operators $(\Lc_{k,\lambda})_{k\in\R^d}$ on $\Ld^2(\Omega)$.
More precisely, the constant function $1$ is an eigenfunction of $\Lc_{k,0}$ associated with the eigenvalue $0$, and the decomposition~\eqref{eq:flow-fibration} shows that it suffices to study the perturbation of this constant eigenfunction in the spectrum of $\Lc_{k,\lambda}$. 
If the fibered operator $\Lc_{k,0}$ on $\Ld^2(\Omega)$ had discrete spectrum, the Kato-Rellich perturbation theory~\cite{Kato-95,Rellich} would ensure the existence of a (local) branch $\lambda\mapsto (\kappa_{k,\lambda},\psi_{k,\lambda})$ of eigenvalues and eigenfunctions of $\Lc_{k,\lambda}$ starting at the eigenvalue $0$ and the constant eigenfunction.
For periodic $V$, the operators $\Lc_{k,0}$ indeed have discrete spectrum and $0$ is typically a simple eigenvalue, while for quasiperiodic $V$ the eigenvalue $0$ is embedded in dense pure point spectrum and for random $V$ it is embedded in absolutely continuous spectrum~\cite{DS-2}.
In the latter two cases, $0$ is  not isolated in the spectrum of $\Lc_{k,0}$,
and there is no general perturbation theory at our disposal that would allow to construct branches $\lambda\mapsto (\kappa_{k,\lambda},\psi_{k,\lambda})$. Such branches could actually be not smooth or even not exist (cf.~Section~\ref{sec:spectral}).
\item \emph{Approximate Bloch waves.}\\
Rather than investigating the existence of a branch $\lambda\mapsto (\kappa_{k,\lambda},\psi_{k,\lambda})$ as in Step~(S2), we consider the formal Rayleigh-Schrödinger perturbation series that would describe such a branch, should it exist, the terms of which are characterized by abstract PDEs in $\Ld^2(\Omega)$. By truncating this series and also regularizing the terms if needed, we are led to an approximate branch $\lambda\mapsto (\kappa_{k,\lambda,\mu}^\ell,\psi_{k,\lambda,\mu}^\ell)$, which is referred to as a branch of approximate Bloch waves, where $\ell$ and $\mu$ are truncation and regularization parameters. Regularization is in fact only needed in the random setting (cf.~Section~\ref{sec:spectral}), while truncation is enough in the quasiperiodic setting.
By construction, $(\kappa_{k,\lambda,\mu}^\ell,\psi_{k,\lambda,\mu}^\ell)$ satisfies the eigenvalue equation for $\Lc_{k,\lambda}$ up to an error, called eigendefect. In other words, $\kappa_{k,\lambda,\mu}^\ell$ belongs to the pseudospectrum, e.g.~\cite{Trefethen-05}, with the precision given by the bound on the eigendefect.
In contrast with approximate spectrum, our analysis further requires to control the constructed pseudomode $\psi_{k,\lambda,\mu}^\ell$ itself in $\Ld^2(\Omega)$, and in particular its closeness to the constant function $1$.
\item \emph{Control of the eigendedect.}\\
Using the approximate Bloch waves of Step~(S3) instead of exact ones yields errors involving the eigendefect, which thus needs to be controlled.
In order to reach optimality in the scaling in $\ell$, we proceed to a crucial resummation of the Rayleigh-Schrödinger series using a tree-counting argument (which allows one to replace naive bounds of the form $C^{n^2}$ by 
bounds of the form $C^n$, cf.~Proposition~\ref{prop:sol-nonlin-rec} below).
The control of the eigendefect depends in an essential way on the nature of the perturbation $V$.
These estimates are the main technical ingredient required by the approach.
\item\emph{Approximate normal form decomposition.}\\
The above construction naturally leads to defining the following transform,
\[\qquad\Tc_{\lambda,\mu}^\ell v(x,\omega):=\int_{\R^d}\hat v(k)\,e^{ik\cdot x}\,\psi_{k,\lambda,\mu}^\ell(x,\omega)\,d^*k,\qquad v\in \Ld^2(\R^d).\]
Considering the pseudo-differential operator $\Kc_{\lambda,\mu}^\ell:=\kappa_{-i\nabla,\lambda,\mu}^\ell$, we obtain
\begingroup\allowdisplaybreaks
\begin{multline}\label{eq:approx-normal}
\qquad\big(\Lc_\lambda\Tc_{\lambda,\mu}^\ell v\big)(x,\omega)=\big(\Tc_{\lambda,\mu}^\ell(\Lc_0+\Kc_{\lambda,\mu}^\ell)v\big)(x,\omega)\\
+\underbrace{\int_{\R^d}\hat v(k)\,e^{ik\cdot x}\,(\Lc_{\lambda,k}-\kappa_{k,\lambda,\mu}^\ell)\,\psi_{k,\lambda,\mu}^\ell(x,\omega)\,d^*k}_{\displaystyle =: \big(\Rc_{\lambda,\mu}^\ell v\big)(x,\omega)},
\end{multline}
\endgroup
where the residual operator $\Rc_{\lambda,\mu}^\ell$ involves the eigendefect and is hopefully shown to be small in Step~(S4). This yields an approximate normal form decomposition of~$\Lc_\lambda$ to be compared to~\eqref{eq:Lnormaldec}.
In addition, since $\psi_{k,\lambda,\mu}^\ell$ is close to the constant function~$1$, the transform $\Tc_{\lambda,\mu}^\ell$ is close to identity.
An effective description of the flow on long timescales then follows as in~\eqref{eq:Lnormaldec-flow}: more precisely, if $\Rc_{\lambda,\mu}^\ell$ is of order $O(g(\lambda))$ in some scaling of $\mu$ and $\ell$, then on the timescale $t\le O(g(\lambda)^{-1})$ the flow $u_\lambda$ remains close to the effective flow
\[\qquad U_{\lambda,\mu}^{\ell;t}:=e^{-it(\Lc_0+\Kc_{\lambda,\mu}^\ell)}u^\circ.\]
\end{enumerate}

\noindent
Steps~(S1)--(S3) (and the crucial resummation argument in Step~(S4)) are detailed in Section~\ref{chap:blochballi} below, which serves as a basis for the rest of this contribution.
The truncation in Step~(S3) and the dynamical properties of approximate branches in Step~(S5) are inspired by the treatment of the classical wave operator with heterogeneous coefficients at low wave number
in~\cite{BG-16}.

\subsection{Comparison to the literature}

The existence of an (exact) Bloch wave transform~$\F_\lambda$ in the quasiperiodic setting is a difficult and open question.
In dimension $d=2$ for a specific class of quasiperiodic potentials,
the already mentioned works by Karpeshina and coauthors~\cite{Karpeshina-Lee-10,Karpeshina-Shterenberg-14,Karpeshina-Lee-Shterenberg-Stolz-15} show that an (exact) normal form decomposition~\eqref{eq:Lnormaldec} holds with $\Tc_\lambda$ replaced by some non-bijective map. More precisely, there is a large set of initial data (at high frequencies) for which~\eqref{eq:Lnormaldec-flow} holds, which in particular implies that ballistic transport holds on all times for such initial data. This can be viewed as a KAM-type result in an infinite-dimensional setting.
Pursuing this analogy with nearly integrable Hamiltonian systems, the so-called Arnol\cprime{}d diffusion phenomenon~\cite{Arnold-64} would suggest that~\eqref{eq:Lnormaldec-flow} could possibly  only hold for some (typically strict) subspace of initial data and under strong assumptions on the quasiperiodic structure.
However, in the small disorder regime $\lambda\ll1$, recent results in 1D~\cite{Zhao-16,Zhao-17} rather advocate that the normal form decomposition~\eqref{eq:Lnormaldec} might generically hold true with bijective $\Tc_\lambda$. The main difficulty for such a result is related to the dense crossings of eigenvalues as explained in the spectral interpretation below.
In contrast, in the present contribution, we focus on approximate versions~\eqref{eq:approx-normal} of the normal form decomposition~\eqref{eq:Lnormaldec}. The validity up to a stretched exponential timescale is obtained by optimizing wrt the truncation parameter $\ell$ in~\eqref{eq:approx-normal}.
Rather than a KAM result, this has the flavor of a Nehoro\v{s}ev stability result~\cite{Nekhoroshev-2,Nekhoroshev-1,Nekhoroshev-3}, which holds for all initial data and for (essentially) any quasiperiodic structure as stated in Theorem~\ref{th:main-quasi}. As such, the present work nicely completes the recent KAM-type results of~\cite{Karpeshina-Lee-10,Karpeshina-Shterenberg-14,Karpeshina-Lee-Shterenberg-Stolz-15,Zhao-16,Zhao-17}.

\subsection{A new spectral perspective}\label{sec:spectral}
As explained in Step~(S2), our general approach reduces the description of the Schrödinger flow $u_\lambda$ to the perturbation analysis of the eigenvalue $0$ in the spectrum of the fibered operators $\Lc_{k,\lambda}$. The drastic difference of the expected transport behaviors in the periodic, quasiperiodic, and random settings can then be related to the fundamental difference of the corresponding perturbation problems.

\begin{enumerate}[$\bullet$]
\item \emph{Periodic case: perturbation of discrete spectrum.}\\
If the spectrum of $\Lc_{k,0}$ is discrete (and if the perturbation $\lambda V$ is $\Lc_{k,0}$-compact), as in case of a periodic disorder $V$, the Kato-Rellich perturbation theory ensures the existence and analyticity of (most of) the fibered branches $\lambda\mapsto(\kappa_{k,\lambda},\psi_{k,\lambda})$, which are then given by their (convergent) Rayleigh-Schrödinger series for $\lambda\ll1$. As discussed in Remark~\ref{rem:per+}, the analyticity of the branches cannot hold uniformly when the fibration parameter $k$ gets close to the so-called diffraction hyperplanes, that is, to the values such that the eigenvalue $0$ is not simple in the spectrum of $\Lc_{k,0}$.
Regardless of this subtlety, in the periodic case, a Bloch transform $\F_\lambda$ is known to exist, is bijective, and is strongly close to the Fourier transform $\F$ (e.g.~\cite{Kuchment-16}).
\item \emph{Quasiperiodic case: perturbation of dense pure point spectrum.}\\
If the spectrum of $\Lc_{k,0}$ is dense pure point with a (typically simple) eigenvalue at $0$, as is the case for a quasiperiodic disorder $V$, and if a branch of eigenvalues $\lambda\mapsto\kappa_{k,\lambda}$ exists, then it is typically not analytic due to dense crossings with other eigenvalues. In particular, the formal Rayleigh-Schrödinger perturbation series should not converge in that setting. Our approach then amounts to using the Rayleigh-Schrödinger series as an asymptotic series describing a likely existing branch. Such perturbative information are of course not strong enough to obtain conclusions on all timescales, as opposed to the non-perturbative approaches in~\cite{Karpeshina-Lee-10,Karpeshina-Shterenberg-14,Karpeshina-Lee-Shterenberg-Stolz-15}.
\item \emph{Random case: perturbation of absolutely continuous spectrum.}\\
If the spectrum of $\Lc_{k,0}$ consists of a simple eigenvalue at $0$ embedded in an absolutely continuous part, as is typically the case for random disorder $V$ (cf.~\cite{DS-2}), we expect the eigenvalue to disappear whenever $\lambda>0$ in view of Fermi's Golden Rule. As no branch of eigenvalues would then exist, we may wonder about the meaning of the approximate branch $\lambda\mapsto\kappa_{k,\lambda,\mu}^\ell$ that we propose to construct.
As observed in~\cite{DS-2}, $\kappa_{k,\lambda,\mu}^\ell$ actually admits a complex-valued limit as $\mu\downarrow0$,
which is naturally interpreted as an (approximate) branch of complex resonances.
The need to regularize the coefficients in the Rayleigh-Schrödinger series in this setting is precisely related to the fact that the corresponding resonant modes cannot belong to the space $\Ld^2(\Omega)$: the approximate Bloch waves $\psi_{k,\lambda,\mu}^\ell$ are viewed as approximate resonant modes and do not admit a limit in $\Ld^2(\Omega)$ as $\mu\downarrow0$.
This explains the limitation in the random setting:
since our general approach requires to stick to the  $\Ld^2(\Omega)$-topology, we are limited to timescales such that resonances are not visible, i.e., to timescales $t=o(\lambda^{-2})$. More precisely, we can show that the Schrödinger flow remains close to the free flow for $t=o(\lambda^{-2}|\!\log\lambda|^{-1})$ whenever the potential $V$ has fast decaying correlations, which is however not new and is essentially already contained in~\cite{Spohn-77}.
We refer to~\cite{DS-2} for a discussion of the Floquet-Bloch approach and resonance analysis beyond that timescale.
\end{enumerate}

}

\section{Approximate stationary Floquet-Bloch theory}\label{chap:blochballi}

In this section, we adapt the standard periodic Floquet-Bloch theory (e.g.~\cite{Kuchment-16}) to the general stationary setting, we show how the behavior of the Schrödinger flow is reduced to a fibered perturbation problem, and we approximately solve this perturbation problem in terms of suitable approximate Bloch waves. This covers Steps~(S1)--(S3) of the general approach of Section~\ref{sec:strat}.

\subsection{Stationary Floquet-Bloch theory}

We start by adapting the standard periodic Bloch-Floquet theory (e.g.~\cite{Kuchment-16}) to the general stationary setting.

\subsubsection{Preliminary:  notion of stationarity}\label{chap:stat-BF}

Assume that the probability space $(\Omega,\p)$ is endowed with a measurable action $\tau:=(\tau_x)_{x\in\R^d}$ of the group $(\R^d,+)$ on $\Omega$, that is, the maps $\tau_x:\Omega\to\Omega$ are measurable for all $x$ and they satisfy
\begin{enumerate}[\quad$\bullet$]
\item $\tau_x\circ\tau_y=\tau_{x+y}$ for all $x,y\in\R^d$;
\item $\pr{\tau_xA}=\pr{A}$ for all $x\in\R^d$ and all measurable $A\subset\Omega$;
\item the map $\R^d\times\Omega\to\Omega:(x,\omega)\mapsto\tau_x\omega$ is jointly measurable.
\end{enumerate}
We then assume that the potential $V$ is given by $V(x,\omega):=\tilde V(\tau_{-x}\omega)$ for some random variable $\tilde V:\Omega\to\R$.
More generally, a measurable function $f:\R^d\times\Omega\to\R$ is said to be \emph{$\tau$-stationary} (or simply \emph{stationary}) if it satisfies $f(x,\omega)=f(0,\tau_{-x}\omega)$ for all $x,\omega$. In particular, this implies $\expec{f(x,\cdot)}=\expec{f(0,\cdot)}$ for all $x$, and it ensures that $f$ is jointly measurable and that the map $\omega\mapsto f(x,\omega)$ is measurable for all $x$.
Setting $\tilde f(\omega):=f(0,\omega)$, stationarity obviously yields a bijection between random variables $\tilde f:\Omega\to\R$ and stationary measurable functions $f:\R^d\times\Omega\to\R$. The function $f$ is then called the \emph{stationary extension} of the random variable $\tilde f$.
In particular, the subspace of stationary functions $f:\R^d\times\Omega\to\R$ in $\Ld^2(\Omega,\Ld^2_\loc(\R^d))$ is identified with the Hilbert space $\Ld^2(\Omega)$, and the weak gradient $\nabla$ on locally square integrable functions then turns by stationarity into a linear operator on $\Ld^2(\Omega)$. For all $l\ge0$, we may further define the (Hilbert) space $H^l(\Omega)$ as the space of all random variables $\tilde f\in\Ld^2(\Omega)$ the stationary extension $f$ of which belongs to $\Ld^2(\Omega;H^l_\loc(\R^d))$.
Also note that, by a stochastic version of Lusin's theorem, the joint measurability condition above implies that $\tau$-stationary functions are necessarily stochastically continuous; in particular, for all $\tilde f\in\Ld^2(\Omega)$, the map $\R^d\to\Ld^2(\Omega):y\mapsto \tilde f(\tau_{-y}\cdot)$ is continuous.
We refer to e.g.~\cite[Appendix~A.2]{D-Gloria-14} for details.

\begin{exs}\mbox{}\label{ex:stat}
It is well-known that periodic and quasiperiodic (as well as almost periodic~\cite{PapaVara}) potentials $V$ can be viewed as random stationary potentials:
\begin{enumerate}[$\bullet$]
\item For periodic $V$, we set $\Omega:=\T^d$ endowed with the Lebesgue measure, we define $\tau_{-x}\omega=\omega+x\mod\T^d$, and we set $W(x,\omega):=V(\omega+x)$, which defines a stationary field $W$.
The stationary gradient on $\Ld^2(\Omega)$ then coincides with the usual weak gradient on $\Ld^2(\T^d)$
\item For quasiperiodic $V$ as in~(QP), we set $\Omega:=\T^M$ endowed with the Lebesgue measure, we define $\tau_{-x}\omega=\omega+F^Tx\mod\T^M$, and we set $W(x,\omega):=\tilde V(\omega+F^Tx)$, which defines a stationary field $W$. The stationary gradient on $\Ld^2(\Omega)$ then coincides with $F\nabla_{\T^M}$ in terms of the weak gradient $\nabla_{\T^M}$ on $\Ld^2(\T^M)$.
\end{enumerate}
In the sequel, we consider the Schrödinger flow with $V$ replaced by $W(\cdot,\omega)$ and we exploit averaging wrt the translation $\omega$. Since the main results
of this paper are stated for $\omega=0$, we need to get rid of the averaging wrt $\omega$, which we do  using Sobolev embeddings.
For simplicity, we make no difference between $V$ and~$W$ in the notation.
\qedhere
\end{exs}

\subsubsection{Stationary Floquet transform}
For $f\in \Ld^2(\R^d\times\Omega)$, we first define the (non-stationary) Floquet transform $\mathcal{U}f:\R^d\times\R^d\times\Omega\to\R$ by
\begin{equation}
\mathcal{U}f(k,x,\omega)=\F \left[ \mathcal{O}_x f(\cdot,\omega)\right] (k),\qquad \mathcal{O}_x f(y,\omega)=f(x+y,\tau_y\omega).
\end{equation}
The following properties directly follow from this definition.
\begin{lem}
Writing $e_k(x):=e^{ik\cdot x}$,
\begin{enumerate}[(i)]
\item the map $\mathcal{O}_x$ (hence also the map $f\mapsto\mathcal Uf(\cdot,x,\cdot)$) is unitary on $\Ld^2(\R^d\times \Omega)$ for all $x$;
\item $\mathcal{U}f(k,\cdot,\cdot)$ is $e_k$-stationary in the sense that $\mathcal{U}f(k,x+z,\omega)=e_k(z)\,\mathcal{U}f(k,x,\tau_{-z}\omega)$;
\item $f(x,\omega)=\F^{-1} \left[ \mathcal{U}f(\cdot,x,\omega) \right](0)$, where the RHS is well-defined in $\Ld^2(\R^d\times\Omega)$.\qedhere
\end{enumerate}
\end{lem}
(The expression $\F^{-1} \left[ \mathcal{U}f(\cdot,x,\omega) \right](0)$ is indeed well-defined in $\Ld^2(\R^d\times\Omega)$ since the stochastic continuity that follows from the measurability of the action~$\tau$ ensures that the map $\R^d\to\Ld^2(\R^d\times\Omega):y\mapsto ((x,\omega)\mapsto\F^{-1}[\mathcal{U}f(\cdot,x,\omega)](y)=f(x+y,\tau_y\omega))$ is continuous, cf.~Section~\ref{chap:stat-BF}.)

For $f\in \Ld^2(\R^d\times\Omega)$, it is then natural to define
\begin{equation}\label{e:def-calV}
\Vc f(k,x,\omega):=e^{-ik\cdot x}\,\mathcal{U}f(k,x,\omega),
\end{equation}
which, for any fixed $k\in\R^d$, is stationary by the above properties. Also, for all $x\in\R^d$, the map $f\mapsto\Vc f(\cdot,x,\cdot)$ is unitary on $\Ld^2(\R^d\times\Omega)$. With the usual identification of $\Vc f$ with its restriction $\Vc f(\cdot,0,\cdot)$, we may thus view $\Vc$ as a unitary operator on $\Ld^2(\R^d\times\Omega)$, which we refer to as the {\it stationary Floquet transform}.
\begin{lem}\label{lem:floquet}
The stationary Floquet transform $\mathcal V$ satisfies
\begin{enumerate}[(i)]
\item $f(x,\omega)=\F^{-1} \left[k\mapsto e_{k}(x) \mathcal{V}f(k,x,\omega) \right](0)$, where the RHS is defined in $\Ld^2(\R^d\times\Omega)$;
\item denoting by $\iota:\Ld^2(\R^d)\hookrightarrow\Ld^2(\R^d\times\Omega)$ the canonical injection, we have $\mathcal V\circ\iota=\iota\circ\F$ on $L^2(\R^d)$;
\item for all $f\in\Ld^2(\R^d\times\Omega)$ and $g\in\Ld^2(\Omega)$ with $gf\in \Ld^2(\R^d\times\Omega)$, we have $\mathcal V(gf)=g\mathcal V f$.\qedhere
\end{enumerate}
\end{lem}

\subsubsection{Stationary Floquet-Bloch fibration}

The stationary Floquet transform $\mathcal V$ decomposes differential operators on $\Ld^2(\R^d\times\Omega)$ into direct integrals of elementary fibered operators on the simpler space~$\Ld^2(\Omega)$ of stationary functions.

On the one hand, the Laplacian $-\triangle$ on $\Ld^2(\R^d\times\Omega)$ is transformed as follows by the stationary Floquet transform $\mathcal V$, for all $f\in\Ld^2(\Omega;H^2(\R^d))$,
\begin{align}\label{eq:lap-k}
\mathcal V[(-\triangle)f](k,x,\omega)=-(\nabla+ik)\cdot(\nabla+ik)\mathcal Vf(k,x,\omega)=(|k|^2-\triangle_k)\mathcal Vf(k,x,\omega)
\end{align}
in terms of the (centered) fibered Laplacian
\[-\triangle_k:=e^{-ik\cdot x}(-\triangle)e^{ik\cdot x}-|k|^2=-(\nabla+ik)\cdot(\nabla+ik)-|k|^2=-\triangle-2ik\cdot\nabla,\]
where all the derivatives are taken in the weak sense wrt the $x$-variable. 
The action of the operator $-\triangle_k$ is considered in~\eqref{eq:lap-k} on stationary functions, hence equivalently on~$\Ld^2(\Omega)$. Its domain is then clearly $D(-\triangle_k)=H^2(\Omega)$, and the centering ensures that the constant function $1$ belongs to its kernel.

On the other hand, since the potential $V\in\Ld^2(\Omega)$ is stationary, it (densely) defines a multiplicative operator on $\Ld^2(\R^d\times\Omega)$.
In the periodic and quasiperiodic settings as considered in the present article, $V$ turns out to be uniformly bounded, hence it defines a bounded self-adjoint operator on $\Ld^2(\R^d\times\Omega)$.
In view of the random case,
we emphasize that the boundedness of $V$ is not needed for our purposes:
if $V$ satisfies a lower bound $V(x,\omega)\ge -K(\omega)(1+|x|^2)$ for some random variable $K$ with $\expec{|K|^2}<\infty$ (which is a mild requirement), the Faris-Lavine argument~\cite{Faris-Lavine-74} ensures that the corresponding Schrödinger operator $\Lc_{\lambda}=-\triangle+\lambda V$ on $\Ld^2(\R^d\times\Omega)$ is essentially self-adjoint on $\Ld^\infty(\Omega;H^2(\R^d,|x|^2dx))$.
As in~\eqref{eq:lap-k}, we then find for all $f\in D(\Lc_\lambda)$, using Lemma~\ref{lem:floquet}(iii),
\begin{align}\label{e.dir-decomp}
\mathcal V[\Lc_\lambda f](k,\omega)=(|k|^2+\Lc_{k,\lambda})\mathcal Vf(k,\omega),
\end{align}
in terms of the fibered Schrödinger operator
\[\Lc_{k,\lambda}:=e^{-ik\cdot x}(-\triangle+\lambda V)e^{ik\cdot x}-|k|^2=-\triangle_k+\lambda V.\]
For fixed $k$, we view the fibered operator $\Lc_{k,\lambda}$ as an essentially self-adjoint operator on~$\Ld^2(\Omega)$.
Using direct integral representation  (see e.g.~\cite[p.280]{Reed-Simon-78}), we may reformulate Lemma~\ref{lem:floquet}(i) as
\begin{gather*}
\Ld^2(\R^d\times\Omega)= \int_{\oplus} e_k \Ld^2(\Omega)\,d^*k,\nonumber\\
-\triangle= \int_{\oplus} e_k\,(|k|^2-\triangle_k)\,d^*k,
\qquad \Lc_\lambda= \int_{\oplus} e_k\,(|k|^2+\Lc_{k,\lambda})\,d^*k.
\end{gather*}

This fibration leads to the following useful decomposition of the Schrödinger flow~\eqref{eq:schr-V}:
for an initial condition $u^\circ\in\Ld^2(\R^d)$, denoting as before by $\iota:\Ld^2(\R^d)\hookrightarrow\Ld^2(\R^d\times\Omega)$ the canonical injection and using Lemma~\ref{lem:floquet}(i)--(ii) and~\eqref{e.dir-decomp}, we find
\begin{eqnarray*}
\lefteqn{u_\lambda^t(x,\omega)=\big(e^{-it\Lc_\lambda}\iota u^\circ\big)(x,\omega)
\,\stackrel{\text{(i)}}{=}\,\F^{-1} \Big[ k\mapsto e^{ik\cdot x}\,\Vc\big(e^{-it\Lc_\lambda}\iota u^\circ\big)(k,x,\omega) \Big](0)}
\\
&\hspace{2cm}\stackrel{\eqref{e.dir-decomp}}{=}&\F^{-1} \Big[ k\mapsto e^{ik\cdot x}e^{-it|k|^2}\big(e^{-it\Lc_{k,\lambda}} \mathcal{V}\iota u^\circ\big)(k,x,\omega) \Big](0)\\
&\hspace{2cm}\stackrel{\text{(ii)}}{=}&\F^{-1} \Big[ k\mapsto \hat u^\circ(k)\,e^{ik\cdot x}e^{-it|k|^2}\big(e^{-it\Lc_{k,\lambda}}1\big)(k,x,\omega) \Big](0)\\
&\hspace{2cm}=&\F^{-1}\Big[k\mapsto  \hat u^\circ(k)\,e^{ik\cdot x} e^{-it|k|^2} \int_{\R}e^{-it\kappa} d\mu_{k,\lambda}^1(\kappa)(x,\omega) \Big](0),
\end{eqnarray*}
in terms of the $\Ld^2(\Omega)$-valued spectral measure $\mu_{k,\lambda}^1$ of $\Lc_{k,\lambda}$ associated with the constant function $1$.
Provided we have enough integrability wrt the $k$-variable, this takes the simpler form
\begin{eqnarray}\label{eq:bloch-decomp}
u_\lambda^t(x,\omega)&=&\int_{\R^d} \hat u^\circ(k)\,e^{-it|k|^2} e^{ik\cdot x}\, \big(e^{-it\Lc_{k,\lambda}}\big)(x,\omega)\,d^*k\nonumber\\
&=&\int_{\R^d}\int_{\R} \hat u^\circ(k)\,e^{-it(|k|^2+\kappa)} e^{ik\cdot x}\, d\mu_{k,\lambda}^1(\kappa)(x,\omega)\,d^*k.
\end{eqnarray}
For $\lambda=0$, we simply have $d\mu_{k,0}^1(\kappa)=d\delta_0(\kappa)$, while for $\lambda>0$ the planar wave $e_k$ is corrected into a (potentially non-atomic) \emph{Bloch measure} $e_k\,d\mu_{k,\lambda}^1(\kappa)$, which is adapted to the potential $V$.
If $\mu_{k,\lambda}^1$ admits an atom at $\kappa_*$, the function $e_k\,\mu_{k,\lambda}^1(\{\kappa_*\})\in\Ld^2(\Omega;\Ld^2_\loc(\R^d))$ is called a \emph{Bloch wave}, which is in particular a ``generalized eigenfunction'' of $\Lc_\lambda$ associated with the ``generalized eigenvalue'' $|k|^2+\kappa_*$.
In the periodic case, the measures $\mu_{k,\lambda}^1$ are all discrete and the situation is thus much simplified~\cite{Kuchment-16}.

\subsubsection{Fibered perturbation problem}\label{chap:blochth-limited}

We have seen that the Schrödinger operator $\Lc_\lambda$ on $\Ld^2(\R^d\times\Omega)$ is equivalent to the collection of all the fibered operators $\Lc_{k,\lambda}$ on $\Ld^2(\Omega)$, for $k\in\R^d$.
More precisely, the decomposition~\eqref{eq:bloch-decomp} implies that the behavior of the Schrödinger flow $u_\lambda$ is equivalent to that of all the spectral measures $\mu_{k,\lambda}^1$ associated with the constant function $1$, for $k\in\R^d$.
Since $1$ is an eigenfunction of $\Lc_{k,0}$ associated with the eigenvalue $0$, we are reduced to study the perturbation problem for this eigenvalue in the spectrum of $\Lc_{k,\lambda}$ in the regime $\lambda\ll1$.
A naïve approach consists in postulating that for $\lambda>0$ the eigenvalue $0$ (resp.\@ the eigenfunction $1$) is perturbed into an eigenvalue $\kappa_{k,\lambda}$ (resp.\@ an eigenfunction $\psi_{k,\lambda}$), and in trying to construct them via their Taylor series, that is, as the sum of the so-called Rayleigh-Schrödinger perturbation series
\begin{align}\label{eq:series-ansatz}
\kappa_{k,\lambda}=\lambda\sum_{n=0}^\infty \lambda^n\nu_{k}^n,\qquad\psi_{k,\lambda}=1+\sum_{n=1}^\infty\lambda^n\phi_k^n.
\end{align}
This can indeed be done in the periodic setting (for most $k$).
The eigenvalue equation
\begin{align}\label{eq:bloch}
\Lc_{k,\lambda}\psi_{k,\lambda}=\kappa_{k,\lambda}\psi_{k,\lambda},\qquad\kappa_{k,\lambda}\in\R,\qquad\psi_{k,\lambda}\in\Ld^2(\Omega),
\end{align}
then splits into a hierarchy of Rayleigh-Schrödinger equations for the coefficients $\nu_{k}^n\in\R$ and $\phi_k^n\in\Ld^2(\Omega)$.
In line with the wording in~\cite{BG-16} related to the homogenization theory, we refer to the $\phi_k^n$'s as the correctors.
This approach however quickly fails: for quasiperiodic $V$ the coefficients $(\nu_{k}^n,\phi_k^n)$ can be constructed from the Rayleigh-Schrödinger equations but the series~\eqref{eq:series-ansatz} is not summable, while for random $V$ the correctors $\phi_k^n$ cannot even be defined in $\Ld^2(\Omega)$.
This is related to the spectral discussion in Section~\ref{sec:spectral}: for quasiperiodic $V$ dense crossings of eigenvalues are expected to destroy analyticity of the branch $\lambda\mapsto(\kappa_{k,\lambda},\psi_{k,\lambda})$, while for random $V$ no such branch should even exist.

\subsection{Approximate Bloch waves}

While solving the eigenvalue problem~\eqref{eq:bloch} beyond the periodic setting is very difficult or impossible, we may at least construct approximate solutions of~\eqref{eq:bloch}, that is, approximate Bloch waves, in the small disorder regime $\lambda\ll1$.
In the quasiperiodic setting, since the Rayleigh-Schrödinger coefficients $(\nu_k^n,\phi_k^n)$ can all be constructed, we view~\eqref{eq:series-ansatz} as an asymptotic series describing a likely existing branch, and we define approximate Bloch waves as the partial sums of this formal series,
\[\kappa_{k,\lambda}^\ell:=\,\lambda\sum_{n=0}^\ell\lambda^n\nu_{k}^n,\qquad\,\psi_{k,\lambda}^\ell\,:=\,\sum_{n=0}^\ell\lambda^n\phi_{k}^n.\]
Such truncated Rayleigh-Schrödinger series are referred to in the sequel as {\it Taylor-Bloch waves}.
In the random case, as the correctors $\phi_k^n$ are not defined in $\Ld^2(\Omega)$, we would further need to regularize the Rayleigh-Schrödinger equations to ensure the existence of solutions in the desired space.
We focus here on the quasiperiodic setting, for which all the Rayleigh-Schrödinger coefficients $(\nu_k^n,\phi_k^n)$ can be constructed and define the jet of a formal branch $\lambda\mapsto(\kappa_{k,\lambda},\psi_{k,\lambda})$ at $\lambda=0$.
The admissible set $O$ below will soon be taken as the non-resonant set $\Oc$ defined in~\eqref{def:non-resonant}.
\begin{defin}\label{def:taylor-waves}
Given $1\le\ell<\infty$ and a nonempty open set $O\subset\R^d$, a family $(\nu_k^n,\phi_k^n:k\in O,\,0\le n\le\ell)\subset\Ld^2(\Omega)\times\R$ is called a \emph{field of $\ell$-jets of Bloch waves} if
\begin{enumerate}[(i)]
\item for all $n$, the map $O\to\Ld^2(\Omega)\times\R:k\mapsto(\nu_k^n,\phi_k^n)$ is continuous;
\item $\nu_k^n:=\expec{ V\phi_k^{n}}$ for all $n\ge0$;
\item for all $k\in O$, we have $\phi_k^0\equiv 1$, and for all $n$ the function $\phi_k^{n+1}$ satisfies $\expec{\phi_k^{n+1}}=0$ and
\begin{align}\label{eq:jetbloch}
-\triangle_{k}\phi_{k}^{n+1}=-\Pi V\phi_{k}^{n}+\sum_{l=0}^{n-1}\expecm{V\phi_{k}^{l}}\phi_k^{n-l},
\end{align}
where $\Pi$ denotes the orthogonal projection onto $\{1\}^\bot$, that is $\Pi f:=f-\expec{f}$.
\end{enumerate}
The corresponding family $(\kappa_{k,\lambda}^\ell,\psi_{k,\lambda}^\ell:k\in O,\,\lambda\ge0)$ of partial sums,
\begin{equation}\label{e.def:TBW}
\kappa_{k,\lambda}^{\ell}:=\lambda\,\expecm{V\psi_{k,\lambda}^{\ell}}=\lambda\sum_{n=0}^\ell\lambda^n\nu_k^n,\qquad \psi_{k,\lambda}^{\ell}:=\sum_{n=0}^{\ell}\lambda^{n}\phi_{k}^{n},
\end{equation}
is then called the \emph{sheet of Taylor-Bloch waves} of order $\ell$. Note that $\nu_k^0=\expec{V}$.
In the statement of Theorem~\ref{th:main-quasi}, we further use the short-hand notation $\tilde \kappa_{k,\lambda}^\ell$ and $\tilde\nu_k$ for the extension of $\kappa_{k,\lambda}^\ell$ and $\nu_k$ by zero outside $O$.
\end{defin}
As the following shows, for small $\lambda$, these Taylor-Bloch waves almost satisfy the eigenvalue equation~\eqref{eq:bloch}.
\begin{lem}\label{lem:eqnssumcorr}
Let $\ell\ge1$, let $(\nu_k^n,\phi_k^n)_{k,n}$ be a field of $\ell$-jets of Bloch waves, and let $(\kappa_{k,\lambda}^\ell,\psi_{k,\lambda}^\ell)_{k,\lambda}$ be the corresponding sheet of Taylor-Bloch waves. Then we have
\[(-\triangle_k+\lambda V)\psi_{k,\lambda}^\ell=\kappa_{k,\lambda}^\ell\psi_{k,\lambda}^\ell+\lambda^{\ell+1}\df^\ell_{k,\lambda},\]
in terms of the Taylor-Bloch eigendefect
\[\mathfrak{d}_{k,\lambda}^\ell:=\Big(\Pi V\phi_k^\ell-\sum_{l=0}^{\ell-1}\nu_k^l\phi_k^{\ell-l}\Big)-\lambda\sum_{n=1}^\ell\sum_{l=\ell-n}^{\ell-1}\lambda^{n+l-\ell}\nu_k^{l+1}\phi_k^n.\qedhere\]
\end{lem}
\begin{proof}
The proof is elementary and follows from several resummations:
\begin{eqnarray*}
\lefteqn{\Lc_{k,\lambda}\psi_{k,\lambda}^\ell
\,=\,\Lc_{k,\lambda}\sum_{n=0}^\ell\lambda^n\phi_k^n}
\\
&=&-\sum_{n=1}^\ell\lambda^n(V\phi_k^{n-1}-\expec{ V\phi_k^{n-1}})+\sum_{n=2}^\ell\lambda^n\sum_{l=0}^{n-2}\phi_k^{n-l-1}\expecm{V\phi_k^l}+V\sum_{n=0}^\ell\lambda^{n+1}\phi_k^n\\
&=&\lambda\sum_{n=0}^{\ell-1}\lambda^n\sum_{l=0}^{n}\nu_k^{l}\phi_k^{n-l}+\lambda^{\ell+1}V\phi_k^\ell\\
&=&\kappa_{k,\lambda}^{\ell}\psi_{k,\lambda}^\ell+\lambda^{\ell+1}\Big(V\phi_k^\ell-\expecm{ V\phi_k^\ell}-\sum_{l=0}^{\ell-1}\nu_k^l\phi_k^{\ell-l}\Big)-\lambda^{\ell+2}\sum_{n=1}^\ell\sum_{l=\ell-n}^{\ell-1}\lambda^{n+l-\ell}\nu_k^{l+1}\phi_k^n,
\end{eqnarray*}
as claimed.
\end{proof}

\subsection{Tree formulas for Rayleigh-Schr\"odinger coefficients}\label{app:sol-nonlin-rec}

For later purposes, we now explicitly solve the nonlinear recurrence equations~\eqref{eq:jetbloch} for the Rayleigh-Schrödinger coefficients $(\nu_k^n,\phi_k^n)$.
Indeed, while solving~\eqref{eq:jetbloch} would naïvely lead to a sum of $C^{n^2}$ terms, the formula below only involves~$C^n$ terms (cf.~exponential number of trees).
This is crucial to prove sharp corrector estimates in Proposition~\ref{prop:cor-QP}, which are key to the stretched exponential timescale of the main results.
Although we believe that the formulas below could be obtained by a careful counting and recombination of the many terms occurring when solving the nonlinear recurrence~\eqref{eq:jetbloch}, we rather display a shorter indirect argument based on a Lagrange series expansion.

\begin{prop}\label{prop:sol-nonlin-rec}
Let $V$ be stationary on $(\Omega,\mathbb P)$ and assume that $V\in\Ld^p(\Omega)$ for all $p<\infty$.
For $m\ge1$ let $\Tc_m\subset\N^m$ denote the set of rooted $m$-trees, which we define as the following set of indices (cf.~Figure~\ref{fig:tree}),
\begin{align}\label{eq:rooted-trees}
\Tc_m:=\big\{a=(a_1,\ldots,a_m)\in\N^m\,:\,a_j+\ldots+a_m\le m-j,\,\forall 1\le j\le m\big\}.
\end{align}
Note that $\sharp\Tc_m\le4^m$.
 If $(\nu_k^n,\phi_k^n)_{k,n}$ is a field of jets of Bloch waves as in Definition~\ref{def:taylor-waves},
 then we have for all $n\ge0$,
\begin{multline*}
\nu_{k}^n=\sum_{m=1}^{n+1}(-1)^{n+1-m}\sum_{a\in\Tc_m}\sum_{c\in\N^m\atop|c|=n+1-m}\sum_{b^1\in\N^{c_1}\atop|b^1|=a_1}~\ldots~\sum_{b^m\in\N^{c_m}\atop|b^m|=a_m}\\
\times\,\expecm{V(-\triangle_k)^{-b_1^1-1}\Pi V\ldots(-\triangle_k)^{-b_{c_1}^1-1}\Pi V}\\
\times\ldots\,\expecm{V(-\triangle_k)^{-b_1^m-1}\Pi V\ldots(-\triangle_k)^{-b_{c_m}^m-1}\Pi V},
\end{multline*}
and for all $n\ge1$,
\[\phi_{k}^n=\sum_{m=1}^n(-1)^{m}\sum_{\ell=0}^{n-m}\sum_{a\in\N^\ell\atop|a|=n-m-\ell}\sum_{b\in\N^m\atop|b|=\ell}\nu_{k}^{a_1}\ldots\nu_{k}^{a_\ell}\,(-\triangle_k)^{-b_1-1}\Pi V\ldots (-\triangle_k)^{-b_m-1}\Pi V,\]
assuming that all the terms make sense in $\Ld^2(\Omega)$ (in fact, whenever $k$ belongs to the non-resonant set $\Oc$, cf.~\eqref{def:non-resonant}).
\end{prop}

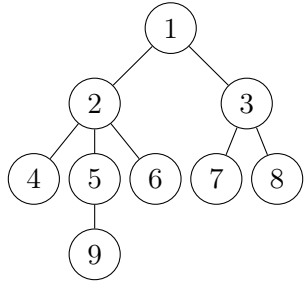
\begin{figure}
{\begin{center}
\begin{tikzpicture}[scale=1,level distance=1cm,
level 1/.style={sibling distance=2cm},
level 2/.style={sibling distance=0.8cm},
level 3/.style={sibling distance=1cm}]
\node[circle,draw]{1}
   child{node[circle,draw]{2}
      child{node[circle,draw]{4}}
      child{node[circle,draw]{5}
         child{node[circle,draw]{9}}}
      child{node[circle,draw]{6}}}
   child{node[circle,draw]{3}
      child{node[circle,draw]{7}}
      child{node[circle,draw]{8}}};
\end{tikzpicture}
\caption{\label{fig:tree}
Given a rooted plane tree, we let the vertices be labelled from the left to the right, from the top to the bottom. A tree of order $m$ then uniquely defines an element $a\in\Tc_m$ by defining $a_j$ as the number of children of the $j$th vertex. For instance, the above plotted tree corresponds to the element $a=(2,3,2,0,1,0,0,0,0)\in\Tc_{9}$.}
\end{center}}
\end{figure}

\begin{proof}
The cardinality $\sharp\Tc_m$ is obviously bounded by the number of ways to put $m-1$ unlabelled balls in $m$ labelled boxes, that is, by $\binom{2(m-1)}{m-1}\le 4^m$.
By a density argument, we may assume that the potential $V$ belongs to $\Ld^\infty(\Omega)$. By an approximation argument, it further suffices to establish the result for {\it regularized} Rayleigh-Schrödinger coefficients: for $\mu>0$ we define $(\nu_{k,\mu}^n,\phi_{k,\mu}^n)$ as follows,
\begin{enumerate}[(i)]
\item $\nu_{k,\mu}^n:=\expecm{ V\phi_{k,\mu}^{n}}$ for all $n\ge0$;
\item  $\phi_{k,\mu}^0\equiv 1$, and for all $n\ge0$ the function $\phi_{k,\mu}^{n+1}$ satisfies $\expecm{\phi_{k,\mu}^{n+1}}=0$ and
\begin{align}\label{eq:jetbloch-reg}
(-i\mu-\triangle_{k})\phi_{k,\mu}^{n+1}=-\Pi V\phi_{k,\mu}^{n}+\sum_{l=0}^{n-1}\expecm{V\phi_{k,\mu}^{l}}\phi_{k,\mu}^{n-l}.
\end{align}
\end{enumerate}
We then define the regularized Bloch waves
\begin{equation}\label{e.def:regTBW} 
\kappa_{k,\lambda,\mu}:=\lambda\expecm{ V\psi_{k,\lambda,\mu}}=\lambda\sum_{n=0}^\infty \lambda^n\nu_{k,\mu}^n,\qquad \psi_{k,\lambda,\mu}:=\sum_{n=0}^{\infty}\lambda^{n}\phi_{k,\mu}^{n},
\end{equation}
where the series indeed converge for $\lambda$ small enough (for fixed $\mu>0$)
and satisfy
\begin{align}\label{eq:reg-eigen}
(-i\mu-\triangle_k+\lambda V)\psi_{k,\lambda,\mu}=\kappa_{k,\lambda,\mu}\psi_{k,\lambda,\mu}-i\mu,\qquad \expec{\psi_{k,\lambda,\mu}}=1.
\end{align}
Note that the Rayleigh-Schrödinger coefficients $(\nu_k^n,\phi_{k}^n)$ are clearly retrieved in the limit $\mu\downarrow 0$, so that it suffices to establish the statement of the proposition for $(\nu_{k}^n,\phi_{k}^n)$  and $-\triangle_k$ replaced by $(\nu_{k,\mu}^n,\phi_{k,\mu}^n)$ and $-i\mu-\triangle_k$.
The strategy of the proof is as follows: we establish another explicit series expansion for $\psi_{k,\lambda,\mu}$ in the restricted regime $\lambda\ll \mu^2\wedge1$ using a fixed point argument together with Lagrange and Neumann expansions, and then we identify the coefficients with $(\nu_{k,\mu}^n,\phi_{k,\mu}^n)$ in~\eqref{e.def:regTBW}.
Since coefficients are independent of~$\lambda$, the restriction $\lambda\ll \mu^2\wedge1$ in the proof is naturally irrelevant for the validity of these new formulas for $(\nu_{k,\mu}^n,\phi_{k,\mu}^n)$, and the corresponding formulas for $(\nu_{k}^n,\phi_{k}^n)$ follow for $\mu\downarrow0$.

Applying the projectors $1-\Pi=\expec{\cdot}$ and $\Pi$, we deduce that the equation \eqref{eq:reg-eigen} is equivalent to
\begin{align}\label{eq:reg-eigen-alter}
\kappa_{k,\lambda,\mu}=\lambda\,\expecm{V\psi_{k,\lambda,\mu}},\qquad \big(-i\mu-\kappa_{k,\lambda,\mu}-\triangle_k+\lambda \Pi V\big)(\psi_{k,\lambda,\mu}-1)=-\lambda \Pi V.
\end{align}
Provided that $\Im(i\mu+\kappa_{k,\lambda,\mu})\ne0$, the $\mu$-regularized eigenvalue $\kappa_{k,\lambda,\mu}$ is therefore a solution of the following fixed-point problem,
\begin{gather*}
\kappa_{k,\lambda,\mu}=G_{k,\lambda,\mu}(\kappa_{k,\lambda,\mu}),\\
G_{k,\lambda,\mu}(\kappa):=\lambda\,\expec{V}-\lambda^2\,\expecm{V(-i\mu-\kappa-\triangle_k+\lambda \Pi V)^{-1}\Pi V}.
\end{gather*}
Provided that $\Im(i\mu+\kappa)\ne0$, the quantity $G_{k,\lambda,\mu}(\kappa)$ can be rewritten in form of a Neumann series for all $\lambda>0$ small enough,
\begin{eqnarray*}
G_{k,\lambda,\mu}(\kappa)&=&\lambda\,\expecm{V}-\lambda^2\,\expec{V\big(1+\lambda (-i\mu-\kappa-\triangle_k)^{-1}\Pi V\big)^{-1}(-i\mu-\kappa-\triangle_k)^{-1}\Pi V}\\
&=&\lambda\,\expecm{V}-\lambda^2\sum_{n=0}^\infty(-\lambda)^n\,\expec{V\big(\Gamma_{k,\mu}(\kappa)\Pi V\big)^{n+1}1},
\end{eqnarray*}
where we use the shorthand notation $\Gamma_{k,\mu}(\kappa):=(-i\mu-\kappa-\triangle_k)^{-1}$.
For $|\kappa|<\mu$, we may further use the Neumann series
\[\Gamma_{k,\mu}(\kappa)=\big(1-\kappa\Gamma_{k,\mu}(0)\big)^{-1}\Gamma_{k,\mu}(0)=\sum_{n=0}^\infty\kappa^n\Gamma_{k,\mu}(0)^{n+1}.\]
Injecting this into the above then leads to the following series expansion for $G_{k,\lambda,\mu}$: given $|\kappa|<\mu$, we obtain for all $\lambda>0$ small enough,
\begin{eqnarray*}
\lefteqn{G_{k,\lambda,\mu}(\kappa)}\\
&=&\lambda\,\expecm{V}-\sum_{n=1}^\infty(-\lambda)^{n+1}\sum_{a_1,\ldots,a_{n}=0}^\infty\kappa^{a_1+\ldots+a_{n}}\,\expec{V\Gamma_{k,\mu}(0)^{a_1+1}\Pi V\ldots\Gamma_{k,\mu}(0)^{a_{n}+1}\Pi V}
\\
&=&\sum_{\ell=0}^\infty\kappa^\ell G_{k,\lambda,\mu;\ell},
\end{eqnarray*}
where we have set for all $\ell\in\N$,
\[G_{k,\lambda,\mu;\ell}:=-\sum_{n=0}^\infty(-\lambda)^{n+1}\sum_{a\in\N^n\atop|a|=\ell}\expec{V\Gamma_{k,\mu}(0)^{a_1+1}\Pi V\ldots\Gamma_{k,\mu}(0)^{a_n+1}\Pi V}.\]
(Note that we use here the natural convention for empty sums and products: for $\ell=0$ the sum takes the form
\[-\sum_{n=0}^\infty(-\lambda)^{n+1}\,\expec{V(\Gamma_{k,\mu}(0)\Pi V)^{n}1},\]
while for $\ell\ge1$ the sum is reduced to $n\ge1$.)
We are now in position to solve the fixed-point equation $\kappa=G_{k,\lambda,\mu}(\kappa)$ in the form of a Lagrange series expansion as e.g.\@ in~\cite{DSV-97}.
By a simple tree-counting argument, the unique solution $\kappa_{k,\lambda,\mu}$ for $\lambda$ small enough can be expressed as the following sum on all possible trees of all sizes,
\begin{align}\label{eq:formal-sol-lambda0}
\kappa_{k,\lambda,\mu}=\sum_{m=1}^\infty\sum_{a\in\Tc_m}G_{k,\lambda,\mu;a_1}\ldots G_{k,\lambda,\mu;a_m},
\end{align}
or alternatively,
$$
\kappa_{k,\lambda,\mu}=\sum_{n=1}^\infty\lambda^n\kappa_{k,\mu;n},
$$
where we have defined
\begin{multline}\label{eq:formal-sol-lambda}
\kappa_{k,\mu;n}:=\sum_{m=1}^{n}(-1)^{n+m}\sum_{a\in\Tc_m}\sum_{c\in\N^m\atop|c|=n-m}\sum_{b^1\in\N^{c_1}\atop|b^1|=a_1}~\ldots~\sum_{b^m\in\N^{c_m}\atop|b^m|=a_m}\\
\times\,\expecm{V\Gamma_{k,\mu}(0)^{b_1^1+1}\Pi V\ldots\Gamma_{k,\mu}(0)^{b_{c_1}^1+1}\Pi V}\\
\times\ldots\,\expecm{V\Gamma_{k,\mu}(0)^{b_1^m+1}\Pi V\ldots\Gamma_{k,\mu}(0)^{b_{c_m}^m+1}\Pi V},
\end{multline}
provided that the above power series is absolutely convergent and satisfies $|\kappa_{k,\lambda,\mu}|<\mu$.
Let us quickly check that these two conditions are indeed satisfied for $\lambda>0$ small enough.
For that purpose, we make use of the following coarse estimate,
\begin{align*}
\big|\expecm{V\Gamma_{k,\mu}(0)^{b_1+1}\Pi V\ldots\Gamma_{k,\mu}(0)^{b_{m}+1}\Pi V}\big|\,\le\,\|V\|_{\Ld^\infty}^{m+1}\mu^{-(b_1+\ldots+b_m)-m},
\end{align*}
which indeed yields for all $n\ge1$,
\begin{eqnarray*}
|\kappa_{k,\mu;n}|&\le&\sum_{m=1}^{n}\sum_{a\in\Tc_m}\sum_{c_1,\ldots,c_m=0}^\infty\!\!\mathds1_{c_1+\ldots+c_m=n-m}\!\!\\
&&\hspace{2cm}\times\sum_{b^1\in\N^{c_1}\atop|b^1|=a_1}\!\!\ldots\!\!\sum_{b^m\in\N^{c_m}\atop|b^m|=a_m}\!\!\|V\|_{\Ld^\infty}^{m+c_1+\ldots+c_m}\mu^{-(|b^1|+\ldots+|b^m|)-(c_1+\ldots+c_m)}\\
&\le&\|V\|_{\Ld^\infty}^n\sum_{m=1}^{n}\sum_{a\in\Tc_m}\sum_{c_1,\ldots,c_m=0}^\infty\mathds1_{c_1+\ldots+c_m=n-m}\\
&&\hspace{2cm}\times\binom{a_1+c_1-1}{c_1-1}\ldots\binom{a_m+c_m-1}{c_m-1} \mu^{m-n-(a_1+\ldots+a_m)}\\
&\le&\|V\|_{\Ld^\infty}^n\sum_{m=1}^{n}\sum_{j=0}^{m-1}\mu^{m-n-j}\mathcal I_{m,n,j},
\end{eqnarray*}
where we have set
\[\mathcal I_{m,n,j}\,:=\,\sum_{a_1,\ldots,a_m=0}^\infty~\sum_{c_1,\ldots,c_m=0}^\infty\mathds1_{a_1+\ldots+a_m=j\atop c_1+\ldots+c_m=n-m}\binom{a_1+c_1-1}{c_1-1}\ldots\binom{a_m+c_m-1}{c_m-1}.\]
(Here we take the convention $\binom{-1}{-1}=1$ and $\binom{s}{-1}=0$ for $s\ge0$.) Using the coarse bound $\Ic_{m,n,j}\le2^{2(n+j-1)-m}$, we deduce
\[|\kappa_{k,\mu;n}|\,\le\,\|V\|_{\Ld^\infty}^n2^{3n}(\mu\wedge1)^{1-n},\]
which leads to
\begin{equation*}
|\kappa_{k,\lambda,\mu}|\,\le\,\sum_{n=1}^\infty\lambda^n|\kappa_{k,\mu;n}|
\,\le \,\sum_{n=1}^\infty\Big(\frac{8\lambda}{\mu\wedge1}  \|V\|_{\Ld^\infty}\Big)^n
\,< \, \mu
\end{equation*}
for $\lambda$ small enough (say $\lambda \ll \mu^2\wedge1$).

We turn to the series representation of $\psi_{k,\lambda,\mu}$, and insert the expression~\eqref{eq:formal-sol-lambda0} into~\eqref{eq:reg-eigen-alter}.
Since $|\kappa_{k,\lambda,\mu}|<\mu$, one may invert the equation for $\psi_{k,\lambda,\mu}$.
Proceeding as above using Neumann series, we obtain the following absolutely convergent expansion in $\Ld^2(\Omega)$ for $\lambda$ small enough,
\begin{eqnarray*}
\psi_{k,\lambda,\mu}&=&1+\sum_{n=1}^\infty(-\lambda)^{n}\sum_{\ell=0}^\infty\kappa_{k,\lambda,\mu}^{\ell}\sum_{a\in\N^n\atop|a|=\ell}(-i\mu-\triangle_k)^{-a_1-1}\Pi V\ldots (-i\mu-\triangle_k)^{-a_n-1}\Pi V
\\
&=&1+\sum_{n=1}^\infty\lambda^n\psi_{k,\mu;n},
\end{eqnarray*}
where we have set for all $n\ge1$
\begin{multline}\label{eq:formal-sol-psi}
\psi_{k,\mu;n}:=\sum_{m=1}^n(-1)^{m}\sum_{\ell=0}^{n-m}\sum_{a\in\N^\ell\atop|a|=n-m-\ell}\sum_{b\in\N^m\atop|b|=\ell}\\
\times\,\kappa_{k,\mu;1+a_1}\ldots\kappa_{k,\mu;1+a_\ell}\,(-i\mu-\triangle_k)^{-b_1-1}\Pi V\ldots (-i\mu-\triangle_k)^{-b_m-1}\Pi V.
\end{multline}

We are in position to conclude.
Comparing the above series representations with the Rayleigh-Schrödinger series~\eqref{e.def:regTBW} satisfying the same regularized eigenvalue equation~\eqref{eq:reg-eigen}, and identifying the powers of $\lambda$ in these (locally convergent) series, we deduce for all $n\ge1$,
\[\psi_{k,\mu;n}=\phi_{k,\mu}^n,\qquad\kappa_{k,\mu;n}=\nu_{k,\mu}^{n-1}.\]
Passing to the limit $\mu\downarrow0$ in these formulas, the conclusion follows.
\end{proof}


\section{Approximate Bloch waves in the quasiperiodic setting}\label{chap:QP}

In this section, we establish fine bounds on the Rayleigh-Schrödinger coefficients $(\nu_k^n,\phi_k^n)$ in the quasiperiodic setting~(QP).
As emphasized in Remark~\ref{rem:QP-res}, the quality of such bounds depends both on the regularity of the lifted map $\tilde V$ and on the algebraic properties of the winding matrix $F$. We start by introducing a suitable form of a Diophantine condition for~$F$, adapted to the fibered structure.

\subsection{Diophantine condition}\label{sec:Dioph}
The key ingredient to estimate the Rayleigh-Schrödinger coefficients in the quasiperiodic setting is the inversion of the fibered Laplacians~$-\triangle_k$, $k\in\R^d$.
The Fourier symbol of $-\triangle_k$ is given by $\xi\mapsto|F\xi+k|^2-|k|^2$ on $\Z^M$ and can in general vanish for $\xi\ne0$, which prohibits any invertibility.
In order to ensure invertibility, we must obviously assume that $F$ is irrational (that is, $F\xi\ne0$ for all $\xi\in\Z^M\setminus\{0\}$) and we must also restrict to values of $k$ away from the so-called diffraction hyperplanes $P_\xi:=\{k'\in\R^d:|F\xi+k'|=|k'|\}$, $\xi\in\Z^M\setminus\{0\}$.
The complement of the union of all those hyperplanes constitutes the so-called {\it non-resonant set},
\begin{equation}\label{def:non-resonant}
\Oc\,:=\,\R^d\setminus\bigcup_{\xi\in\Z^M\setminus\{0\}}P_\xi,
\end{equation}
which typically has a Cantor-like structure.
To obtain precise bounds on the Rayleigh-Schrödinger coefficients, fine estimates on the inverse symbol $\xi\mapsto(|F\xi+k|^2-|k|^2)^{-1}$ are further required for non-resonant $k$. This is provided by the following when $F$ is Diophantine.
To ensure uniform bounds, we must naturally restrict to values of $k$ outside some {\it fattened} resonant set $\Rc_R$, $R\ge1$, with $\Oc_R:=\R^d\setminus\overline{\Rc_R}\,\uparrow\,\Oc$ as $R\uparrow\infty$.
\begin{lem}[Diophantine condition]\label{lem:dioph}
Assume that the winding matrix $F\in\R^{d\times M}$ satisfies for some $r_0>0$
the  Diophantine condition
\begin{align}\label{eq:classical-Dioph2}
|F\xi|\ge\frac1{C_0}|\xi|^{-r_0}\qquad\text{for all $\xi\in\Z^M\setminus\{0\}$},
\end{align}
and let $s_0>M+r_0$ be fixed.
Then there exists a decreasing collection $(\Rc_R)_{R\ge1}$ of open (resonant) subsets $\Rc_R\subset \R^d$ and there exists a constant $C>0$ (depending on $C_0,F,M,s_0$) such that the following hold for all $R>0$:
\begin{enumerate}[(i)]
\item For all $k\in\R^d\setminus\Rc_R$ and all $\xi\in\Z^M\setminus\{0\}$, we have
\begin{align}\label{eq:dioph}
\big||F\xi+k|^2-|k|^2\big|\ge R^{-1}|\xi|^{-s_0}.
\end{align}
\item For all $\kappa>0$, we have $|\Rc_R\cap B_\kappa|\le CR^{-1}|\partial B_\kappa|$.
\item We can decompose $\Rc_R=\bigcup_{n=1}^\infty\Rc_R^n$, where $(\Rc_R^n)_n$ is an increasing sequence of open subsets of $\R^d$, such that for all $n$,
\begin{itemize}
\item $\Rc_R^n$ is a finite union of regular open sets;
\item the condition~\eqref{eq:dioph} holds for all $k\in\R^d\setminus\Rc_R^n$ and $\xi\in\Z^M\setminus\{0\}$ with $|\xi|\le n$.\qedhere
\end{itemize}
\end{enumerate}
\end{lem}

\noindent
(In the sequel, we conveniently write $\Rc_R^t:=\Rc_R^{\ceil{t}}$ for non-integer $t\ge 0$.)

\begin{rems}\mbox{}\label{rem:per+}
\begin{enumerate}[$\bullet$]
\item \emph{Diophantine condition~\eqref{eq:classical-Dioph2}:} Given $r_0>d-1$, the standard theory of Diophantine conditions ensures that for almost every $F\in\R^{d\times M}$ there exists $C_0>0$ (depending on $F,M,r_0$) such that~\eqref{eq:classical-Dioph2} holds.
\item \emph{Property~(ii):} Two dual behaviors are included in property~(ii): on the one hand the density of the resonant set $\Rc_R$ decreases to $0$ as~$R\uparrow\infty$, and on the other hand for fixed $R$ the set $\R^d\setminus\Rc_R$ is extensive in the sense that the density of $\Rc_R$ in a ball $B_\kappa$ decreases to $0$ as $\kappa\uparrow\infty$.
As exploited in the proof of Corollary~\ref{th:main-quasi-tsp-2}, this duality precisely allows us to argue alternatively in the small disorder regime or at high frequencies.
\item \emph{Periodic setting:} Let us argue that in the periodic setting (that is, (QP) with $M=d$ and $F=\Id$) the above lemma holds with $s_0=0$. More precisely, there exists a decreasing collection $(\Rc_R)_{R\ge1}$ of regular open subsets $\Rc_R\subset\R^d$ and there exists a constant $C>0$ such that the following hold for all $R\ge1$:
\begin{enumerate}[(i')]
\item For all $k\in\R^d\setminus\Rc_R$ and all $\xi\in\Z^M\setminus\{0\}$, we have
\begin{align}\label{eq:dioph-per}
\big||\xi+k|^2-|k|^2\big|\ge R^{-1}.
\end{align}
\item For all $\kappa>0$, we have $|\Rc_R\cap B_\kappa|\le C_\kappa R^{-1}$.\footnote{Similarly as in Lemma~\ref{lem:dioph}, the factor $C_\kappa$ could  be improved into $C|\partial B_\kappa|$ if we replace the RHS $R^{-1}$ in~\eqref{eq:dioph-per} by $R^{-1}|\xi|^{-s_0}$ with $s_0>d-1$.}
\end{enumerate}
This is obtained as a direct adaptation of the proof below, noting that the diffraction hyperplane $P_\xi=\{k'\in\R^d:|F\xi+k'|=|k'|\}$ does not intersect the ball $B_{\frac12|\xi|}$.\qedhere
\end{enumerate}
\end{rems}

\begin{proof}[Proof of Lemma~\ref{lem:dioph}]
For $R\ge1$ and $\xi\in\Z^M\setminus \{0\}$, we consider the fattened diffraction hyperplane
\begin{eqnarray*}
\Rc_R(\xi)\,:=\,\big\{k\in \R^{d}:\big||F\xi+k|^2-|k|^2\big|<R^{-1}|\xi|^{-s_0}\big\},
\end{eqnarray*}
and the fattened resonance set
\begin{align*}
\Rc_R\,:=\,\bigcup_{\xi\in\Z^M\setminus\{0\}}\Rc_R(\xi).
\end{align*}
For $\xi\in\Z^M\setminus\{0\}$, we note that the distance of a point $k\in\R^d$ to the diffraction hyperplane $P_\xi:=\{k'\in\R^d:|F\xi+k'|=|k'|\}$ is given by $\big|k\cdot\frac{F\xi}{|F\xi|}-\frac12|F\xi|\big|$, hence a point $k\in\R^d$ satisfies $||F\xi+k|^2-|k|^2|<R^{-1}|\xi|^{-s_0}$ if and only if it is at distance $<\frac12R^{-1}|F\xi|^{-1}|\xi|^{-s_0}$ from $P_\xi$. This implies
\begin{align}\label{eq:redef-RRxi}
\Rc_R(\xi)=P_\xi+B_{\frac12R^{-1}|F\xi|^{-1}|\xi|^{-s_0}},
\end{align}
and the Diophantine condition~\eqref{eq:classical-Dioph2} then allows to estimate, for all $\kappa>0$,
\[|\Rc_R(\xi)\cap B_\kappa|\,\le\,CR^{-1}\kappa^{d-1}|F\xi|^{-1}|\xi|^{-s_0}\,\le\,CR^{-1}\kappa^{d-1}|\xi|^{r_0-s_0}.\]
By a coarse union bound and the choice $s_0>M+r_0$, this yields
\begin{eqnarray*}
|\Rc_R\cap B_\kappa|\,\le\,\sum_{\xi\in\Z^d\setminus\{0\}}CR^{-1}\kappa^{d-1}|\xi|^{r_0-s_0}\,\le\,CR^{-1}\kappa^{d-1}.
\end{eqnarray*}
Items~(i)--(ii) follow.
It remains to prove item~(iii). For that purpose, for all $n\ge1$, we consider
\[\Rc_R^n:=\bigcup_{\xi\in\Z^M\setminus\{0\}\atop |\xi|\le n}\Rc_R(\xi).\]
By definition, the sequence $(\Rc_R^n)_n$ is increasing and satisfies $\Rc_R=\bigcup_n\Rc_R^n$, and~\eqref{eq:dioph} holds for all $k\in\R^d\setminus\Rc_R^n$ and $|\xi|\le n$.
\end{proof}

\subsection{Control of the Rayleigh-Schrödinger coefficients}
For $k\in\Oc$, the Fourier symbol of $-\triangle_k$ does not vanish outside the origin, which ensures that all the terms of the form $(-\triangle_k)^{-b_1-1}\Pi V\ldots(-\triangle_k)^{-b_m-1}\Pi V$ are well-defined in $\Ld^2(\Omega)$. In view of Proposition~\ref{prop:sol-nonlin-rec}, this entails that the Rayleigh-Schrödinger coefficients $(\nu_k^n,\phi_k^n)$ are well-defined for all $n$ whenever $k\in\Oc$. It remains to establish fine estimates on these coefficients, for which we exploit the Diophantine condition~\eqref{eq:classical-Dioph2}.
\begin{prop}\label{prop:cor-QP}
Consider the quasiperiodic setting~\emph{(QP)},
assume that the winding matrix~$F$ satisfies the Diophantine condition~\eqref{eq:classical-Dioph2} with $r_0>0$, that the lifted map $\tilde V$ has compactly supported Fourier transform, and set $s_0>M+r_0$ and $K:=\sup\{1\vee|\xi|:\xi \in \supp \F\tilde V\}$.
There exists a constant $C$ (depending on $F,M,s_0$) and for all $\ell,R>0$ there exists a field of $\ell$-jets of Bloch waves $(\nu_{k}^n,\phi_{k}^n:k\in\R^d\setminus\Rc_R^{K\ell},\,0\le n\le\ell)$ in the sense of Definition~\ref{def:taylor-waves}, which satisfy
for all $n\ge1$, $k\in\R^d\setminus\Rc_R^{Kn}$, and $s,j\ge0$,
\begin{eqnarray}
|\nabla_k^j\nu_{k}^n|&\le&(CRjK^{s_0+1}n^{s_0+1})^j(CRK^{s_0+M}n^{s_0})^{n}\|\F \tilde V\|_{\Ld^\infty}^{n+1},\label{eq:est-nu-QP}\\
\|\nabla_k^j \phi_k^n\|_{H^s(\Omega)}&\le& (CKn)^{s}(CRjK^{s_0+1}n^{s_0+1})^j(CRK^{s_0+M}n^{s_0})^{n}\|\F \tilde V\|_{\Ld^\infty}^n.\label{eq:est-cor-QP}
\end{eqnarray}
In particular, for all $\hat u\in C^\infty_c(\R^d)$ supported in $\R^d\setminus\Rc_R^{Kn}$, we deduce
\begin{multline}\label{eq:est-cor-add-QP}
\sup_{\omega\in\Omega}\bigg(\int_{\R^d}\Big|\int_{\R^d}e^{ik\cdot x}\,\nabla_k^j\nabla^s\phi_k^n(x,\omega)\,\hat u(k)\,d^*k\Big|^2dx\bigg)^\frac12 
\\
\le\,(CKn)^{s+M+1}(CRjK^{s_0+1}n^{s_0+1})^j(CRK^{s_0+M}n^{s_0})^{n}\|\F \tilde V\|_{\Ld^\infty}^n\,\|\hat u\|_{\Ld^2},
\end{multline}
which holds uniformly wrt tranlations $\omega\in\Omega=\T^M$.
\end{prop}

\begin{samepage}\begin{rems}\mbox{}\label{rem:per++}
\begin{enumerate}[$\bullet$]
\item \emph{Dependence on $n$:} We believe that the $(Cn^{s_0})^{n}$-growth of the above bounds on $(\nu_k^n,\phi_k^n)$ is essentially optimal, which implies in particular that the Rayleigh-Schrödinger series~\eqref{eq:series-ansatz} cannot be absolutely convergent,
cf.~the discussion in Section~\ref{sec:spectral}.
\item \emph{Irrational winding matrix:} If we merely assume $F\xi \ne 0$ for all $\xi \in \Z^M\setminus\{0\}$, then \eqref{eq:est-cor-QP}, \eqref{eq:est-nu-QP}, and~\eqref{eq:est-cor-add-QP} hold in a modified form where the dependence on $K$ and $n$ is replaced by some constant $C_{K,n}$ that is no longer explicit (and can grow much faster, for instance for Liouville frequencies). The easy adaptation of the proof is left to the reader.
\item \emph{Periodic setting:} In the periodic setting (that is, (QP) with $M=d$ and $F=\Id$), the bounds of Proposition~\ref{prop:cor-QP} above hold in the following improved form: for all $R,n\ge1$, $s,j\ge0$ , and $k\in\R^d\setminus\Rc_R$ (with $\Rc_R$ defined as in the last item of Remarks~\ref{rem:per+}),
\[\|\nabla_k^j\phi_k^n\|_{H^s(\Omega)}\,\le\, n^{s}(CRjn)^j(CR)^{n}.\]
This ensures the convergence of the Rayleigh-Schrödinger perturbation series~\eqref{eq:series-ansatz} for small $\lambda$ (which also follows from the Kato-Rellich theorem).\qedhere
\end{enumerate}
\end{rems}\end{samepage}

\begin{proof}[Proof of Proposition~\ref{prop:cor-QP}]
We use Fourier series~$\F$ in the (high-dimensional) torus $\Omega=\T^M$.
We make a slight abuse of notation in this proof and write $\hat V=\F\tilde V:\Z^M\to\C$. We split the proof into two steps.

\medskip
\step1 Proof of~\eqref{eq:est-nu-QP} and~\eqref{eq:est-cor-QP} for $j=0$.

Lemma~\ref{lem:dioph} allows to invert the Fourier symbol of $-\triangle_k$.
By Proposition~\ref{prop:sol-nonlin-rec}, it suffices to control for all $m\ge1$ and $b\in\N^m$ the functions
\[\chi_k^{m,b}:=(-\triangle_k)^{-b_1-1}\Pi \tilde V\ldots(-\triangle_k)^{-b_{m}-1}\Pi\tilde V~:~\T^M\to\C.\]
Since $\F[\Pi f](\xi)=\hat f(\xi)\mathds1_{\xi\ne0}$ for all $f\in\Ld^2(\T^M)$, the Fourier transform of $\chi_k^{m,b}$ takes the form 
\begin{equation}\label{e.FS}
\hat \chi_k^{m,b}(\xi)\,=\,\mathds1_{\xi\ne0}\sum_{\xi_2,\ldots,\xi_m\in\Z^M\setminus\{0\}}\frac{\hat V(\xi-\xi_2)\ldots\hat V(\xi_{m-1}-\xi_m)\hat V(\xi_m)}{(|F\xi+k|^2-|k|^2)^{b_1+1}\ldots(|F\xi_m+k|^2-|k|^2)^{b_m+1}},
\end{equation}
so that Parseval's formula yields
\begin{equation*} 
\expecm{V  \chi_k^{m,b}}\,
=\,\sum_{\xi_1,\ldots,\xi_m\in\Z^M\setminus\{0\}}\frac{\hat V(-\xi_1)\hat V(\xi_m)}{(|F\xi_m+k|^2-|k|^2)^{b_m+1}}\prod_{j=1}^{m-1}\frac{\hat V(\xi_j-\xi_{j+1})}{(|F\xi_j+k|^2-|k|^2)^{b_j+1}}.
\end{equation*}
Recall that $K:=\max\{1\vee|\xi|:\xi\in\supp\hat V\}<\infty$, so that the above sum can be restricted to
\[\big\{(\xi_1,\ldots,\xi_m)\in(\Z^M\setminus\{0\})^m\,:\,|\xi_{\ell}|,|\xi_{m-\ell+1}|\le \ell K,~\forall\, 1\le \ell\le \lceil \tfrac m2\rceil\big\}.\]
Combining this observation with Lemma~\ref{lem:dioph} then yields the following for all $m\ge1$ and $b\in\N^m$: if $m\in2\N$, we have for all $k\in\R^d\setminus\Rc_R^{K\frac m2}$,
\begin{eqnarray*}
\big|\expecm{V \chi_k^{m,b}}\big|
&\le&(RK^{s_0})^{m+|b|}\,\|\hat V\|_{\Ld^\infty}^{m+1}\sum_{\xi_1,\ldots,\xi_m\in\Z^M}\mathds1_{|\xi_1|,|\xi_m|\le K}\\
&&\hspace{2cm}\times\prod_{j=2}^{m/2}\Big(j^{s_0(b_j+b_{m-j+1}+2)}\,\mathds1_{|\xi_j-\xi_{j-1}|\le K}\,\mathds1_{|\xi_{m-j+1}-\xi_{m-j+2}|\le K}\Big)\\
&\le&(CK)^{mM}(RK^{s_0}m^{s_0})^{m+|b|}\,\|\hat V\|_{\Ld^\infty}^{m+1}.
\end{eqnarray*}
Likewise, if $m\in 2\N+1$, the same estimate holds for all $k\in\R^d\setminus\Rc_R^{K\frac{m+1}2}$.
Injecting this into the formula of Proposition~\ref{prop:sol-nonlin-rec} for the $\nu_{k}^n$'s (recall that $\Tc_m$ denotes the set of rooted $m$-trees~\eqref{eq:rooted-trees}), we obtain for all $n\ge1$ and $k\in\R^d\setminus\Rc_R^{K\lceil\frac n2\rceil}$,
\begin{multline*}
|\nu_{k}^n|\,\le\,\sum_{m=1}^{n+1}\sum_{a\in\Tc_m}\sum_{c\in\N^m\atop|c|=n+1-m}\sum_{b^1\in\N^{c_1}\atop|b^1|=a_1}~\ldots~\sum_{b^m\in\N^{c_m}\atop|b^m|=a_m}
\big(RK^{s_0}c_1^{s_0}\big)^{c_1+a_1}(CK)^{c_1 M}\|\hat V\|_{\Ld^\infty}^{c_1+1}\\
\times \ldots \big(RK^{s_0}c_m^{s_0}\big)^{c_m+a_m}(CK)^{c_m M}\|\hat V\|_{\Ld^\infty}^{c_m+1},
\end{multline*}
for some $C$ depending on $M$.
Since $\sharp\Tc_m\le4^m$, this directly leads to
\begin{align}\label{eq:est-nukn}
|\nu_{k}^n|\,\le\,(CRK^{s_0+M}n^{s_0})^{n}\|\hat V\|_{\Ld^\infty}^{n+1}.
\end{align}
Since $\nu_k^0=\expec{V}$, the same estimate obviously holds for $n=0$, and the conclusion~\eqref{eq:est-nu-QP} for $j=0$ follows.
We turn to the bounds on the correctors.
For that purpose, for all $s\ge0$, $m\ge1$, and $b\in\N^m$, we combine \eqref{e.FS} with Parseval's identity in the form
\begin{multline*}
\|\chi_k^{m,b}\|_{H^s(\Omega)}^2\,\lesssim\,\sum_{\xi_1\in\Z^M\setminus\{0\}}\langle\xi_1\rangle^{2s}\\
\times\bigg|\sum_{\xi_2,\ldots,\xi_m\in\Z^M\setminus\{0\}}\frac{\hat V(\xi_1-\xi_2)\ldots\hat V(\xi_{m-1}-\xi_m)\hat V(\xi_m)}{(|F\xi_1+k|^2-|k|^2)^{b_1+1}\ldots(|F\xi_m+k|^2-|k|^2)^{b_m+1}}\bigg|^2.
\end{multline*}
By the Cauchy-Schwarz inequality together with the compactness of the support of $\hat V$,
\begin{multline*}
\|\chi_k^{m,b}\|_{H^s(\Omega)}^2\,\le\,\|\hat V\|_{\Ld^2}^{2m}\\
\times\sup_{\xi_1\in\Z^M\setminus\{0\}}\sum_{\xi_2,\ldots,\xi_m\in\Z^M\setminus\{0\}}\frac{\langle\xi_1\rangle^{2s}\,\mathds1_{|\xi_1-\xi_2|\le K}\ldots\mathds1_{|\xi_{m-1}-\xi_m|\le K}\mathds1_{|\xi_m|\le K}}{(|F\xi_1+k|^2-|k|^2)^{2(b_1+1)}\ldots(|F\xi_m+k|^2-|k|^2)^{2(b_m+1)}}.
\end{multline*}
Hence, using Lemma~\ref{lem:dioph}, for all $k\in\R^d\setminus\Rc_R^{Km}$,
\begin{eqnarray*}
\|\chi_k^{m,b}\|_{H^s(\Omega)}^2&\le&(CK)^{2mM}(Km+1)^{2s}\|\hat V\|_{\Ld^\infty}^{2m}\prod_{j=1}^m\big(RK^{s_0}(m-j+1)^{s_0}\big)^{2(b_j+1)}\\
&\le&(CK)^{2mM}(Km+1)^{2s}(RK^{s_0}m^{s_0})^{2(m+|b|)}\|\hat V\|_{\Ld^\infty}^{2m},
\end{eqnarray*}
for some $C$ depending on $M$.
Injecting this estimate into the formula of Proposition~\ref{prop:sol-nonlin-rec} for the $\phi_k^n$'s, we obtain for all $n\ge1$ and $k\in\R^d\setminus\Rc_R^{Kn}$,
\[\|\phi_{k}^n\|_{H^s(\Omega)}\,\le\,(CKn)^{s}\sum_{m=1}^n\sum_{\ell=0}^{n-m}(CK)^{mM}(CRK^{s_0}m^{s_0})^{m+\ell}\|\hat V\|_{\Ld^\infty}^{m}\sum_{a\in\N^\ell\atop|a|=n-m-\ell}|\nu_{k}^{a_1}|\ldots|\nu_{k}^{a_\ell}|.\]
Combined with the bound~\eqref{eq:est-nukn} on the $\nu_k^n$'s, this yields
\[\|\phi_{k}^n\|_{H^s(\Omega)}\,\le\,(CKn)^{s}(CRK^{s_0+M}n^{s_0})^{n}\|\hat V\|_{\Ld^\infty}^{n},\]
that is,~\eqref{eq:est-cor-QP} for $j=0$.

\medskip
\step2 Conclusion.

We start with the proof of~\eqref{eq:est-cor-QP} for all $j\ge1$.
For that purpose, we note that, for $n\ge1$, $|\xi|\le n$, and $k\in\R^d\setminus \Rc_R^n$, Lemma~\ref{lem:dioph} yields
\begin{align*}
\bigg|\nabla_k^j \frac{1}{|F\xi+k|^2-|k|^2}\bigg|\,=\,\frac{j!\,|2F\xi|^j}{(|F\xi+k|^2-|k|^2)^{j+1}}\,\le\,\frac{(Cjn)^j(Rn^{s_0})^{j}}{||F\xi+k|^2-|k|^2|}.
\end{align*}
Taking the derivative $\nabla_k$ in both sides of the formulas of Proposition~\ref{prop:sol-nonlin-rec} for $(\nu_k^n,\phi_k^n)$, this bound allows to repeat the argument of Step~1, and the conclusion easily follows.

We turn to the proof of~\eqref{eq:est-cor-add-QP} and show that it follows from~\eqref{eq:est-cor-QP}. In the rest of this proof, we distinguish between the corrector $\phi_k^n$ defined on $\R^d\times\T^M$ and its folded version $\tilde\phi_k^n$ defined on $\T^M$, which are related via $\phi_k^n(x,\omega)=\tilde \phi_k^n(F^Tx+\omega)$.
By the Sobolev embedding, for $a>M$, the space $H^a(\T^M)$ is embedded into $\Ld^\infty(\T^M)$. Denote by $\nabla_{\T^M}$ the weak gradient on $\T^M$.
For all $\hat u\in C^\infty_c(\R^d)$, we have by definition of $\phi_k^n$, the Sobolev embedding, and Fubini's theorem,
\begin{eqnarray*}
\lefteqn{\sup_{\omega\in \T^M} \bigg(\int_{\R^d}\Big|\int_{\R^d}e^{ik\cdot x}\,\nabla_k^j\nabla^s\phi_k^n(x,\omega)\,\hat u(k)\,d^*k\Big|^2dx\bigg)^\frac12}\\
&= &\sup_{\omega\in \T^M}\bigg(\int_{\R^d}\Big|\int_{\R^d}e^{ik\cdot x}\,\hat u(k)\,\nabla_k^j(F\nabla_{\T^M})^s\tilde \phi_k^n(F^Tx+\omega)\,d^*k\Big|^2dx\bigg)^\frac12\\
&\le &\bigg(\int_{\R^d}\Big\|\int_{\R^d}e^{ik\cdot x}\,\hat u(k)\,\nabla_k^j\nabla_{\T^M}^s\tilde \phi_k^n\,d^*k\Big\|_{\Ld^\infty(\T^M)}^2dx\bigg)^\frac12\\
&\le &C\bigg(\int_{\R^d}\Big\|\int_{\R^d}e^{ik\cdot x}\,\hat u(k)\,\nabla_k^j\nabla_{\T^M}^{s+a}\tilde \phi_k^n\,d^*k\Big\|_{\Ld^2(\T^M)}^2dx\bigg)^\frac12.
\end{eqnarray*}
Hence, by Parseval's identity,
\begin{multline*}
\sup_{\omega\in \T^M} \bigg(\int_{\R^d}\Big|\int_{\R^d}e^{ik\cdot x}\,\nabla_k^j\nabla^s\phi_k^n(x,\omega)\,\hat u(k)\,d^*k\Big|^2dx\bigg)^\frac12\\
\,\le\,C\|u\|_{\Ld^2(\R^d)}\|\nabla_k^j\nabla_{\T^M}^{s+a}\tilde \phi_k^n\|_{\Ld^2(\T^M)}\,\le\,C\|u\|_{\Ld^2(\R^d)}\|\nabla_k^j \phi_k^n\|_{H^{s+a}(\T^M)}.
\end{multline*}
The conclusion~\eqref{eq:est-cor-add-QP} then follows from~\eqref{eq:est-cor-QP}.
\end{proof}


\section{Schr\"odinger flow 
} \label{chap:pr-taylorbloch}

\subsection{Main result and structure of the proof} 
The following asserts that bounds on the Rayleigh-Schrödinger coefficients imply an effective description of the Schrödinger flow on long timescales.
\begin{prop}[Approximate normal form via approximate Bloch waves]\label{prop:taylorbloch}
For $s_0>0$ and $\ell,R,K,M\ge1$, assume that the Rayleigh-Schrödinger coefficients $(\nu_k^n,\phi_k^n)_{k,n}$ satisfy the conclusions of Proposition~\ref{prop:cor-QP} for all $1\le n\le\ell$ and $k\in\R^d\setminus\mathcal R_R^{Kn}$,
and consider an initial condition $u^\circ\in\Sc(\R^d)$ with $\hat u^\circ$ compactly supported in~$\R^d\setminus \mathcal R_{R}^{K\ell}$.
Denote by $u_\lambda$ the Schrödinger flow
\begin{align}\label{eq:Schreqn}
i\partial_tu_\lambda=(-\triangle+\lambda V)u_\lambda,\qquad u_\lambda|_{t=0}=u^\circ,
\end{align}
and consider the following approximate flow
\begin{align}\label{eq:uapprox}
U_{\lambda}^{\ell;t}(x)\,:=\,\int_{\R^d} e^{-it(|k|^2+\kappa_{k,\lambda}^\ell)}\,e^{ik\cdot x}\,\hat u^\circ(k)\,d^*k.
\end{align}
Then, for all  $\lambda\le\frac12\big(CRK^{s_0+M}\ell^{s_0}\|\F \tilde V\|_{\Ld^\infty}\big)^{-1}$  and $T\ge0$,
\begin{align*}
&\sup_\Omega\sup_{0\le\lambda'\le\lambda}\, \sup_{0\le t\le T}\big\|u_{\lambda'}^t-U_{\lambda'}^{\ell;t}\big\|_{\Ld^2}
\,\le\,\lambda CR K^{s_0+2M+1}\ell^{s_0+M+1}\|\F \tilde V\|_{\Ld^\infty} \|u^\circ\|_{\Ld^2}\\
&\hspace{5cm}+\lambda^{\ell+1}T(K\ell)^{M+1}(CRK^{s_0+M}\ell^{s_0})^\ell \|\F \tilde V\|_{ \Ld^\infty}^{\ell+1} \|u^\circ\|_{\Ld^2}.\qedhere
\end{align*}
\end{prop}
The following result shows that such an $\Ld^2$ approximation result leads to ballistic transport properties. More precisely, the moments of $u_\lambda$ are shown to be close to those of the free flow~$u_0$ (although $u_\lambda^t$ is not close to $u_0^t$ in an $\Ld^2$ sense for $t\gtrsim\lambda^{-2}$, say if $\nu_k^1\ne0$).

\begin{prop}[Asymptotic ballistic transport]\label{prop:ball-tsp-weak}
Given $m\ge1$, $u^\circ\in\Sc(\R^d)$, and $V\in W^{m,\infty}(\R^d)$, denote by $u_\lambda$ the Schrödinger flow~\eqref{eq:Schreqn} and by $u_0$ the corresponding free flow, and for $\ell\ge1$ consider the approximate flow $U_\lambda^\ell$ defined in~\eqref{eq:uapprox}.
Then for all $t\ge1$,
\begin{multline*}
|M_m^t(u_\lambda)-M_m^t(u_0)|\,\lesssim_{m,u^\circ}\,\|u_\lambda^t-U_\lambda^{\ell;t}\|_{\Ld^2}^\frac1{m+1}+\|u_\lambda^t-U_\lambda^{\ell;t}\|_{\Ld^2}\\
+\sum_{j=1}^{m+1}\sum_{l=1}^j\Big(\int_{\R^d}|\nabla_k^l\kappa_{k,\lambda}^\ell|^\frac{2j}l|\langle\tfrac1t\nabla\rangle^{m+1}\hat u_0^t(k)|^2d^*k\Big)^\frac12,
\end{multline*}
where the multiplicative constant only depends on $d,m$, $\|V\|_{W^{m,\infty}}$, $\|\langle\nabla\rangle^{m+1}u^\circ\|_{\Ld^2}$, and $\|\langle\cdot\rangle^{m+1}u^\circ\|_{\Ld^2}$.
\end{prop}

Theorem~\ref{th:main-quasi} will follow from Proposition~\ref{prop:taylorbloch} and Proposition~\ref{prop:cor-QP} together with an optimization in the truncation parameter $\ell$ and an approximation argument for the initial data.
Corollary~\ref{th:main-quasi-tsp} will follow from Proposition~\ref{prop:ball-tsp-weak} together with further approximation arguments.
In the following section we split the proof of Proposition~\ref{prop:taylorbloch} into a string of lemmas, which are then proved in the subsequent sections. Next, we turn to the proof of Proposition~\ref{prop:ball-tsp-weak} and we conclude the section with the proofs of Theorem~\ref{th:main-quasi}, Corollary~\ref{th:main-quasi-tsp}, and Corollary~\ref{th:main-quasi-tsp-2}.

 \subsection{Structure of the proof of Proposition~\ref{prop:taylorbloch}}
As motivated in Section~\ref{chap:blochballi}, we start by considering the following {approximate Bloch expansion} of the initial data $u^\circ$,
\begin{align}\label{eq:taylblochIC}
W_{\lambda}^{\ell;\circ}(x):=\int_{\R^d} e^{ik\cdot x}\,\psi_{k,\lambda}^{\ell}(x)\,\hat u^\circ(k)\,d^*k.
\end{align}
For small $\lambda$, we can indeed formally replace the Fourier modes $x\mapsto e^{ik\cdot x}$ by the corresponding approximate Bloch waves $x\mapsto e^{ik\cdot x}\psi_{k,\lambda}^\ell(x)$. 
The following lemma quantifies the resulting error and its propagation in time.
\begin{lem}[Preparation of initial data]\label{lem:closeinitcond}
In the setting of Proposition~\ref{prop:taylorbloch}, we define $W_{\lambda}^{\ell;\circ}$ as in~\eqref{eq:taylblochIC} and we denote by $W_{\lambda}^{\ell}$ the solution of
\begin{align}\label{eq:approx-ini-cond-v}
i\partial_t W_{\lambda}^{\ell}=(-\triangle+\lambda V)W_{\lambda}^{\ell},\qquad W_{\lambda}^{\ell}|_{t=0}=W_{\lambda}^{\ell;\circ}.
\end{align}
Then, for all $\lambda\le\frac12\big(CRK^{s_0+M}\ell^{s_0}\|\F \tilde V\|_{ \Ld^\infty}\big)^{-1}$ and $t\ge0$,
\begin{equation*}
\sup_\Omega\sup_{0\le\lambda'\le\lambda}\big\|u_{\lambda'}^t-W_{\lambda'}^{\ell;t}\big\|_{\Ld^2}
\,\le\,\lambda CR K^{s_0+2M+1}\ell^{s_0+M+1}\|\F \tilde V\|_{\Ld^\infty} \|u^\circ\|_{\Ld^2}.\qedhere
\end{equation*}
\end{lem}

Next, starting with the approximate Bloch expansion $W_{\lambda}^{\ell;\circ}$ of the initial data and using that approximate Bloch waves approximately diagonalize the Schrödinger operator (cf.~Lemma~\ref{lem:eqnssumcorr}), we arrive at an approximate Bloch expansion of the Schrödinger flow.

\begin{lem}[Approximate diagonalization]\label{lem:almdiag}
In the setting of Proposition~\ref{prop:taylorbloch}, denote by $W_{\lambda}^{\ell}$ the solution of~\eqref{eq:approx-ini-cond-v} and let $V_{\lambda}^\ell$ be given by
\begin{align}\label{eq:def-Vlambd-apprBloch}
V_{\lambda}^{\ell;t}(x)\,:=\,\int_{\R^d} e^{-it(|k|^2+\kappa_{k,\lambda}^\ell)}\,e^{ik\cdot x}\,\psi_{k,\lambda}^\ell(x)\,\hat u^\circ(k)\,d^*k.
\end{align}
Then, for all $\lambda\le\frac12\big(CRK^{s_0+M}\ell^{s_0}\|\F \tilde V\|_{\Ld^\infty}\big)^{-1}$ and $t\ge0$,
\begin{equation*}
\sup_\Omega\sup_{0\le\lambda'\le\lambda}\|W_{\lambda'}^{\ell;t}-V_{\lambda'}^{\ell;t}\|_{\Ld^2}\,\le\,\lambda^{\ell+1}t(K\ell)^{M+1}(CRK^{s_0+M}\ell^{s_0})^\ell \|\F \tilde V\|_{ \Ld^\infty}^{\ell+1} \|u^\circ\|_{\Ld^2}.\qedhere
\end{equation*}
\end{lem}

In turn, the approximate Bloch waves $x\mapsto e^{ik\cdot x}\psi_{k,\lambda}^\ell(x)$ can be replaced  by the Fourier modes $x\mapsto e^{ik\cdot x}$, which leads to the expected effective flow $U_\lambda^\ell$.
\begin{lem}\label{lem:throw-away}
In the setting of Proposition~\ref{prop:taylorbloch}, let $V_\lambda^\ell$ be defined as in~\eqref{eq:def-Vlambd-apprBloch} and let $U_\lambda^\ell$ be given by
\[U_{\lambda}^{\ell;t}(x)\,:=\,\int_{\R^d} e^{-it(|k|^2+\kappa_{k,\lambda}^\ell)}\,e^{ik\cdot x}\,\hat u^\circ(k)\,d^*k.\]
Then, for all $\lambda\le\frac12\big(CRK^{s_0+M}\ell^{s_0}\|\F \tilde V\|_{\Ld^\infty}\big)^{-1}$ and $t\ge0$,
\begin{equation*}
\sup_\Omega\sup_{0\le\lambda'\le\lambda}\|V_{\lambda'}^{\ell;t}-U_{\lambda'}^{\ell;t}\|_{\Ld^2}\,\le\,\lambda CRK^{s_0+2M+1}\ell^{s_0+M+1}\|\F \tilde V\|_{ \Ld^\infty} \|u^\circ\|_{\Ld^2}.\qedhere
\end{equation*}
\end{lem}

Proposition~\ref{prop:taylorbloch} follows from the decomposition $u_{\lambda}-U_{\lambda}^{\ell}=(u_{\lambda}-W_{\lambda}^{\ell})+(W_{\lambda}^{\ell}-V_{\lambda}^{\ell})+(V_{\lambda}^{\ell}-U_{\lambda}^{\ell})$, the triangle inequality, and the combination of Lemmas~\ref{lem:closeinitcond},~\ref{lem:almdiag}, and~\ref{lem:throw-away}.


\subsection{Proof of Lemma~\ref{lem:closeinitcond}: Preparation of initial data}

Since the difference $u_\lambda-W_{\lambda}^\ell$ satisfies
\[i\partial_t(u_\lambda-W_{\lambda}^\ell)=(-\triangle+\lambda V)(u_\lambda-W_{\lambda}^\ell),\qquad (u_\lambda-W_{\lambda}^\ell)|_{t=0}=u^\circ-W_{\lambda}^{\ell;\circ},\]
the unitarity of the Schrödinger flow yields for all $t\ge0$,
\begin{equation*} 
\| u_\lambda^t-W_{\lambda}^{\ell;t}\|_{\Ld^2}\,=\, \|u^\circ-W_{\lambda}^{\ell;\circ}\|_{\Ld^2},
\end{equation*}
so that it suffices to prove that for all $\lambda\le\frac12(CRK^{s_0+M}\ell^{s_0}\|\F \tilde V\|_{\Ld^\infty})^{-1}$,
\begin{equation}\label{eq:S1-lem1-tec}
\sup_\Omega\sup_{0\le\lambda'\le\lambda}\|u^\circ-W_{\lambda'}^{\ell;\circ}\|_{\Ld^2}
\,\le \,\lambda CRK^{s_0+2M+1}\ell^{s_0+M+1}\|\F \tilde V\|_{ \Ld^\infty}\|u^\circ\|_{\Ld^2}.
\end{equation}
By definition~\eqref{eq:taylblochIC} of~$W_{\lambda}^{\ell;\circ}$ and by Definition~\ref{def:taylor-waves},
\begin{align*} 
(u^\circ-W_{\lambda}^{\ell;\circ})(x)=-\sum_{n=1}^\ell\lambda^n\int_{\R^d} e^{ik\cdot x}\,\phi_k^n(x)\,\hat u^\circ(k)\,d^*k.
\end{align*}
Hence, by assumption~\eqref{eq:est-cor-add-QP},
\begin{multline}
{\sup_\Omega\sup_{0\le\lambda'\le\lambda}\|u^\circ-W_{\lambda'}^{\ell;\circ}\|_{\Ld^2}
\,\le\,\sum_{n=1}^\ell\lambda^n\,\sup_{\omega\in\Omega}\bigg(\int_{\R^d}\Big|\int_{\R^d} e^{ik\cdot x}\, \phi_k^n(x,\omega)\,\hat u^\circ(k)\,d^*k\Big|^2dx\bigg)^\frac12} \\
\,\lesssim \,K^{M+1}\| u^\circ\|_{\Ld^2}\sum_{n=1}^\ell\lambda^nn^{M+1} (CRK^{s_0+M}n^{s_0})^n \|\F \tilde V\|_{\Ld^\infty}^n,
\label{e.pr-well-prepared}
\end{multline}
and the claim~\eqref{eq:S1-lem1-tec} follows.\qed


\subsection{Proof of Lemma~\ref{lem:almdiag}: Approximate diagonalization}

We first claim that
\begin{equation}\label{eq:boundL2-util-v}
\sup_\Omega\sup_{0\le\lambda'\le\lambda}\, \sup_{0\le t\le T}\|W_{\lambda'}^{\ell;t}-V_{\lambda'}^{\ell;t}\|_{\Ld^2} \,\lesssim\, \int_0^T\lambda^{\ell+1}\sup_\Omega\sup_{0\le\lambda'\le\lambda}\|F_{\lambda'}^{\ell;t}\|_{\Ld^2}\,dt,
\end{equation}
in terms of
\begin{eqnarray}
F_{\lambda}^{\ell;t}(x)&:=&\int_{\R^d} e^{-it(|k|^2+\kappa_{k,\lambda}^\ell)}e^{ik\cdot x}\,\df^\ell_{k,\lambda}(x)\,\hat u^\circ(k)\,d^*k,
\label{eq:def-F-S1}
\end{eqnarray}
where $\df^\ell_{k,\lambda}$ denotes the eigendefect (cf.~Lemma~\ref{lem:eqnssumcorr}).
Indeed, by definition of $V_{\lambda}^\ell$ and by Lemma~\ref{lem:eqnssumcorr}, we find
\begin{eqnarray*}
\lefteqn{\big(i\partial_t+\triangle-\lambda V(x)\big) V_{\lambda}^{\ell;t}(x)}\\
&=&\int_{\R^d} e^{-it(|k|^2+\kappa_{k,\lambda}^\ell)}e^{ik\cdot x}\,\big((\triangle_k-\lambda V+\kappa_{k,\lambda}^\ell) \psi_{k,\lambda}^\ell\big)(x)\,\hat u^\circ(k)\,d^*k\\
&=&-\lambda^{\ell+1}F_{\lambda}^{\ell;t}(x),
\end{eqnarray*}
so that the difference $W_{\lambda}^\ell-V_{\lambda}^\ell$ satisfies
\begin{align*}
\big(i\partial_t+\triangle-\lambda V\big)(W_{\lambda}^{\ell}-V_{\lambda}^\ell)\,=\,\lambda^{\ell+1}F_{\lambda}^{\ell},\qquad(W_{\lambda}^{\ell}-V_{\lambda}^\ell)|_{t=0}=0.
\end{align*}
Duhamel's formula together with the unitarity of the Schrödinger flow then yields
\begin{equation*} 
 \sup_{0\le t\le T}\|W_{\lambda}^{\ell;t}-V_{\lambda}^{\ell;t}\|_{\Ld^2} \,\le\,
 \int_0^T  \lambda^{\ell+1} \| F_{\lambda}^{\ell;t}\|_{\Ld^2}\,dt,
\end{equation*}
and~\eqref{eq:boundL2-util-v} follows.
By~\eqref{eq:boundL2-util-v}, it now suffices to prove that for all $\lambda\le\frac12(CRK^{s_0+M}\ell^{s_0}\|\F \tilde V\|_{\Ld^\infty})^{-1}$ and $0\le t\le T$,
\begin{equation}\label{eq:bound-diag-pre-nabm}
\sup_\Omega\sup_{0\le\lambda'\le\lambda}\|F_{\lambda'}^{\ell;t}\|_{\Ld^{2}}
\,\lesssim\, (K\ell)^{M+1}(CRK^{s_0+M}\ell^{s_0})^\ell \|\F \tilde V\|_{\Ld^\infty}^{\ell+1}\|u^\circ\|_{\Ld^2}.
\end{equation}
By definition~\eqref{eq:def-F-S1} of $F_{\lambda}^\ell$ and by definition of $\df_{k,\lambda}^\ell$ (cf.~Lemma~\ref{lem:eqnssumcorr}), the assumptions~\eqref{eq:est-nu-QP} and~\eqref{eq:est-cor-add-QP} yield
\begin{multline*}
\|F_\lambda^{\ell;t}\|_{\Ld^2}
\,=\,\bigg(\int_{\R^d}\Big|\int_{\R^d} e^{-it(|k|^2+\kappa_{k,\lambda}^\ell)}e^{ik\cdot x}\, \df^\ell_{k,\lambda}(x)\,\hat u^\circ(k)\,d^*k\Big|^2dx\bigg)^\frac12\nonumber\\
\,\le\,(K\ell)^{M+1}(CRK^{s_0+M}\ell^{s_0})^\ell\|\F\tilde V\|_{\Ld^\infty}^{\ell+1}\|u^\circ\|_{\Ld^2}\\
+(K\ell)^{M+1}\|u^\circ\|_{\Ld^2}\sum_{n=1}^\ell\lambda^n(CRK^{s_0+M}\ell^{s_0})^{n+\ell}\|\F\tilde V\|_{\Ld^\infty}^{n+\ell+1},
\end{multline*}
and the claim~\eqref{eq:bound-diag-pre-nabm} follows.
\qed

\subsection{Proof of Lemma~\ref{lem:throw-away}}\label{sec:theor5i}

Since
\begin{eqnarray*}
(V_{\lambda}^{\ell;t}-U_{\lambda}^{\ell;t})(x)\,=\,\sum_{n=1}^\ell\lambda^n\int_{\R^d} e^{-it(|k|^2+\kappa_{k,\lambda}^\ell)}e^{ik\cdot x}\,\phi_{k}^n(x)\,\hat u^\circ(k)\,d^*k,
\end{eqnarray*}
the desired estimate directly follows from assumption~\eqref{eq:est-cor-add-QP} as in~\eqref{e.pr-well-prepared}.
\qed

\subsection{Proof of Proposition~\ref{prop:ball-tsp-weak}: Asymptotic ballistic transport}\label{sec:prop:tsp-weak}
Before proceeding to the proof, we recall the following a priori estimate for the Schrödinger flow in weighted norms. This result is due to Ozawa~\cite[Theorem~1]{Ozawa-91}.
\begin{lem}[\cite{Ozawa-91}]\label{prop:moments}
Given $z^\circ\in\Sc(\R^d)$, $F\in\Ld^\infty_\loc(\R^+;\Ld^2(\R^d))$, and a (real-valued) potential $V\in W^{m-1,\infty}(\R^d)$ with $m\ge1$, denote by $z\in\Ld^\infty(\R^+;\Ld^2(\R^d))$ the solution of the Schrödinger equation
\begin{equation}\label{e.prop:moments-1}
(i\partial_t+\triangle-V)z^t=F^t,\qquad z^t|_{t=0}=z^\circ.
\end{equation}
Then, for all $t\ge0$ and $i,j\ge0$ with $i+j=m$, we have
\begin{multline*}
\langle t\rangle^{-i}\|\langle\cdot\rangle^{i} \langle \nabla \rangle^{j} z^t\|_{\Ld^2}\,\lesssim_{m}\,\|\langle\nabla\rangle^{m}z^\circ\|_{\Ld^2}+\langle t\rangle^{-m}\|\langle\cdot\rangle^{m}z^\circ\|_{\Ld^2}\\
+\int_0^t\Big(\|\langle\nabla\rangle^{m}F^s\|_{\Ld^2}+\langle t\rangle^{-m}\|\langle\cdot\rangle^{m}F^s\|_{\Ld^2}\Big)ds,
\end{multline*}
where the multiplicative constant depends only on $d,m,\|V\|_{W^{m-1,\infty}}$.
\end{lem}

With this estimate at hand, we show that the rescaled moments $M_m^t(u_\lambda)$ of the Schrödinger flow (cf.~\eqref{eq:def-moment-Mm}) can be truncated in a ballistic scaling. For $C_0\ge1$, we define the ballistically truncated moment via
\begin{align}\label{eq:trunc-mom}
\tilde M_m^t(u_\lambda;C_0)\,:=\,\big\|\big(\tfrac{|\cdot|}{t}\big)^me^{-\frac12(\frac{|\cdot|}{C_0t})^2}u_\lambda^t\big\|_{\Ld^2}.
\end{align}
This is the starting point for the proof of Proposition~\ref{prop:ball-tsp-weak}.

\begin{cor}[Ballistic truncation]\label{cor:moment-trunc}
Given $u^\circ\in\Sc(\R^d)$ and a (real-valued) potential $V\in W^{m,\infty}(\R^d)$ with $m\ge1$, denote by $u\in\Ld^\infty(\R^+;\Ld^2(\R^d))$ the solution of the Schrödinger flow
\[i\partial_tu=(-\triangle+V)u,\qquad u|_{t=0}=u^\circ.\]
Then, for all $C_0\ge1$ and $t\ge0$, we have
\[\big|M_m^t(u)-\tilde M_m^t(u;C_0)\big|\,\lesssim_m\,C_0^{-1}\big(\|\langle\nabla\rangle^{m+1}u^\circ\|_{\Ld^2}+\|\langle\cdot\rangle^{m+1}u^\circ\|_{\Ld^2}\big),\]
where the multiplicative constant depends only on $d,m,\|V\|_{W^{m,\infty}}$.
\end{cor}

\begin{proof}
We claim that it suffices to establish the following estimate,
\begin{gather}
\big|M_m^t(u)-\tilde M_m^t(u;C_0)\big|\,\lesssim\,C_0^{-1}M_{m+1}^t(u),\label{eq:cruc-ball-1}
\end{gather}
since then the conclusion follows from Lemma~\ref{prop:moments} in the form $M_{m+1}^t(u)\lesssim_m\|\langle\nabla\rangle^{m+1}u^\circ\|_{\Ld^2}+\|\langle\cdot\rangle^{m+1}u^\circ\|_{\Ld^2}$.
In order to prove~\eqref{eq:cruc-ball-1}, it suffices to show that for all $R>0$,
\begin{align}\label{eq:reducebis-moment-Schr}
\||\cdot|^m(1-\gamma_R)\,u\|_{\Ld^2}\,\lesssim\,R^{-1}\||\cdot|^{m+1}u\|_{\Ld^2},
\end{align}
in terms of the Gaussian cut-off $\gamma_R(x):=e^{-\frac12(\frac{|x|}{R})^2}$. For that purpose, we write in Fourier space
\begin{align*}
\||\cdot|^m(1-\gamma_R)\,u\|_{\Ld^2}^2\,=\,\int_{\R^d}\big|\nabla^m\hat u(k)-\hat\gamma_R\ast\nabla^m\hat u(k)\big|^2d^*k,
\end{align*}
where $\hat\gamma_{R} (k):=(\sqrt{2\pi}R)^de^{-\frac12 (R|k|)^2}$.
Since $\int_{\R^d}\hat\gamma_R(k)\,d^*k=1$ and $\int_{\R^d}|k|^2\hat\gamma_R(k)\,d^*k\lesssim R^{-2}$, the Cauchy-Schwarz inequality yields
\begin{eqnarray*}
\||\cdot|^m(1-\gamma_R)\,u\|_{\Ld^2}^2&=&\int_{\R^d}\Big|\int_{\R^d}\hat\gamma_R(k')\big(\nabla^m\hat u(k)-\nabla^m\hat u(k+k')\big)\,d^*k'\Big|^2d^*k\\
&\lesssim&R^{-2}\int_{\R^d}|\nabla^{m+1}\hat u(k)|^2\,d^*k,
\end{eqnarray*}
that is,~\eqref{eq:reducebis-moment-Schr}.
\end{proof}

We may now turn to the proof of Proposition~\ref{prop:ball-tsp-weak}.

\begin{proof}[Proof of Proposition~\ref{prop:ball-tsp-weak}]
Let $m\ge0$ be fixed. In this proof, we use the notation $\lesssim_{m,u^\circ}$ for $\le$ up to a multiplicative constant that only depends on $d$, $m$, $\|V\|_{W^{m,\infty}}$, $\|\langle\nabla\rangle^{m+1}u^\circ\|_{\Ld^2}$, and $\|\langle\cdot\rangle^{m+1}u^\circ\|_{\Ld^2}$.
The starting point is the triangle inequality in the following form, for all $C_0\ge1$,
\begin{multline*}
|M_m^t(u_\lambda)-M_m^t(u_0)|\,\le\,\big|M_m^t(u_\lambda)-\tilde M_m^t(u_\lambda;C_0)\big|+\big|\tilde M_m^t(u_\lambda;C_0)-\tilde M_m^t(U_\lambda^\ell;C_0)\big|\\
+\big|M_m^t(U_\lambda^\ell)-\tilde M_m^t(U_\lambda^\ell;C_0)\big|+\big|M_m^t(U_\lambda^\ell)-M_m^t(u_0)\big|.
\end{multline*}
Using Corollary~\ref{cor:moment-trunc} to estimate the first RHS term and using~\eqref{eq:cruc-ball-1} to estimate the third one, this yields
\begin{multline}\label{eq:first-decomp-Mmulambda}
|M_m^t(u_\lambda)-M_m^t(u_0)|\,\lesssim_{m,u^\circ}\,C_0^{-1}+C_0^m\|u_\lambda^t-U_\lambda^{\ell;t}\|_{\Ld^2}\\
+C_0^{-1}M_{m+1}^t(U_\lambda^\ell)+\big|M_m^t(U_\lambda^\ell)-M_m^t(u_0)\big|.
\end{multline}
It remains to prove the following estimates for all $t\ge1$:
\begin{gather}
M_{m+1}^t(U_\lambda^\ell)\,\lesssim_{m,u^\circ}\,1+\sum_{j=1}^{m+1}\sum_{l=1}^j\Big(\int_{\R^d}|\nabla_k^l\kappa_{k,\lambda}^\ell|^\frac{2j}l|(\tfrac1t\nabla)^{m+1-j}\hat u_0^t(k)|^2d^*k\Big)^\frac12,\label{eq:cruc-ball-2}\\
\big|M_m^t(U_\lambda^\ell)-M_m^t(u_0)\big|\,\lesssim_{m,u^\circ}\,\sum_{j=1}^m\sum_{l=1}^j\Big(\int_{\R^d}|\nabla_k^l\kappa_{k,\lambda}^\ell|^\frac{2j}l|(\tfrac1t\nabla)^{m-j}\hat u_0^t(k)|^2d^*k\Big)^\frac12,\label{eq:cruc-ball-3}
\end{gather}
Injecting these estimates into~\eqref{eq:first-decomp-Mmulambda} and optimizing wrt $C_0\ge1$, the conclusion follows.

We start with the proof of~\eqref{eq:cruc-ball-3},
for which we argue in Fourier space. By definition~\eqref{eq:uapprox} of $U_\lambda^{\ell}$, we have
\[\hat U_\lambda^{\ell;t}(k)\,=\,e^{-it(|k|^2+\kappa_{k,\lambda}^\ell)}\hat u^\circ(k)\,=\,e^{-it\kappa_{k,\lambda}^\ell}\,\hat u_0^t(k),\]
so that the Leibniz rule leads to
\[\big|\nabla^m\hat U_\lambda^{\ell;t}(k)-e^{-it\kappa_{k,\lambda}^\ell}\nabla^m\hat u_0^t(k)\big|\,\le\,\sum_{j=1}^m\binom{m}j |\nabla_k^{j}e^{-it\kappa_{k,\lambda}^\ell}||\nabla^{m-j}\hat u_0^t(k)|.\]
Integrating wrt $k$ and using the triangle inequality, we deduce
\[|M_m^t(U_\lambda^\ell)-M_m^t(u_0)|\,\lesssim_m\,t^{-m}\sum_{j=1}^m\Big(\int_{\R^d}|\nabla_k^{j}e^{-it\kappa_{k,\lambda}^\ell}|^2|\nabla^{m-j}\hat u_0^t(k)|^2d^*k\Big)^\frac12.\]
For $j\ge1$, we compute
\[|\nabla_k^je^{-it\kappa_{k,\lambda}^\ell}|\,\lesssim_j\,\sum_{l=1}^jt^{\frac jl}|\nabla_k^l\kappa_{k,\lambda}^\ell|^\frac jl\lesssim_j\,\langle t\rangle^{j}\sum_{l=1}^j|\nabla_k^l\kappa_{k,\lambda}^\ell|^\frac jl,\]
and the claim~\eqref{eq:cruc-ball-3} follows for all $t\ge1$.
Likewise, this argument for $m$ replaced by $m+1$ and combined with the a priori estimate of Lemma~\ref{prop:moments} in the form $M_m^t(u_0)\lesssim_{m,u^\circ}1$ yields~\eqref{eq:cruc-ball-2}.
\end{proof}

\subsection{Proof of Theorem~\ref{th:main-quasi}}\label{chap:proof-main-quasi}
In order to apply Proposition~\ref{prop:taylorbloch} with finite parameters $\ell,R,K\ge1$, we first need to proceed to two truncation procedures:
\begin{enumerate}[$\bullet$]
\item cut frequencies higher than $K$ in the potential $V$;
\item project the initial data $u^\circ$ onto the restricted non-resonant set $\R^d\setminus\Rc_R$ in Fourier space.
\end{enumerate}
We start with the frequency cut-off: for all $K\ge1$ we define $\tilde V_K:=\F^{-1}(\F \tilde V\,\mathds1_{|\cdot|\le K})$,
and we consider the corresponding Schrödinger flow,
\[i\partial_t u_{K,\lambda}=(-\triangle+\lambda V_K) u_{K,\lambda},\qquad  u_{K,\lambda}|_{t=0}= u^\circ.\]
The Gevrey regularity assumption on $V$ implies $\|V_K-V\|_{\Ld^\infty} \le\|\F(\tilde V_K-\tilde V)\|_{\Ld^1}\le e^{-K^\alpha}$, so that Duhamel's formula and unitarity yield for all $t\ge 0$,
\begin{equation}\label{e.frq-cut-off}
\|u^t_\lambda -u_{K,\lambda}^t\|_{\Ld^2} \,\lesssim\, \lambda t e^{-K^\alpha} \|u^\circ\|_{\Ld^2}.
\end{equation}
We turn to the projection of the initial data: for $R\ge1$ we define
 \[u^\circ_{R}\,:=\,\F^{-1}\big[\mathds1_{B_{R^{1/d}}}\,\mathds1_{\R^d\setminus {\Rc_R}}\,\hat u^\circ\big],\]
and we consider the corresponding Schrödinger flow,
\begin{align}\label{eq:def-uKRlam}
i\partial_t u_{K,R,\lambda}=(-\triangle+\lambda V_K) u_{K,R,\lambda},\qquad  u_{K,R,\lambda}|_{t=0}= u_{R}^\circ.
\end{align}
By unitarity and by Lemma~\ref{lem:dioph}(ii), we find for all $t\ge0$,
\begin{eqnarray}\label{eq:approx-10-qp}
\|u_{K,\lambda}^t- u_{K,R,\lambda}^t\|_{\Ld^2}\,=\,\|u^\circ-u_{R}^\circ\|_{\Ld^2}&\le&\|\mathds1_{\R^d\setminus B_{R^{1/d}}}\hat u^\circ\|_{\Ld^2}+\|\mathds1_{B_{R^{1/d}}\cap\Rc_R}\,\hat u^\circ\|_{\Ld^2}\nonumber\\
&\le&R^{-\frac 1d}\|\langle\cdot\rangle\hat u^\circ\|_{\Ld^2}+\big|B_{R^{1/d}}\cap\Rc_R\big|^\frac12\|\hat u^\circ\|_{\Ld^\infty}\nonumber\\
&\lesssim&R^{-\frac1{2d}}\big(\|\langle\cdot\rangle\hat u^\circ\|_{\Ld^2}+\|\hat u^\circ\|_{\Ld^\infty}\big).
\end{eqnarray}
Combining this with~\eqref{e.frq-cut-off}, we deduce
\begin{align}\label{eq:approx-1-qp}
\|u^t_\lambda -u_{K,R,\lambda}^t\|_{\Ld^2} \,\lesssim\, \big(R^{-\frac1{2d}}+\lambda t e^{-K^\alpha}\big)\big(\|\langle\cdot\rangle\hat u^\circ\|_{\Ld^2}+\|\hat u^\circ\|_{\Ld^\infty}\big).
\end{align}

We may now apply Proposition~\ref{prop:taylorbloch} to the truncated Schrödinger flow $u_{K,R,\lambda}$.
Denote by $\kappa^{\ell}_{K,k,\lambda}$ the Bloch eigenvalues associated with $V$ replaced by $V_K$ (and by $\nu_{K,k}^n$ the corresponding Rayleigh-Schrödinger coefficients), and define the corresponding truncated approximate flow
\begin{align}\label{eq:def-approx-flow-UKR}
U^{\ell;t}_{K,R,\lambda}(x)\,:=\,\int_{\R^d}e^{-it(|k|^2+\kappa_{K,k,\lambda}^{\ell})}\,e^{ik\cdot x}\,\hat u^\circ_{R}(k)\,d^*k.
\end{align}
Proposition~\ref{prop:taylorbloch} together with~\eqref{eq:approx-1-qp} and the assumption $\|\F\tilde V\|_{\Ld^\infty}\le1$ then yields for all $\lambda\le\frac12(CRK^{s_0+M}\ell^{s_0})^{-1}$ and $T\ge0$,
\begin{multline}\label{eq:approx-apply-1}
\sup_\Omega\sup_{0\le\lambda'\le\lambda}\, \sup_{0\le t\le T}\big\|u_{\lambda'}^t-U_{K,R,\lambda'}^{\ell;t}\big\|_{\Ld^2}
\,\lesssim\,\big(\|\langle\cdot\rangle\hat u^\circ\|_{\Ld^2}+\|\hat u^\circ\|_{\Ld^\infty}\big)\\
\times\big(R^{-\frac1{2d}}+\lambda Te^{-K^\alpha}+\lambda R K^{s_0+2M+1}\ell^{s_0+M+1}+\lambda^{\ell+1}T(K\ell)^{M+1}(CRK^{s_0+M}\ell^{s_0})^\ell\big).
\end{multline}

Finally, we show how to replace $U_{K,R,\lambda}^\ell$ by the approximate flow $U_\lambda^\ell$ without cut-off as defined in~\eqref{eq:def-Uell-appr}.
Since $\hat u_R^\circ$ is supported in $\R^d \setminus \overline{\Rc_R}\subset \mathcal O$, we can write
\[U^{\ell;t}_{K,R,\lambda}(x)\,=\,\int_{\R^d}e^{-it(|k|^2+\mathds{1}_{\Oc}(k)\kappa_{K,k,\lambda}^\ell)}\,e^{ik\cdot x}\,\hat u^\circ_{R}(k)\,d^*k,\]
so that comparing with~\eqref{eq:def-Uell-appr} yields
\begin{align*}
\|U^{\ell;t}_\lambda-U^{\ell;t}_{K,R,\lambda}\|_{\Ld^2}\,\lesssim\,\|u^\circ-u_R^\circ\|_{\Ld^2}+t\,\Big(\int_{\R^d\setminus\overline{\Rc_R}}|\kappa_{K,k,\lambda}^\ell-\kappa_{k,\lambda}^\ell|^2|\hat u^\circ(k)|^2\,d^*k\Big)^\frac12.
\end{align*}
It remains to estimate the second RHS term. Since $\|\F(\tilde V-\tilde V_K)\|_{\Ld^\infty}\le \|\F(\tilde V-\tilde V_K)\|_{\Ld^1}\le e^{-K^\alpha}$, the argument in Step~1 of the proof of Proposition~\ref{prop:cor-QP} easily yields for all $n\ge1$ and $k\in\R^d\setminus\Rc_R^{Kn}$,
\[|\nu_{K,k}^n-\nu_{2K,k}^n|\,\lesssim\, e^{-K^\alpha}(CRK^{s_0+M}n^{s_0})^n.\]
Hence, for all $n\ge1$ and $k\in\R^d\setminus\Rc_R=\R^d \setminus \big( \cup_m  \Rc_R^m\big)$, provided that $K\ge C( n \log n)^\frac1\alpha$ for some large enough constant $C$  (depending only on $s_0,M,\alpha$), we obtain by a dyadic decomposition
\begin{eqnarray*}
|\nu_{k}^n-\nu_{K,k}^n|\,\le\,\sum_{r=0}^\infty e^{-(2^rK)^\alpha}\big(CR(2^rK)^{s_0+M}n^{s_0}\big)^n\,\le\,e^{-K^\alpha}(CRK^{s_0+M}n^{s_0})^n,
\end{eqnarray*}
which implies for all $\lambda\le\frac12(CRK^{s_0+M}\ell^{s_0})^{-1}$, $k\in\R^d\setminus\Rc_R$, and $K \ge C( \ell \log \ell)^\frac1\alpha$,
\[|\kappa_{k,\lambda}^\ell-\kappa_{K,k,\lambda}^\ell|\,\le\,\lambda \sum_{n=1}^\ell \lambda^n |\nu_{k}^n-\nu_{K,k}^n|\,\le\,\lambda e^{-K^\alpha}.\]
Injecting this together with~\eqref{eq:approx-10-qp} into the above, we obtain
\begin{align*}
\|U^{\ell;t}_\lambda-U^{\ell;t}_{K,R,\lambda}\|_{\Ld^2}\,\lesssim\,\big(R^{-\frac1{2d}}+\lambda t e^{-K^\alpha}\big)\big(\|\langle\cdot\rangle\hat u^\circ\|_{\Ld^2}+\|\hat u^\circ\|_{\Ld^\infty}\big),
\end{align*}
and~\eqref{eq:approx-apply-1} turns into
\begin{multline}\label{eq:pre-optimiz-approxu}
\sup_\Omega\sup_{0\le\lambda'\le\lambda}\, \sup_{0\le t\le T}\big\|u_{\lambda'}^t-U_{\lambda'}^{\ell;t}\big\|_{\Ld^2}
\,\lesssim\,\big(\|\langle\cdot\rangle\hat u^\circ\|_{\Ld^2}+\|\hat u^\circ\|_{\Ld^\infty}\big)\\
\times\Big(R^{-\frac1{2d}}+\lambda Te^{-K^\alpha}+\lambda R K^{s_0+2M+1}\ell^{s_0+M+1}+\lambda^{\ell+1}T(CRK^{s_0+2M+1}\ell^{s_0+M+1})^\ell\Big).
\end{multline}
For $\gamma<\frac1{2d+1}$, choosing
\begin{eqnarray*}
R\,=\,R(\lambda)&:=&\lambda^{-2d\gamma},\\
K\,=\,K(\lambda)&:=&C(\ell(\lambda)\log \ell(\lambda))^\frac1\alpha,\\
T\,=\,T(\lambda)&:=&e^{\ell(\lambda)},\\
\ell\,=\,\ell(\lambda)&:=&\lambda^{-\frac{\alpha(1-(2d+1)\gamma)}{s_0+2M+1+\alpha(s_0+M+1)}},
\end{eqnarray*}
the conclusion follows after straightforward simplifications.\qed

\subsection{Proof of Corollary~\ref{th:main-quasi-tsp}.}
Let $m\ge0$ be fixed. In this proof, we use the notation $\lesssim_{m,u^\circ}$ for $\le$ up to a multiplicative constant that only depends on $d$, $m$, $\|\langle\nabla\rangle^{m+1}u^\circ\|_{\Ld^2}$, $\|\langle\cdot\rangle^{m+1}u^\circ\|_{\Ld^2}$, and $\|\langle\cdot\rangle^m\hat u^\circ\|_{\Ld^\infty}$.
Again, we may not directly apply Proposition~\ref{prop:ball-tsp-weak} since the approximate Bloch eigenvalues $\kappa_{k,\lambda}^\ell$ can only be estimated for $k$ away from the resonant set. Therefore, we first need to proceed to similar truncations of the initial data $u^\circ$ as in the proof of Theorem~\ref{th:main-quasi}. Additional care is needed here since all the truncations need to be  smooth.
We start with the construction of suitable smooth truncations.
Given $R\ge1$, for each $\xi\in\Z^M\setminus\{0\}$, recalling the definition~\eqref{eq:redef-RRxi} of $\Rc_R(\xi)$ in the proof of Lemma~\ref{lem:dioph}, we choose a cut-off function $\chi_R^\xi$ with $\chi_R^\xi=0$ in $\Rc_{2R}(\xi)$, $\chi_R^\xi=1$ outside $\Rc_{R}(\xi)$, and $\|\nabla^j\chi_R^\xi\|_{\Ld^\infty}\lesssim_j (R|\xi|^{s_0+1})^j$ for all $0\le j\le m$. We also choose a cut-off function $\chi_R$ with $\chi_R=1$ in $B_{R^{1/d}}$, $\chi_R=0$ outside $B_{2R^{1/d}}$, and $\|\nabla^j\chi_R\|_{\Ld^\infty}\lesssim_j1$. We then define the product cut-off function
\begin{align}\label{eq:def-zeta-cut}
\zeta_R\,:=\,\chi_R\prod_{\xi\in\Z^M\setminus\{0\}\atop|\xi|\le K\ell}\chi_R^\xi.
\end{align}
By construction, for all $0\le j\le m$, we have
\begin{align}\label{eq:bound-zetaR}
\|\nabla^j\zeta_R\|_{\Ld^\infty}\,\lesssim\,(RK^{s_0+M+1}\ell^{s_0+M+1})^j,
\end{align}
and
\[\mathds1_{B_{R^{1/d}}}\mathds1_{\R^d\setminus{\Rc_{R}^{K\ell}}}\,\le\,\zeta_R\,\le\,\mathds1_{B_{2R^{1/d}}}\mathds1_{\R^d\setminus{\Rc_{2R}^{K\ell}}}.\]
We now define the truncated initial data $u_R^\circ:=\F^{-1}[\zeta_R\,\hat u^\circ]$, the truncated Schrödinger flow $u_{K,R,\lambda}$ as the solution of the corresponding equation~\eqref{eq:def-uKRlam}, and we let $u_{R,0}$ denote the corresponding truncation of the free flow $u_0$.
We further define the truncated approximate flow $U_{K,R,\lambda}^\ell$ as in~\eqref{eq:def-approx-flow-UKR}.

With these definitions at hand, we now turn to the proof of Corollary~\ref{th:main-quasi-tsp}. This is about repeating the proof of Proposition~\ref{prop:ball-tsp-weak} with suitable truncation arguments.
The starting point is the triangle inequality in the form
\begin{multline*} 
\big|M_m^t(u_\lambda)- M_m^t(u_0)\big|\,\le\,\big|M_m^t(u_\lambda)-\tilde M_m^t(u_\lambda;C_0)\big|+\big|\tilde M_m^t(u_\lambda;C_0)-\tilde M_m^t(u_{K,R,\lambda};C_0)\big|\\
+\big|\tilde M_m^t(u_{K,R,\lambda};C_0)-\tilde M_m^t(U_{K,R,\lambda}^\ell;C_0)\big|
+\big|M_m^t(U_{K,R,\lambda}^\ell)-\tilde M_m^t(U_{K,R,\lambda}^\ell;C_0)\big|\\
+\big|M_m^t(U_{K,R,\lambda}^\ell)- M_m^t(u_{R,0})\big|+\big|M_m^t(u_0)- M_m^t(u_{R,0})\big|.
\end{multline*}
Using Corollary~\ref{cor:moment-trunc} to estimate the first RHS term and using~\eqref{eq:cruc-ball-1} to estimate the fourth one, this takes the form
\begin{multline}\label{eq:decomp-Mmt-ball}
\big|M_m^t(u_\lambda)- M_m^t(u_0)\big|\,\lesssim_{m,u^\circ}\,C_0^{-1}+C_0^m\|u_\lambda^t-u^t_{K,R,\lambda}\|_{\Ld^2}+C_0^m\|u_{K,R,\lambda}^t-U_{K,R,\lambda}^{\ell;t}\|_{\Ld^2}\\
+C_0^{-1}M_{m+1}^t(U_{K,R,\lambda}^\ell)+\big|M_m^t(U_{K,R,\lambda}^\ell)- M_m^t(u_{R,0})\big|
+\big|M_m^t(u_0)- M_m^t(u_{R,0})\big|.
\end{multline}
We separately estimate the last five RHS terms and we start with the first one, which is a truncation error.
Arguing as for~\eqref{eq:approx-1-qp} (now with smooth truncations), we find
\begin{align*}
\|u^t_\lambda -u_{K,R,\lambda}^t\|_{\Ld^2} \,\lesssim_{u^\circ}\, R^{-\frac1{2d}}+\lambda t e^{-K^\alpha}.
\end{align*}
Applying Proposition~\ref{prop:taylorbloch} to estimate the third RHS term in~\eqref{eq:decomp-Mmt-ball}, we obtain for all $\lambda\le\frac12(CRK^{s_0+M}\ell^{s_0})^{-1}$,
\begin{align*}
\|u_{K,R,\lambda}^t-U_{K,R,\lambda}^{\ell;t}\|_{\Ld^2}
\,\lesssim_{u^\circ}\,\lambda CR K^{s_0+2M+1}\ell^{s_0+M+1}
+\lambda^{\ell+1}t(K\ell)^{M+1}(CRK^{s_0+M}\ell^{s_0})^\ell.
\end{align*}
Using~\eqref{eq:cruc-ball-2} to estimate the fourth RHS term in~\eqref{eq:decomp-Mmt-ball} yields
\begin{align*}
M_{m+1}^t(U_{K,R,\lambda}^\ell)\,\lesssim_{m,u^\circ}\,1+\sum_{j=1}^{m+1}\sum_{l=1}^j\Big(\int_{\R^d}|\nabla_k^l\kappa_{K,k,\lambda}^\ell|^\frac{2j}l|(\tfrac1t\nabla)^{m+1-j}\hat u_{R,0}^t(k)|^2d^*k\Big)^\frac12,
\end{align*}
and hence, by the definition~\eqref{e.def:TBW} of $\kappa_{K,k,\lambda}^\ell$ and the bounds of Proposition~\ref{prop:cor-QP}, we deduce for all $\lambda\le\frac12(CRK^{s_0+M}\ell^{s_0})^{-1}$,
\begin{align*}
M_{m+1}^t(U_{K,R,\lambda}^\ell)\,\lesssim_{m,u^\circ}\,1+\lambda(RK^{s_0+1}\ell^{s_0+1})^{m+1}\|\langle\tfrac1t\nabla\rangle^{m+1}\hat u_{R,0}^t\|_{\Ld^2}.
\end{align*}
Since Lemma~\ref{prop:moments} and~\eqref{eq:bound-zetaR} yield
\begin{multline*}
\|\langle\tfrac1t\nabla\rangle^{m+1}\hat u_{R,0}^t\|_{\Ld^2}\,\lesssim_m\,\|\langle\nabla\rangle^{m+1}(\zeta_R\hat u^\circ)\|_{\Ld^2}+\|\langle\cdot\rangle^{m+1}\zeta_R\hat u^\circ\|_{\Ld^2}\\
\,\lesssim_{m,u^\circ}\,(RK^{s_0+M+1}\ell^{s_0+M+1})^{m+1},
\end{multline*}
the above turns into
\begin{align*}
M_{m+1}^t(U_{K,R,\lambda}^\ell)\,\lesssim_{m,u^\circ}\,1+\lambda(R^2K^{2(s_0+1)+M}\ell^{2(s_0+1)+M})^{m+1}.
\end{align*}
Likewise, combining~\eqref{eq:cruc-ball-3} with the bounds of Proposition~\ref{prop:cor-QP}, we may estimate the fifth RHS term in~\eqref{eq:decomp-Mmt-ball} as follows: for all $\lambda\le\frac12(CRK^{s_0+M}\ell^{s_0})^{-1}$,
\begin{multline*}
\big|M_m^t(U_{K,R,\lambda}^\ell)- M_m^t(u_{R,0})\big|\,\lesssim\,\sum_{j=1}^m\sum_{l=1}^j\Big(\int_{\R^d}|\nabla_k^l\kappa_{k,\lambda}^\ell|^\frac{2j}l|(\tfrac1t\nabla)^{m-j}\hat u_{R,0}^t(k)|^2d^*k\Big)^\frac12\\
\,\lesssim_m\,\lambda(RK^{s_0+1}\ell^{s_0+1})^{m}\|\langle\tfrac1t\nabla\rangle^{m}\hat u_{R,0}^t\|_{\Ld^2}\,\lesssim_{m,u^\circ}\,\lambda(R^2K^{2(s_0+1)+M}\ell^{2(s_0+1)+M})^{m}.
\end{multline*}
It remains to estimate the last RHS term in~\eqref{eq:decomp-Mmt-ball}, which is a truncation error.
We write
\begin{multline*}
\big|\nabla^m(\hat u_{0}^t-\hat u_{R,0}^t)(k)\big|\,=\,\big|\nabla_k^m\big(e^{-it|k|^2}(1-\zeta_R(k))\,\hat u^\circ(k)\big)\big|\\
\,\lesssim\,\sum_{j=0}^m(t^{j}|k|^j+t^{\frac j2})\big|\nabla^{m-j}\big((1-\zeta_R(k))\hat u^\circ(k)\big)\big|,
\end{multline*}
so that for $t\ge1$,
\begin{align*}
t^{-m}\|\nabla^m(\hat u_{0}^t-\hat u_{R,0}^t)\|_{\Ld^2}\,\lesssim_{m,u^\circ}\,\|\langle\cdot\rangle^m(1-\zeta_R)\,\hat u^\circ\|_{\Ld^2}+\sum_{j=0}^{m-1}t^{j-m}\|\zeta_R\|_{W^{m-j,\infty}}.
\end{align*}
Applying~\eqref{eq:approx-10-qp} in the form $\|\langle\cdot\rangle^m(1-\zeta_R)\,\hat u^\circ\|_{\Ld^2}\,\lesssim_{m,u^\circ}\,R^{-\frac1{2d}}$, and using~\eqref{eq:bound-zetaR}, we deduce for all $t\ge RK^{s_0+M+1}\ell^{s_0+M+1}$,
\begin{multline*}
t^{-m}\|\nabla^m(\hat u_{0}^t-\hat u_{R,0}^t)\|_{\Ld^2}\,\lesssim_{m,u^\circ}\,R^{-\frac1{2d}}+\sum_{j=0}^{m-1}(t^{-1}RK^{s_0+M+1}\ell^{s_0+M+1})^{m-j}\\
\,\lesssim_m\,R^{-\frac1{2d}}+t^{-1}RK^{s_0+M+1}\ell^{s_0+M+1}.
\end{multline*}
Injecting all the above estimates into~\eqref{eq:decomp-Mmt-ball} and optimizing wrt $C_0\ge1$, we find for all $\lambda\le\frac12(CRK^{s_0+M}\ell^{s_0})^{-1}$,
\begin{multline*}
\big|M_m^t(u_\lambda)- M_m^t(u_0)\big|\\
\,\lesssim_{m,u^\circ}\big(R^{-\frac1{2d}}+\lambda t e^{-K^\alpha}+\lambda CR K^{s_0+2M+1}\ell^{s_0+M+1}+\lambda^{\ell+1}t(CRK^{s_0+2M+1}\ell^{s_0+M+1})^\ell\big)^\frac1{m+1}\\
+\lambda(R^2K^{2(s_0+1)+M}\ell^{2(s_0+1)+M})^{m+1}+t^{-1}RK^{s_0+M+1}\ell^{s_0+M+1}.
\end{multline*}
For $\gamma<\frac1{4d(m+1)+1}$, choosing
\begin{eqnarray*}
R\,=\,R(\lambda)&:=&\lambda^{-2d\gamma},\\
K\,=\,K(\lambda)&:=&\ell(\lambda)^\frac1\alpha,\\
\ell\,=\,\ell(\lambda)&:=&\lambda^{-\frac{\alpha(1-(4d(m+1)+1)\gamma)}{(m+1)(\alpha+1)(2(s_0+1)+M)}},
\end{eqnarray*}
the conclusion follows in the regime $\lambda^{-\frac1{m+1}}\le t\le e^{\ell(\lambda)}$ after straightforward simplifications. For shorter timescales, the conclusion is easier.
\qed

\subsection{Proof of Corollary~\ref{th:main-quasi-tsp-2}}\label{sec:pr-cor2}

The proof is split into two steps. First, we reformulate the argument for Propositions~\ref{prop:taylorbloch} and~\ref{prop:ball-tsp-weak}, now at high frequencies rather than small disorder. Second, we post-process the result to treat general initial data, based on the orthogonality between the well-described non-resonant part and the resonant remainder. The latter step requires some new technical care.

\medskip
\step1 Effective approximation result for the flow.

Let $u^\circ_\lambda\in\Sc(\R^d)$ with $\hat u^\circ_\lambda$ supported in $\R^d\setminus B_{\lambda^{-1}}$.
The idea is to consider fattened resonant sets $\Rc_{\lambda R}$ with fattening parameter $R^{-1}\lambda^{-1}$ chosen as a small multiple $R^{-1}$ of the large frequency $\lambda^{-1}$.
We first argue under the simplifying assumption that the potential $V$ satisfies $K:=\sup\{1\vee|\xi|:\xi\in\supp\F\tilde V\}<\infty$.
Proposition~\ref{prop:cor-QP} then yields for all $n\ge1$, $k\in\R^d\setminus\Rc_{\lambda R}^{Kn}$, and $s,j\ge0$,
\begin{eqnarray}\label{eq:bound-nu-freq}
|\nabla_k^j\nu_{k}^n|&\le&(\lambda RCjK^{s_0+1}n^{s_0+1})^j(\lambda RCK^{s_0+M}n^{s_0})^{n},\\
\|\nabla_k^j \phi_k^n\|_{H^s(\Omega)}&\le& (CKn)^{s}(\lambda RCjK^{s_0+1}n^{s_0+1})^j(\lambda RCK^{s_0+M}n^{s_0})^{n},\nonumber
\end{eqnarray}
and we set $\kappa_k^\ell:=\sum_{n=0}^\ell\nu_k^n$. (Here, we allow constants $C$'s to further depend on $\|\F \tilde V\|_{\Ld^\infty}$.)
The setting is similar to the small disorder regime, as the inverse frequency $\lambda\ll1$ now plays essentially the same role as the small disorder intensity.
We may thus repeat the proof of Proposition~\ref{prop:taylorbloch}: if $u_\lambda^\circ$ is supported in $\R^d\setminus(B_{\lambda^{-1}}\cup\Rc_{\lambda R}^{K\ell})$ and if $V$ satisfies $K:=\sup\{1\vee|\xi|:\xi\in\supp\F\tilde V\}<\infty$, denoting by $u_\lambda$ the Schrödinger flow
\[i\partial_tu_\lambda=(-\triangle+V)u_\lambda,\qquad u_\lambda|_{t=0}=u_\lambda^\circ,\]
and considering the approximate flow
\[U_\lambda^{\ell;t}(x):=\int_{\R^d}e^{-it(|k|^2+\kappa_k^\ell)}e^{ik\cdot x}\hat u^\circ_\lambda(k)\,d^*k,\]
there holds for all $\lambda\le\frac12\big(CRK^{s_0+M}\ell^{s_0}\big)^{-1}$ and $T\ge0$,
\begin{multline}\label{eq:est-base-freqhigh}
\sup_\Omega\sup_{0\le\lambda'\le\lambda}\sup_{0\le t\le T}\|u_{\lambda'}^t-U_{\lambda'}^{\ell;t}\|_{\Ld^2}\,\le\,\lambda CRK^{s_0+2M+1}\ell^{s_0+M+1}\|\F\tilde V\|_{\Ld^\infty}\|u^\circ_\lambda\|_{\Ld^2}\\
+\lambda^{\ell}T(K\ell)^{M+1}(CRK^{s_0+M}\ell^{s_0})^\ell\|\F\tilde V\|_{\Ld^\infty}^{\ell+1}\|u^\circ_\lambda\|_{\Ld^2}.
\end{multline}
Next, as in the proof of Theorem~\ref{th:main-quasi}, we can proceed to a truncation argument to remove the compact support assumption for $\tilde V$, and the same result~\eqref{eq:est-base-freqhigh} then holds for all $K\ge C(\ell\log\ell)^\frac1\alpha$ with an additional error $Te^{-K^\alpha}\|u^\circ_\lambda\|_{\Ld^2}$.

\begin{rem}
As in the proof of Theorem~\ref{th:main-quasi}, we may further try to remove the assumption that $\hat u_\lambda^\circ$ is supported in the non-resonant set. Assuming that $\hat u^\circ_\lambda$ is supported in the high frequency annulus $B_{2\lambda^{-1}}\setminus B_{\lambda^{-1}}$, we define for $R\ge1$,
\[u_{\lambda,R}^\circ:=\F^{-1}\big[\mathds1_{\R^d\setminus{\Rc_{\lambda R}}}\hat u_\lambda^\circ\big],\]
and the truncation error is then estimated as follows, in view of Lemma~\ref{lem:dioph}(ii),
\[\|u^\circ_\lambda-u^\circ_{\lambda,R}\|_{\Ld^2}\,\le\,\big|B_{2\lambda^{-1}}\cap\Rc_{\lambda R}\big|^\frac12\|\hat u^\circ_\lambda\|_{\Ld^\infty}\,\lesssim\, \lambda^{-\frac d2}R^{-\frac12}\|\hat u^\circ_\lambda\|_{\Ld^\infty}.\]
If we wish to conclude as in Theorem~\ref{th:main-quasi}, we then need to assume that $u_\lambda^\circ$ satisfies $\|u_\lambda^\circ\|_{\Ld^2}\le1$ and $\lambda^{-\frac d2}\|\hat u^\circ_\lambda\|_{\Ld^\infty}\le1$, which is for instance satisfied for rescaled initial data of the form $u_\lambda^\circ:=\lambda^{-\frac d2}u^\circ(\lambda^{-1}\cdot)$.
In that setting, we may directly adapt the proof of Corollary~\ref{th:main-quasi-tsp}.
\end{rem}

\medskip
\step2 Conclusion.

Let $u^\circ_\lambda\in\Sc(\R^d)$ with $\hat u^\circ_\lambda$ supported in $\R^d\setminus B_{\lambda^{-1}}$.
Decompose $u_\lambda^\circ:=u_{\lambda,1}^\circ+u_{\lambda,2}^\circ$ in terms of the non-resonant part $\hat u_{\lambda,1}^\circ:=\hat u_{\lambda}^\circ\mathds1_{\R^d\setminus\Rc_{\lambda R}^{K\ell}}$ and the resonant remainder $\hat u_{\lambda,2}^\circ:=\hat u_{\lambda}^\circ\mathds1_{\Rc_{\lambda R}^{K\ell}}$.
More precisely, as $\hat u_\lambda^\circ$ may concentrate on the boundary of $\Rc_{\lambda R}^{K\ell}$, we appeal to an averaging method à la De Giorgi~\cite{DeGiorgi-75}, which consists in averaging estimates for a sequence of increasing neighborhoods of the resonant set. Given $N\ge1$, for all $1\le p\le N$, consider
\[\Rc_{\lambda R,N}^p:=\Rc_{\lambda R(1+\frac{p-1}N)}^{K\ell},\]
decompose $u_\lambda^\circ:=u_{\lambda,1}^{p;\circ}+u_{\lambda,2}^{p;\circ}$ in terms of
\[\hat u_{\lambda,1}^{p;\circ}:=\hat u_{\lambda}^\circ\mathds1_{\R^d\setminus \Rc_{\lambda R,N}^p},\qquad
\hat u_{\lambda,2}^{p;\circ}:=\hat u_{\lambda}^\circ\mathds1_{\Rc_{\lambda R,N}^p},\]
and consider the corresponding evolutions $u_\lambda=u_{\lambda,1}^p+u_{\lambda,2}^p$. In view of~\eqref{eq:est-base-freqhigh} in Step~1, we know that $u_{\lambda,1}^p$ remains close to the approximate flow
\[U_{\lambda,1}^{\ell,p;t}(x):=\int_{\R^d\setminus\Rc_{\lambda R,N}^p}e^{-it(|k|^2+\kappa_k^\ell)}e^{ik\cdot x}\hat u_{\lambda}^\circ(k)\,d^*k.\]
As calculating moments requires smoothness in Fourier space, we must replace the integral over the non-resonant set by a smooth cut-off. 
For each $1\le p\le N$, similarly as in~\eqref{eq:def-zeta-cut}, we can construct a cut-off function $\zeta_{\lambda R,N}^p$ such that
\begin{eqnarray*}
&&\zeta_{\lambda R,N}^p=0\qquad\text{in~$\Rc_{\lambda R,N}^{p+1}$,}\\
&&\zeta_{\lambda R,N}^p=1\qquad\text{outside~$\Rc_{\lambda R,N}^p$,}
\end{eqnarray*}
and $\|\nabla^j\zeta_{\lambda R,N}^p\|_{\Ld^\infty}\lesssim_j (N\lambda RK^{s_0+M+1}\ell^{s_0+M+1})^j$,
and we define
\[\tilde U_{\lambda,1}^{\ell,p;t}(x):=\int_{\R^d}e^{-it(|k|^2+\kappa_k^\ell)}e^{ik\cdot x}\hat u_{\lambda}^\circ(k)\,\zeta_{\lambda R,N}^p(k)\,d^*k.\]
With the notation~\eqref{eq:trunc-mom} for ballistically truncated moments, we then decompose
\begin{multline}\label{eq:lower-bound-Mm}
M_1^t(u_\lambda)\ge\tilde M_1^t(u_\lambda)\ge\tilde M_1^t(\tilde U_{\lambda,1}^{\ell,p}+u_{\lambda,2}^p)\\
-CC_0\|u_{\lambda,1}^{p;t}-U_{\lambda,1}^{\ell,p;t}\|_{\Ld^2}-CC_0\|U_{\lambda,1}^{\ell,p;t}-\tilde U_{\lambda,1}^{\ell,p;t}\|_{\Ld^2}.
\end{multline}
We turn to the first RHS term.
Expanding the square
\[\tilde M_1^t(\tilde U_{\lambda,1}^{\ell,p}+u_{\lambda,2}^p)^2\,\ge\,\tilde M_1^t(\tilde U_{\lambda,1}^{\ell,p})^2+2\Re\int_{\R^d}\big(\tfrac{|x|}t\big)^2e^{-(\frac{|x|}{C_0t})^2}\,\tilde U_{\lambda,1}^{\ell,p;t}(x)\,\overline{u_{\lambda,2}^t}(x)dx,\]
decomposing
\begin{multline}\label{eq:decomp-triangle-four}
\big(\tfrac{|x|}{t}\big)^2\tilde U_{\lambda,1}^{\ell,p;t}(x)\,=\,-t^{-2}\int_{\R^d}e^{ik\cdot x}\triangle\Big(e^{-it(|k|^2+\kappa_k^\ell)}\hat u_{\lambda}^\circ(k)\,\zeta_{\lambda R,N}^p(k)\Big)\,d^*k\\
\,=\,\tilde V_{\lambda}^{\ell,p;t}(x)+\tilde W_{\lambda}^{\ell,p;t}(x),
\end{multline}
in terms of
\begin{eqnarray*}
\tilde V_{\lambda}^{\ell,p;t}(x)&:=&\int_{\R^d}e^{-it(|k|^2+\kappa_k^\ell)}e^{ik\cdot x}|2k+\nabla\kappa_k^\ell|^2\,\hat u_{\lambda}^\circ(k)\,\zeta_{\lambda R,N}^p(k)\,d^*k,\\
\tilde W_{\lambda}^{\ell,p;t}(x)&:=&it^{-1}\int_{\R^d}e^{-it(|k|^2+\kappa_k^\ell)}e^{ik\cdot x}(2d+\triangle\kappa_k^\ell)\,\hat u_{\lambda}^\circ(k)\,\zeta_{\lambda R,N}^p(k)\,d^*k\\
&&+it^{-1}\int_{\R^d} e^{-it(|k|^2+\kappa_k^\ell)}e^{ik\cdot x}(2k+\nabla\kappa_k^\ell)\cdot\nabla(\hat u_{\lambda}^\circ\zeta_{\lambda R,N}^p)(k)\,d^*k\\
&&-t^{-2}\int_{\R^d}e^{-it(|k|^2+\kappa_k^\ell)}e^{ik\cdot x}\triangle(\hat u_{\lambda}^\circ\zeta_{\lambda R,N}^p)(k)\,d^*k,
\end{eqnarray*}
and also defining
\[V_{\lambda}^{\ell,p;t}(x)\,:=\,\int_{\R^d\setminus\Rc_{\lambda R,N}^p}e^{-it(|k|^2+\kappa_k^\ell)}e^{ik\cdot x}|2k+\nabla\kappa_k^\ell|^2 \hat u_{\lambda}^\circ(k)\,d^*k,\]
we deduce using unitarity in form of $\|u_{\lambda,2}^{p;t}\|_{\Ld^2} = \| u_{\lambda,2}^{p;\circ}\|_{\Ld^2}\le \|u_{\lambda}^\circ\|_{\Ld^2}$
\begin{multline*}
\tilde M_1^t(\tilde U_{\lambda,1}^{\ell,p}+u_{\lambda,2}^p)^2\ge\tilde M_1^t(\tilde U_{\lambda,1}^{\ell,p})^2+2\Re\int_{\R^d}e^{-(\frac{|x|}{C_0t})^2}\,V_{\lambda}^{\ell,p;t}(x)\,\overline{u_{\lambda,2}^{p;t}}(x)dx\\
-C\Big(\|\tilde W_{\lambda}^{\ell,p;t}\|_{\Ld^2}+\|V_{\lambda}^{\ell,p;t}-\tilde V_\lambda^{\ell,p;t}\|_{\Ld^2}\Big)\|u_{\lambda}^\circ\|_{\Ld^2}.
\end{multline*}
Since $V_\lambda^{\ell;t}$ and $u_{\lambda,2}^t$ have disjoint supports in the Floquet-Bloch fibration (cf.~\eqref{eq:flow-fibration}),
we have $\int_{\R^d}V_{\lambda}^{\ell;t}\,\overline{u_{\lambda,2}^t}=0$, and
\begin{multline*}
\Big|\int_{\R^d}e^{-(\frac{|x|}{C_0t})^2}V_{\lambda}^{\ell,p;t}(x)\,\overline{u_{\lambda,2}^{p;t}}(x)dx\Big|\,=\,\Big|\int_{\R^d}\big(1-e^{-(\frac{|x|}{C_0t})^2}\big)\,V_{\lambda}^{\ell,p;t}(x)\,\overline{u_{\lambda,2}^{p;t}}(x)dx\Big|\\
\,\lesssim\,C_0^{-1}\|(\tfrac{|\cdot|}{t})\tilde V_\lambda^{\ell,p;t}\|_{\Ld^2}\|u_\lambda^\circ\|_{\Ld^2}+\|V_\lambda^{\ell,p;t}-\tilde V_\lambda^{\ell,p;t}\|_{\Ld^2}\|u_\lambda^\circ\|_{\Ld^2}.
\end{multline*}
Injecting the above into~\eqref{eq:lower-bound-Mm}, we are led to
\begin{multline}\label{eq:main-decomp}
M_1^t(u_\lambda)^2\ge \tfrac12\tilde M_1^t(\tilde U_{\lambda,1}^{\ell,p})^2-CC_0^2\|u_{\lambda,1}^{p;t}-U_{\lambda,1}^{\ell,p;t}\|_{\Ld^2}^2-CC_0^2\|U_{\lambda,1}^{\ell,p;t}-\tilde U_{\lambda,1}^{\ell,p;t}\|_{\Ld^2}^2\\
-C\Big(\|\tilde W_{\lambda}^{\ell,p;t}\|_{\Ld^2}+C_0^{-1}\|(\tfrac{|\cdot|}{t})\tilde V_\lambda^{\ell,p;t}\|_{\Ld^2}+\|V_{\lambda}^{\ell,p;t}-\tilde V_\lambda^{\ell,p;t}\|_{\Ld^2}\Big)\|u_{\lambda}^\circ\|_{\Ld^2}.
\end{multline}
It remains to examine the first RHS term. As in~\eqref{eq:cruc-ball-1}, we can write
\[\tilde M_1^t(\tilde U_{\lambda,1}^{\ell,p})\,\ge\, M_1^t(\tilde U_{\lambda,1}^{\ell,p})-CC_0^{-1}\|(\tfrac{|\cdot|}t)^2\tilde U_{\lambda,1}^{\ell,p}\|_{\Ld^2},\]
hence, decomposing
\[\tfrac{x}{t}\tilde U_{\lambda,1}^{\ell,p;t}(x)\,=\,-2\int_{\R^d}e^{-it(|k|^2+\kappa_k^\ell)}e^{ik\cdot x}k\hat u_{\lambda}^\circ(k)\,\zeta_{\lambda R,N}^p(k)\,d^*k+\tilde Z_\lambda^{\ell,p;t}(x),\]
in terms of
\begin{multline*}
\tilde Z_\lambda^{\ell,p;t}(x)\,:=\,-\int_{\R^d}e^{-it(|k|^2+\kappa_k^\ell)}e^{ik\cdot x}(\nabla\kappa_k^\ell)\,\hat u_{\lambda}^\circ(k)\,\zeta_{\lambda R,N}^p(k)\,d^*k\\
-it^{-1}\int_{\R^d}e^{ik\cdot x}e^{-it(|k|^2+\kappa_k^\ell)}\nabla(\hat u_{\lambda}^\circ\zeta_{\lambda R,N}^p)(k)\,d^*k,
\end{multline*}
and noting that $\zeta_{\lambda R,N}^p=1$ outside $\Rc_{\lambda R}^{K\ell}\subset\Rc_{\lambda R}$, we find
\begin{align*}
\tilde M_1^t(\tilde U_{\lambda,1}^{\ell,p})\,\ge\, 2\Big(\int_{\R^d\setminus\Rc_{\lambda R}}|k|^2|\hat u_\lambda^\circ(k)|^2d^*k\Big)^\frac12-\|\tilde Z_\lambda^{\ell,p;t}\|_{\Ld^2}-CC_0^{-1}\|(\tfrac{|\cdot|}t)^2\tilde U_{\lambda,1}^{\ell,p;t}\|_{\Ld^2}.
\end{align*}
Further decomposing $\tilde U_{\lambda,1}^{\ell,p}$ as in~\eqref{eq:decomp-triangle-four}, and injecting this into~\eqref{eq:lower-bound-Mm}, we obtain
\begin{multline}\label{eq:main-decomp+}
M_1^t(u_\lambda)^2\ge\Big(\int_{\R^d\setminus\Rc_{\lambda R}}|k|^2|\hat u_\lambda^\circ(k)|^2d^*k\Big)-CC_0^2\|u_{\lambda,1}^{p;t}-U_{\lambda,1}^{\ell,p;t}\|_{\Ld^2}^2-\|\tilde Z_\lambda^{\ell,p;t}\|_{\Ld^2}^2\\
-CC_0^{-2}\big(\|\tilde V_{\lambda,1}^{\ell,p;t}\|_{\Ld^2}^2+\|\tilde W_{\lambda,1}^{\ell,p;t}\|_{\Ld^2}^2\big)-C\big(\|\tilde W_{\lambda}^{\ell,p;t}\|_{\Ld^2}+C_0^{-1}\|(\tfrac{|\cdot|}{t})\tilde V_\lambda^{\ell,p;t}\|_{\Ld^2}\big)\|u_{\lambda}^\circ\|_{\Ld^2}\\
-CC_0^2\|U_{\lambda,1}^{\ell,p;t}-\tilde U_{\lambda,1}^{\ell,p;t}\|_{\Ld^2}^2
-C\|V_{\lambda}^{\ell,p;t}-\tilde V_\lambda^{\ell,p;t}\|_{\Ld^2}\|u_{\lambda}^\circ\|_{\Ld^2}.
\end{multline}
It remains to analyze each of the RHS error terms. The second RHS term is estimated by~\eqref{eq:est-base-freqhigh} in Step~1:
for $\lambda\le\frac12\big(CRK^{s_0+M}\ell^{s_0}\big)^{-1}$ and $K\ge C(\ell\log\ell)^\frac1\alpha$,
\begin{multline*}
\|u_{\lambda,1}^{p;t}-U_{\lambda,1}^{\ell,p;t}\|_{\Ld^2}\\
\,\le\,\Big(\lambda RCK^{s_0+2M+1}\ell^{s_0+M+1}
+t(K\ell)^{M+1}(\lambda RCK^{s_0+M}\ell^{s_0})^\ell+te^{-K^\alpha}\Big)\|u^\circ_\lambda\|_{\Ld^2}.
\end{multline*}
The third RHS term in~\eqref{eq:main-decomp+} can be directly estimated using~\eqref{eq:bound-nu-freq} and the definition of~$\zeta_{\lambda R,N}^p$: for $\lambda\le\frac12\big(CRK^{s_0+M+1}\ell^{s_0+M+1}\big)^{-1}$,
\[\|\tilde Z_\lambda^{\ell,p;t}\|_{\Ld^2}\,\lesssim\,(\lambda RK^{s_0+M}\ell^{s_0+1})^2\|u_\lambda^\circ\|_{\Ld^2}+\tfrac{N}{t}\|\langle\cdot\rangle u_\lambda^\circ\|_{\Ld^2}.\]
Similarly, the fourth, fifth, and sixth RHS terms are estimated by
\begin{gather*}
\|\tilde V_{\lambda,1}^{\ell,p;t}\|_{\Ld^2}\,\lesssim\,\|\langle\nabla\rangle^2 u_\lambda^\circ\|_{\Ld^2},\\
\|\tilde W_{\lambda,1}^{\ell,p;t}\|_{\Ld^2}\,\lesssim\,t^{-1}\|u_\lambda^\circ\|_{\Ld^2}+\tfrac{N}{t}\|\langle\cdot\rangle\langle\nabla\rangle u_\lambda^\circ\|_{\Ld^2}+\big(\tfrac{N}{t}\big)^2\|\langle\cdot\rangle^2 u_\lambda^\circ\|_{\Ld^2}.
\end{gather*}
To estimate the seventh RHS term in~\eqref{eq:main-decomp+}, we proceed to a similar computation as in~\eqref{eq:decomp-triangle-four}, and we easily find for $\lambda\le\frac12\big(CRK^{s_0+M+1}\ell^{s_0+M+1}\big)^{-1}$,
\[\|(\tfrac{|\cdot|}{t})\tilde V_\lambda^{\ell,p;t}\|_{\Ld^2}\,\lesssim\,\|\langle\nabla\rangle^3u_\lambda^\circ\|_{\Ld^2}+\tfrac{N}{t}\|\langle\cdot\rangle\langle\nabla\rangle^2u_\lambda^\circ\|_{\Ld^2}.\]
The eighth RHS term in~\eqref{eq:main-decomp+}, when averaging over $p$, is estimated by
\begin{align*}
\frac1N\sum_{p=1}^N\|U_{\lambda,1}^{\ell,p;t}-\tilde U_{\lambda,1}^{\ell,p;t}\|_{\Ld^2}^2\,\le\,\frac1N\sum_{p=1}^N\int_{\Rc_{\lambda R,N}^{p}\setminus\Rc_{\lambda R,N}^{p+1}}|\hat u_\lambda^\circ(k)|^2d^*k
\,\le\,N^{-1}\|u_\lambda^\circ\|_{\Ld^2}^2,
\end{align*}
and similarly, further using~\eqref{eq:bound-nu-freq}, for $\lambda\le\frac12\big(CRK^{s_0+M+1}\ell^{s_0+M+1}\big)^{-1}$, the last RHS term in~\eqref{eq:main-decomp+} is estimated by
\[\frac1N\sum_{p=1}^N\|V_{\lambda}^{\ell,p;t}-\tilde V_\lambda^{\ell,p;t}\|_{\Ld^2}\,\le\,\Big(\frac1N\sum_{p=1}^N\|V_{\lambda}^{\ell,p;t}-\tilde V_\lambda^{\ell,p;t}\|_{\Ld^2}^2\Big)^\frac12\,\le\,N^{-\frac12}\|\langle\nabla\rangle^2u_\lambda^\circ\|_{\Ld^2}.\]
Injecting all the above estimates into~\eqref{eq:main-decomp+}, and using that the initial data satisfy
\[\|u_\lambda^\circ\|_{\Ld^2}\le1,\qquad \|\langle\nabla\rangle^3u_\lambda^\circ\|_{\Ld^2}\le\lambda^{-3},\]
which entails by interpolation (cf. Lin~\cite{Lin-86})
$$
\|\langle\cdot\rangle\langle\nabla\rangle u_\lambda^\circ\|_{\Ld^2}\, \lesssim \,
\|u_\lambda^\circ\|_{\Ld^2}^{1/3}\|\langle\nabla\rangle^3 u_\lambda^\circ\|_{\Ld^2}^{1/3}\|\langle\cdot\rangle^3 u_\lambda^\circ\|_{\Ld^2}^{1/3}
\, \lesssim\,
\lambda^{-1}\|\langle\cdot\rangle^3u_\lambda^\circ\|_{\Ld^2}^{1/3}
$$
and 
$$
\|\langle\cdot\rangle\langle\nabla\rangle^2 u_\lambda^\circ\|_{\Ld^2}
\,\lesssim\, \| \langle\nabla\rangle^3 u_\lambda^\circ\|_{\Ld^2}^{2/3} \|\langle\cdot\rangle^3 u_\lambda^\circ\|_{\Ld^2}^{1/3}
\,\lesssim\, \lambda^{-2}\|\langle\cdot\rangle^3u_\lambda^\circ\|_{\Ld^2}^{1/3},
$$
we conclude for $t,C_0,N\ge1$, $\ell\ge2$, $K\ge C(\ell\log\ell)^\frac1\alpha$, and $\lambda\le\frac12\big(CRK^{s_0+M+1}\ell^{s_0+M+1}\big)^{-1}$,
\begin{multline*}
\lambda^2M_1^t(u_\lambda)^2\ge \lambda^2\Big(\int_{\R^d\setminus\Rc_{\lambda R}}|k|^2|\hat u_\lambda^\circ(k)|^2d^*k\Big)\\
-C(C_0\lambda)^2\Big(N^{-\frac12}+\lambda RCK^{s_0+2M+1}\ell^{s_0+M+1}+t(\lambda RCK^{s_0+2M+1}\ell^{s_0+M+1})^\ell+te^{-K^\alpha}\Big)^2\\
-C(\lambda+C_0^{-1})\tfrac{N}{t}\|\langle\cdot\rangle^3 u_\lambda^\circ\|_{\Ld^2}^\frac13
-C(\lambda^4+C_0^{-4})\big(\tfrac{N}{t}\big)^4\|\langle\cdot\rangle^3 u_\lambda^\circ\|_{\Ld^2}^\frac43\\
-C\big(\lambda^2+N^{-\frac12}+(C_0\lambda)^{-1}+(C_0\lambda)^{-2}\big).
\end{multline*}
For $\gamma<\frac1{2d+1}$, choosing
\begin{eqnarray*}
R\,=\,R(\lambda)&:=&\lambda^{-2d\gamma},\\
K\,=\,K(\lambda)&:=&C\big(\ell(\lambda)\log \ell(\lambda)\big)^\frac1\alpha,\\
\ell\,=\,\ell(\lambda)&:=&\lambda^{-\frac{\alpha(1-(2d+1)\gamma)}{s_0+2M+1+\alpha(s_0+M+1)}},\\
C_0\,=\,C_0(\lambda)&:=&\lambda^{-1-\frac\gamma4},\\
N\,=\,N(\lambda,u_\lambda^\circ,t)&:=&t^{\frac23}\lambda^{-\frac{2+\gamma}3}\|\langle\cdot\rangle^3u_\lambda^\circ\|_{\Ld^2}^{-\frac29},
\end{eqnarray*}
we deduce for $\lambda\ll1$ and $t\ge\|\langle\cdot\rangle^3u_\lambda^\circ\|_{\Ld^2}^\frac13$, with $s:=\frac{s_0+2M+1+\alpha(s_0+M+1)}{\alpha(1-(2d+1)\gamma)}$,
\[\lambda^2M_1^t(u_\lambda)^2\ge \lambda^2\Big(\int_{\R^d\setminus\Rc_{\lambda R}}|k|^2|\hat u_\lambda^\circ(k)|^2d^*k\Big)-C\lambda^\frac\gamma4-C(te^{-\lambda^{-\frac1s}})^2.\]
Setting $G:=\bigcup_{n=0}^\infty A_n\setminus \Rc_{(2^{-n})^{1-2d\gamma}}$, in terms of the dyadic annuli $A_n:=B_{2^n}\setminus B_{2^{n-1}}$ for $n\ge1$ and $A_0:=B$, the conclusion follows.


\section{Classical flow }\label{app:gen}

This section is devoted to the main results for the wave flow.
It is organized as follows.
We start by adapting the approximate stationary Floquet-Bloch theory to the operator $-\nabla\cdot(\Id+\lambda a)\nabla$, defining a notion of approximate Bloch waves via the corresponding Rayleigh-Schr\"odinger perturbation series.
We then turn to the proof of Theorem~\ref{thm:class-wave-qp} and we mainly focus on the differences wrt the proofs of the corresponding results for the Schr\"odinger flow, to which we refer the reader for most of the arguments.
Finally, we study the case of peaked initial data in Fourier space and prove Corollary~\ref{cor:class-wave-qp}.

\subsection{Approximate Bloch waves}\label{sec:7.3}
Consider the operator $-\nabla\cdot(\Id+\lambda a)\nabla$ where $\lambda a$ is a quasiperiodic perturbation in the sense of~(QP$'$). We construct a branch of approximate Bloch waves in terms of the corresponding Rayleigh-Schrödinger series for the perturbed fibered operator $-(\nabla+ik)\cdot (\Id+\lambda a)(\nabla+ik)=-\triangle_k+|k|^2-\lambda (\nabla+ik)\cdot a(\nabla+ik)$.
The corresponding Rayleigh-Schrödinger coefficients $(\nu_{k}^n,\phi_{k}^n)_{k,n}$ (compare to Definition~\ref{def:taylor-waves} in the Schrödinger case) then takes on the following guise:
\begin{enumerate}[$\bullet$]
\item $\nu_k^n:=-ik\cdot\expec{a(\nabla+ik)\phi_k^n}$ for all $n\ge0$;
\item for all $k\in\R^d$, we have $\phi_{k}^0\equiv 1$, and for all $n$ the function $\phi_{k}^{n+1}$ satisfies $\expec{\phi_{k}^{n+1}}=0$ and
\begin{gather}\label{e.def-corr-per+wave}
-\triangle_k \phi_{k}^{n+1}=\,\Pi\,(\nabla+ik)\cdot a(\nabla +ik) \phi_{k}^{n}+\sum_{l=0}^{n-1}\nu_{k}^l \phi_{k}^{n-l}.\nonumber
\end{gather}
\item As before, we set $\kappa_{k,\lambda}^{\ell}=\lambda \sum_{n=0}^\ell\lambda^n \nu_k^n$ for all $k\in O$.
In the statement of Theorem~\ref{thm:class-wave-qp}, we further use the short-hand notation $\tilde \kappa_{k,\lambda}^\ell$ for $\min\{-|k|^2,\kappa_{k,\lambda}^{\ell}\}$ on $O$, extended by zero outside $O$.
\end{enumerate}
We are exactly in the same situation as in the previous sections with the multiplication by $V$ replaced by the operator $-(\nabla+ik)\cdot a(\nabla+ik)$.
A direct adaptation of the proof of Proposition~\ref{prop:cor-QP} yields the following corrector estimates.

\begin{prop}\label{prop:cor-QP-wave}
Consider the quasiperiodic setting~\emph{(QP$'$)}, assume that the winding matrix $F$ satisfies the Diophantine condition~\eqref{eq:classical-Dioph2} with $r_0>0$, let $s_0>M+r_0$, assume that the lifted map $\tilde a$ has compactly supported Fourier transform, and set $K:=\sup\{1\vee|\xi|:\xi \in \supp \F\tilde a\}$.
Then there exists a constant $C$ (depending on $F,M,s_0$) and for all $\ell,R\ge1$ there exists a field of $\ell$-jets of Bloch waves $(\nu_{k}^n,\phi_{k}^n:k\in\R^d\setminus\Rc_R^{K\ell},\,0\le n\le\ell)$ in the above sense, which satisfy
the following estimates for all $n\ge1$, $k\in\R^d\setminus\Rc_R^{Kn}$, and $s,j\ge0$,
\begin{eqnarray}
|\nabla_k^j\nu_{k}^n|&\le&\langle k\rangle^{2n}(CRjK^{s_0+1}n^{s_0+1})^j(CRK^{s_0+M}n^{s_0})^{n}\|\F \tilde a\|_{\Ld^\infty}^{n+1},\label{eq:corr-bound1-wave}\\
\|\nabla_k^j \phi_k^n\|_{H^s(\Omega)}&\le&\langle k\rangle^{2n} (CKn)^{s}(CRjK^{s_0+1}n^{s_0+1})^j(CRK^{s_0+M}n^{s_0})^{n}\|\F \tilde a\|_{\Ld^\infty}^n.\nonumber
\end{eqnarray}
and for all $\hat u\in C^\infty_c(\R^d)$ supported in $\R^d\setminus\Rc_R^{Cn}$,
\begin{align}\label{eq:corr-bound2-wave}
&\sup_{\omega\in\Omega}\bigg(\int_{\R^d}\Big|\int_{\R^d}e^{ik\cdot x}\,\nabla_k^j\nabla^s\phi_k^n(x,\omega)\,\hat u(k)\,d^*k\Big|^2dx\bigg)^\frac12 
\\
&\qquad\,\le\,(CKn)^{s+M+1}(CRjK^{s_0+1}n^{s_0+1})^j(CRK^{s_0+M}n^{s_0})^{n}\|\F\tilde a\|_{\Ld^\infty}^n\,\|\langle\cdot\rangle^{2n}\hat u\|_{\Ld^2}.\qedhere
\end{align}
\end{prop}

For later purposes, we further show that the corrector $\phi_{k}^n$ can be written as $\phi_{k}^n=(\nabla+ik)\cdot \Phi_{k}^n$ for some $\Phi_{k}^n$ that satisfies the same bounds as $\phi_{k}^n$ itself.

\begin{prop}\label{prop:rep-nablaik-corr}
Under the assumptions of Proposition~\ref{prop:cor-QP-wave}, for all $n\ge0$ and $k\in \R^d\setminus\Rc_R^{Kn}$, there exists $\Phi_{k}^n \in \Ld^2(\Omega)^d$
such that $\phi_{k}^n=(\nabla +ik)\cdot \Phi_{k}^n$ and such that $\Phi_{k}^n$ satisfies the same bounds~\eqref{eq:corr-bound1-wave}--\eqref{eq:corr-bound2-wave} multiplied by the factor $\langle k\rangle^{-1}(1+|k|^{-1})^{2+2(n-1)\vee0}$.
\qedhere
\end{prop}

\begin{proof}
As in Proposition~\ref{prop:sol-nonlin-rec}, the correctors can be expressed via the following tree formulas: for all $n\ge1$,
\begin{multline}\label{eq:tree-form-diva-a0}
\phi_{k}^n=\sum_{m=1}^n\sum_{\ell=0}^{n-m}\sum_{c\in\N^\ell\atop|c|=n-m-\ell}\sum_{b\in\N^m\atop|b|=\ell}\nu_{k}^{c_1}\ldots\nu_{k}^{c_\ell}\\
\times(-\triangle_k)^{-b_1-1}\Pi  (\nabla+ik)\cdot a(\nabla+ik)\\
\ldots (-\triangle_k)^{-b_m-1}\Pi  (\nabla+ik)\cdot a(\nabla+ik)1.
\end{multline}
On the one hand, we claim that for all $G\in C^\infty_c(\R^d)^d$ there holds
\begin{equation*}
(-\triangle_k)^{-b-1}(\nabla+ik)\cdot G
=\frac1{(-|k|^2)^{b+1}}(\nabla+ik)\cdot\Big(\Id+(\nabla+ik)(-\triangle_k)^{-1}(\nabla+ik)\cdot\Big)^{b+1}G.
\end{equation*}
Indeed, the relation
\[-\triangle_kf=(\nabla+ik)\cdot G\]
implies $f=(\nabla+ik)\cdot F$ with
\[F=\frac1{-|k|^2}\big(G+ (\nabla+ik)f\big),\]
and the claim follows.
Injecting this identity into~\eqref{eq:tree-form-diva-a0} and decomposing the first projection $\Pi=\Id-\expec{\cdot}$, we deduce $\phi_{k}^n=(\nabla+ik)\cdot\Phi^n_{k}$ where $\Phi_{k}^n$ is given by
\begin{align*}
\Phi_{k}^n&\,=\,\sum_{m=1}^n\sum_{\ell=0}^{n-m}\sum_{c\in\N^\ell\atop|c|=n-m-\ell}\sum_{b\in\N^m\atop|b|=\ell}\frac{\nu_{k}^{c_1}\ldots\nu_{k}^{c_\ell}}{(-|k|^2)^{b_1+1}}\\
&\qquad\times\Big(\Id+(\nabla+ik)(-\triangle_k)^{-1}(\nabla+ik)\cdot\Big)^{b_1+1} a(\nabla+ik)\\
&\qquad\times(-\triangle_k)^{-b_2-1} \Pi (\nabla+ik)\cdot a(\nabla+ik)\\
&\qquad\qquad\qquad\ldots (-\triangle_k)^{-b_m-1}\Pi  (\nabla+ik)\cdot aik .
\end{align*}
We may now proceed to the estimate of $\Phi_{k}^n$ exactly as for $\phi_{k}^n$. The details are left to the reader.
\end{proof}

\subsection{Proof of Theorem~\ref{thm:class-wave-qp}}\label{sec:7.4}
The proof of Theorem~\ref{thm:class-wave-qp} follows the general strategy of Section~\ref{chap:pr-taylorbloch}:
the initial data $(u^\circ,v^\circ)$ are first replaced by an approximate Bloch expansion and the corresponding flow then admits an explicit approximate formula using 
approximate Bloch waves (since the latter approximately diagonalize the wave operator). In order to control the errors, we exploit the corrector estimates of Proposition~\ref{prop:cor-QP-wave}.
We recall that $\Ld^2$-estimates for the wave equation are as follows.

\begin{lem}[$\Ld^2$-estimates for classical waves]\label{prop:moments-wav}
Given $z^\circ,w^\circ_1\in\Ld^2(\R^d)$, $w^\circ_2 \in \Ld^2(\R^d)^d$, $F\in\Ld^1_\loc(\R^+;\Ld^2(\R^d))$, and given a uniformly elliptic matrix field $\Aa$, denote by $z$ the solution of the classical wave equation
\begin{align*}
(\partial_{tt}^2-\nabla\cdot\Aa\nabla)z=F,\qquad z|_{t=0}=z^\circ,\qquad \partial_tz|_{t=0}=w^\circ_1+\nabla \cdot w_2^\circ.
\end{align*}
Then, for all $t\ge0$, we have
\[\|z\|_{\Ld^\infty_t\Ld^2}\lesssim\|z^\circ\|_{\Ld^2}+t\|w_1^\circ\|_{\Ld^2}+\|w_2^\circ\|_{\Ld^2}+\int_0^t\Big\|\int_0^sF\Big\|_{\Ld^2}ds.\qedhere\]
\end{lem}
\noindent
The rest of the proof is similar to that of Theorem~\ref{th:main-quasi} in Section~\ref{chap:pr-taylorbloch} with however three main differences:
\begin{enumerate}[(I)]
\item  In view of Lemma~\ref{prop:moments-wav}, general initial velocities $v^\circ$ in~\eqref{eq:class-a} lead the $\Ld^2$-norm of the flow to grow linearly in time. The initial data in Theorem~\ref{thm:class-wave-qp}  are chosen so that this does not happen, and we need to ensure that the various errors made in the approximation procedure do not grow in time either.
\item As opposed to quantum waves, the approximate solution is not necessarily well-defined here. This requires to make some further assumptions on the initial data, which can a posteriori be dropped by an approximation argument.
\item Comparing the $\Ld^2$-estimates for quantum and classical waves with a nontrivial source term $F$ in the equation, we note that an additional time integral appears in the case of classical waves (cf.~Lemma~\ref{prop:moments-wav}). If treated na\?ively, this would lead to an additional time factor in the error estimates and would reduce the maximal allowed timescale. However, we only have to deal with source terms $F$ displaying a particular structure that indeed ensures that $\int_0^tF$ remains bounded uniformly in time.
\end{enumerate}
We now briefly comment on these three differences.

\subsubsection*{Argument for \emph{(I)}}

Since the only additional difficulty here comes from the initial velocity, we assume $u^\circ\equiv0$, $v^\circ\not\equiv0$, so that~\eqref{eq:class-a} reads
\[\partial_{tt}^2u_\lambda=\nabla\cdot(\Id+\lambda a)\nabla u_\lambda,\qquad u_\lambda|_{t=0}=0,\qquad \partial_tu_\lambda|_{t=0}=v^\circ,\]
where $v^\circ$ has Fourier transform $\hat v^\circ$ compactly supported in $\R^d\setminus\{0\}$. In particular, this entails that $v^\circ=\nabla\cdot g^\circ$ for some $g^\circ\in H^1(\R^d)^d$.
The first step in the proof of Theorem~\ref{thm:class-wave-qp} is to replace $v^\circ$ by its approximate Bloch wave expansion
\begin{align}\label{eq:Bloch-expansion-v0}
Z_{\lambda}^{\ell;\circ}(x):=\int_{\R^d}e^{ik\cdot x}\,\psi_{k,\lambda}^\ell(x)\,\hat v^\circ(k)\,d^*k,
\end{align}
to consider the corresponding flow,
\[\partial_{tt}^2W_{\lambda}^{\ell}=\nabla\cdot(\Id+\lambda a)\nabla W_{\lambda}^{\ell}=0,\qquad W_{\lambda}^{\ell}|_{t=0}=0,\qquad \partial_tW_{\lambda}^{\ell}|_{t=0}=Z_{\lambda}^{\ell;\circ},\]
and to estimate $\|u_\lambda^t-W_{\lambda}^{\ell;t}\|_{\Ld^2}$.
In order to avoid a linear time growth, it is crucial to check that $Z_{\lambda}^{\ell;\circ}$ can be written in divergence form. This indeed follows from Proposition~\ref{prop:rep-nablaik-corr}, which yields  $Z_{\lambda}^{\ell;\circ}=\nabla\cdot g_{\lambda}^{\ell;\circ}$ with
\[g_{\lambda}^{\ell;\circ}:=\int_{\R^d}e^{ik\cdot x}\,\Psi_{k,\lambda}^\ell(x)\,\hat v^\circ(k)\,d^*k,\qquad\Psi_{k,\lambda}^\ell:=\sum_{n=0}^\ell\lambda^n\Phi_{k}^n.\]
By Lemma~\ref{prop:moments-wav}, we then obtain for all $t\ge0$,
\begin{align*}
\|u_\lambda^t-W_{\lambda}^{\ell;t}\|_{\Ld^\infty_t\Ld^2}\,\lesssim\,\|g^\circ-g_{\lambda}^{\ell;\circ}\|_{\Ld^2}
\,\le\,\sum_{n=1}^\ell\lambda^n\,\bigg(\int_{\R^d}\Big|\int_{\R^d}e^{ik\cdot x}\,\Phi_{k}^n(x)\,\hat v^\circ(k)\,d^*k\Big|^2dx\bigg)^\frac12,
\end{align*}
and it remains to use the corrector estimates of Proposition~\ref{prop:rep-nablaik-corr} for the $\Phi_{k}^n$'s to conclude.

\subsubsection*{Argument for \emph{(II)}}
The second step of the proof of Theorem~\ref{thm:class-wave-qp} consists in writing an approximate representation formula for the solution $W_{\lambda}^\ell$ of
\[\partial_{tt}^2W_{\lambda}^{\ell}=\nabla\cdot(\Id+\lambda a)\nabla W_{\lambda}^{\ell},\qquad W_{\lambda}^{\ell}|_{t=0}=W_{\lambda}^{\ell;\circ},\qquad \partial_tW_{\lambda}^{\ell}|_{t=0}=Z_{\lambda}^{\ell;\circ},\]
where $W_{\lambda}^{\ell;\circ}$ and $Z_{\lambda}^{\ell;\circ}$ are the approximate Bloch expansions of $u^\circ$ and $v^\circ$, cf.~\eqref{eq:Bloch-expansion-v0}.
It is natural to set
\[V_{\lambda}^{\ell;t}(x):=\int_{\R^d}\bigg(\cos\Big(t\sqrt{|k|^2+\kappa_{k,\lambda}^{\ell}}\Big)\hat u^\circ(k)+t\sinc\Big(t\sqrt{|k|^2+\kappa_{k,\lambda}^\ell}\Big)\hat v^\circ(k)\bigg)\,e^{ik\cdot x}\psi_{k,\lambda}^{\ell}(x)\,d^*k.\]
This is however a priori not well-defined since the approximate Bloch eigenvalues $|k|^2+\kappa_{k,\lambda}^{\ell}$ may be non-positive. To avoid this possibility, given $\ell,R\ge1$ and $\lambda_0>0$, we must further assume that the Fourier transform of the initial data $(u^\circ,v^\circ)$ is compactly supported in 
\[\Oc_{R,\lambda_0}^{\ell}\,:=\, \{k\in \R^d\setminus\Rc_R^{\ell}: |k|^2-\lambda_0\sum_{n=0}^\ell\lambda_0^n|\nu_k^n|\ge 0\}.\]
Under this assumption, for $0<\lambda\le\lambda_0$, the formula for $V_{\lambda}^{\ell}$ makes perfect sense.
Since $\Oc_{R,\lambda_0}^{\ell}\uparrow\R^d\setminus\Rc_R^{\ell}$ as $\lambda_0\downarrow0$, we deduce that this assumption on $(u^\circ,v^\circ)$ can a posteriori be dropped by an approximation argument as in the proof of Theorem~\ref{th:main-quasi}.

\subsubsection*{Argument for \emph{(III)}}
The error $W_{\lambda}^\ell-V_{\lambda}^\ell$ satisfies
\begin{gather*}
\big(\partial_{tt}^2-\nabla\cdot(\Id+\lambda a)\nabla\big)(W_{\lambda}^\ell-V_{\lambda}^\ell)=F_{\lambda}^\ell,\\
(W_{\lambda}^\ell-V_{\lambda}^\ell)|_{t=0}=0,\qquad\partial_t(W_{\lambda}^\ell-V_{\lambda}^\ell)|_{t=0}=0,
\end{gather*}
in terms of
\begin{eqnarray*}
F_{\lambda}^{\ell;t}(x)&:=&-\int_{\R^d}\bigg(\cos\Big(t\sqrt{|k|^2+\kappa_{k,\lambda}^{\ell}}\Big)\hat u^\circ(k)+t\sinc\Big(t\sqrt{|k|^2+\kappa_{k,\lambda}^{\ell}}\Big)\hat v^\circ(k)\bigg)\,e^{ik\cdot x}\\
&&\hspace{3cm}\times \big(-\triangle_k-\lambda(\nabla+ik)\cdot a(\nabla+ik)-\kappa_{k,\lambda}^\ell\big)\psi_{k,\lambda}^\ell(x)\,d^*k.
\end{eqnarray*}
We then apply Lemma~\ref{prop:moments-wav}
and we must estimate the contribution $\int_0^t\|\int_0^sF^{\ell}_{\lambda}\|_{\Ld^2}ds$ due to the source term. Note that this contribution displays two time integrals in contrast to the case of quantum waves. However, the explicit time integration of $F^\ell_{\lambda}$ yields
\begin{multline*}
\int_0^tF^{\ell;s}_{\lambda}(x)\,ds=-\int_{\R^d}\bigg(t\sinc\Big(t\sqrt{|k|^2+\kappa_{k,\lambda}^{\ell}}\Big)\hat u^\circ(k)-{\textstyle\frac{1-\cos\big(t\sqrt{|k|^2+\kappa_{k,\lambda}^{\ell}}\big)}{{|k|^2+\kappa_{k,\lambda}^\ell}}}\hat v^\circ(k)\bigg)\,e^{ik\cdot x}\\
\times \big(-\triangle_k-\lambda(\nabla+ik)\cdot a(\nabla+ik)-\kappa_{k,\lambda}^\ell\big)\psi_{k,\lambda}^\ell(x)\,d^*k,
\end{multline*}
which can then be estimated as in the case of quantum waves (without loosing any time factor).
A similar argument was used in~\cite[Proof of Proposition~3, Substep~3.2]{BG-16}.

We may then conclude that $W_{\lambda}^\ell$ remains close to $V_{\lambda}^\ell$ in a suitable regime, and therefore close to $U_{\lambda}^\ell$ defined by
\begin{equation}\label{e.approx-sol-class}
U_{\lambda}^{\ell;t}(x):=\int_{\R^d}\bigg(\cos\Big(t\sqrt{|k|^2+\kappa_{k,\lambda}^{\ell}}\Big)\hat u^\circ(k)+t\sinc\Big(t\sqrt{|k|^2+\kappa_{k,\lambda}^{\ell}}\Big)\hat v^\circ(k)\bigg)\,e^{ik\cdot x}\,d^*k,
\end{equation}
and solution of 
\[\pushQED{\qed}(\partial_{tt}^2-\triangle+\kappa_{\lambda,-i\nabla}^\ell)U_{\lambda}^\ell=0,\qquad U_{\lambda}^\ell|_{t=0}=u^\circ,\qquad\partial_tU_{\lambda}^\ell|_{t=0}=v^\circ.\qedhere\popQED\]

\subsection{Proof of Corollary~\ref{cor:class-wave-qp}}\label{sec:7.5} 
Let $k_0\in \Oc\setminus\{0\}$ and consider the case when the initial data in~\eqref{eq:class-a} is a slow modulation of the plane wave $x\mapsto e^{ik_0 \cdot x}$,
\begin{gather*}
\partial_{tt}^2 u_{\lambda,\e}=\nabla \cdot (\Id+\lambda a) \nabla u_{\lambda,\e},\nonumber\\
u_{\lambda,\e}(x)|_{t=0}=\e^{\frac d2} e^{ik_0 \cdot x} u^\circ(\e x), \qquad \partial_t u_{\lambda,\e}(x)|_{t=0}=\e^\frac d2 e^{ik_0 \cdot x}v^\circ(\e x),
\end{gather*}
where $\e$ satisfies a scaling relation $\e:=\lambda^{\beta}$ for some $\beta>0$.
Let $\ell\ge1$ be fixed.
We assume for simplicity that $(\hat u^\circ,\hat v^\circ)$ is compactly supported in the unit ball $B$ (say) and that $\tilde a$ is a trigonometric polynomial with $K:=\sup\{1\vee|\xi|:\xi\in\supp\F\tilde a\}<\infty$.
There exist $R=R(k_0,\ell,K)$ and $\e_0=\e_0(k_0,\ell,K)$ such that $B_\e(k_0)\subset\R^d\setminus(\Rc_R^{K\ell}\cup\{0\})$.
Since $\hat u_{\lambda,\e}(k)|_{t=0}=\e^{-\frac d2}\hat u^\circ(\frac1\e(k-k_0))$ and $\partial_t\hat u_{\lambda,\e}(k)|_{t=0}=\e^{-\frac d2}\hat v^\circ(\frac1\e(k-k_0))$,
the proof of Proposition~\ref{prop:taylorbloch} yields for all $t\ge0$, $\e\le\e_0$, and $\lambda\ll_{k_0,\ell,K}1$,
\begin{align}\label{eq:approx-uU-ellfix-wave}
\|u_{\lambda,\e}^t-U_{\lambda,\e}^{\ell;t}\|_{\Ld^2}
\,\lesssim_{k_0,\ell,K,u^\circ,v^\circ}\,\lambda(1+\lambda^\ell T),
\end{align}
in terms of
\begin{multline*}
U_{\lambda,\e}^{\ell;t}(x):=\e^{-\frac d2}\int_{\R^d}\bigg(\cos\Big(t\sqrt{|k|^2+\kappa_{k,\lambda}^{\ell}}\Big)\hat u^\circ(\tfrac1\e(k-k_0))\\
+t\sinc\Big(t\sqrt{|k|^2+\kappa_{k,\lambda}^{\ell}}\Big)\hat v^\circ(\tfrac1\e(k-k_0))\bigg)\,e^{ik\cdot x}\,d^*k,
\end{multline*}
where the smallness condition on $\lambda$ ensures that $|k|^2+\kappa_{k,\lambda}^\ell\ge 0$ on $B_\e(k_0)$.
We now simplify the formula for $U_{\lambda,\e}^{\ell}$ in the limit $\lambda,\e\downarrow0$.
We split the approximate flow $U_{\lambda,\e}^\ell$ into two contributions $U_{\lambda,\e}^\ell=\frac12(U_{\lambda,\e,+}^{\ell}+U_{\lambda,\e,-}^{\ell})$, where
\[U_{\lambda,\e,\pm}^{\ell;t}(x):=\e^{-\frac d2}\int_{\R^d}e^{\pm it\sqrt{|k|^2+\kappa_{k,\lambda}^\ell}}\textstyle\Big(\hat u^\circ(\tfrac1\e(k-k_0))\mp\frac{i\hat v^\circ(\frac1\e(k-k_0))}{\sqrt{|k|^2+\kappa_{k,\lambda}^\ell}}\Big)\,e^{ik\cdot x}\,d^*k.\]
Changing variables yields
$U_{\lambda,\e,\pm}^{\ell;t}(x)= \e^\frac d2 e^{ik_0\cdot x} R_{\lambda,\e,\pm}^{\ell;t}( \e x)$ with
\[R_{\lambda,\e,\pm}^{\ell;t}(x)\,:=\,\int_{\R^d}e^{\pm it\sqrt{|k_0+\e k|^2+\kappa_{k_0+\e k,\lambda}^\ell}}\textstyle\Big(\hat u^\circ (k)\mp\frac{i \hat v^\circ(k)}{\sqrt{|k_0+\e k|^2+\kappa_{k_0+\e k,\lambda}^\ell}}\Big)\,e^{ik\cdot x}\,d^*k.\]
The boundedness of $k\mapsto\kappa_{k,\lambda}^\ell$ on $\R^d\setminus\Rc_R^{K\ell}$ (cf.~Proposition~\ref{prop:cor-QP-wave}) yields for all $\e\le\e_0$ and $\lambda\ll_{k_0,\ell,K}1$,
\begin{align}\label{eq:approx-V-last}
R_{\lambda,\e,\pm}^{\ell;t}(x)\,=\,\int_{\R^d}e^{\pm it\sqrt{|k_0+\e k|^2+\kappa_{k_0+\e k,\lambda}^\ell}}\Big(\hat u^\circ (k)\mp\tfrac{i \hat v^\circ(k)}{|k_0|}\Big)\,e^{ik\cdot x}\,d^*k+O_{\ell,v^\circ}(\lambda+\e),
\end{align}
where the approximation holds in $\Ld^2(\R^d)$.
It remains to make a Taylor expansion of the square-root appearing in the time exponential.
We start by expanding the approximate Bloch eigenvalue: for $L\ge\ell$ and $\lambda\ll_{k_0,\ell,K}1$, using the boundedness of $k\mapsto\kappa_{k,\lambda}^\ell$ on $\R^d\setminus\Rc_R^{K\ell}$ (cf.~Proposition~\ref{prop:cor-QP-wave}),
\begin{eqnarray*}
\lefteqn{\sqrt{|k|^2+\kappa_{k,\lambda}^\ell}\,=\,|k|\sum_{n=0}^L\frac{(-1)^n(2n)!}{(1-2n)4^n(n!)^2}\Big(\frac{\kappa_{k,\lambda}^\ell}{|k|^2}\Big)^n+O_{k,L}((\kappa_{k,\lambda}^\ell)^{L+1})}\\
&=&\sum_{n=0}^L\lambda^n\frac{(-1)^n(2n)!}{(1-2n)4^n(n!)^2}\Big(\sum_{m=0}^\ell \lambda^m\nu_{k}^m\Big)^n|k|^{1-2n}+O_{k,\ell,L,K}(\lambda^{L+1})\\
&=&\sum_{n=0}^L\sum_{m=0}^{\ell}\lambda^{n+m}\bigg(\frac{(-1)^n(2n)!}{(1-2n)4^n(n!)^2}\sum_{\alpha\in\N^n\atop|\alpha|=m}\nu_{k}^{\alpha_1}\ldots \nu_{k}^{\alpha_n}\bigg)|k|^{1-2n}+O_{k,\ell,L,K}(\lambda ^{\ell+1}).
\end{eqnarray*}
Given $P\ge0$, we then replace $k$ by $k_0+\e k$ and Taylor expand the corresponding symbol up to order $P$, using the smoothness of $k\mapsto\nu_{k}^n$ on $\R^d\setminus\Rc_R^{K\ell}$ for $n\le\ell$ (cf.~Proposition~\ref{prop:cor-QP-wave}),
\begin{multline*}
{\sqrt{|k_0+\e k|^2+\kappa_{k_0+\e k,\lambda}^\ell}}
=\sum_{n=0}^L\sum_{m=0}^{\ell}\sum_{p=0}^P\lambda^{n+m}\e^p|k_0+\e k|^{1-2n}k^{\otimes p}\odot C_{n,m,p}^L(k_0)\\
+O_{k_0,\ell,L,P,K}(\lambda^{\ell+1}+\e^{P+1}),
\end{multline*}
where we have set
\[C_{n,m,p}^L(k_0):=\frac{(-1)^n(2n)!}{(1-2n)4^n(n!)^2}\sum_{\alpha\in\N^n\atop|\alpha|=m}\sum_{\beta\in\N^n\atop|\beta|=p}\frac1{\beta!}\nabla^{\beta_1}\nu_{k_0}^{\alpha_1}\otimes\ldots\otimes\nabla^{\beta_n}\nu_{k_0}^{\alpha_n}\]
(with the convention that the double sum equals $1$ if $n=0$).
Finally, noting that
\begin{eqnarray*}
\lefteqn{|k_0+\e k|^{1-2n}\,=\,\big(|k_0|^2+2\e k\cdot k_0+\e^2|k|^2 \big)^{\frac12-n}}
\\
&=&|k_0|^{1-2n}\sum_{s=0}^\infty(-1)^s\frac{\prod_{j=0}^{s-1}(2n-1+2j)}{2^ss!}\Big(2\e k\cdot \frac{k_0}{|k_0|^2}+\e^2 \frac{|k|^2}{|k_0|^2} \Big)^{s}\\
&=&|k_0|^{1-2n}\sum_{s=0}^\infty(-1)^s\frac{\prod_{j=0}^{s-1}(2n-1+2j)}{2^ss!}  \sum_{l=0}^s\binom{s}{l} \e^{2s-l} |k_0|^{l-2s}|k|^{2(s-l)}\Big(2 k\cdot \frac{k_0}{|k_0|}\Big)^l
\\
&=&\sum_{r=0}^P\e^r|k_0|^{1-2n-r}\sum_{0\le l\le s \atop 2s-l=r}(-1)^s\frac{\prod_{j=0}^{s-1}(2n-1+2j)}{2^{s-l}l!(s-l)!} \Big(\frac{k\cdot k_0}{|k_0|}\Big)^{l}|k|^{r-l}+O_{k_0,P}(\e^{P+1}),
\end{eqnarray*}
and using the boundedness of $k\mapsto\nu_{k}^n$ on $\R^d\setminus\Rc_R^{K\ell}$ for $n\le\ell$ (cf.~Proposition~\ref{prop:cor-QP-wave}), we conclude for all $\lambda,\e\ll_{k_0,\ell,P,K}1$,
\begin{align*}
{\sqrt{|k_0+\e k|^2+\kappa_{k_0+\e k,\lambda}^\ell}}
=\sum_{m=0}^{\ell}\sum_{p=0}^P\lambda^{m}\e^{p}k^{\otimes p}\odot C_{m,p}(k_0)
+O_{k_0,\ell,P,K}(\lambda^{\ell+1}+\e^{P+1}),
\end{align*}
where we have set
\begin{multline*}
C_{m,p}(k_0):=\sum_{r=0}^p\sum_{n=0}^m  |k_0|^{1-2n-r} \bigg(\sum_{0\le l\le s \atop 2s-l=r}  \frac{(-1)^s\prod_{j=0}^{s-1}(2n-1+2j)}{2^{s-l}l!(s-l)!} (\frac{k_0}{|k_0|})^{\otimes l}
\otimes \Id^{\otimes (s-l)}\bigg)\\
\otimes\bigg( \frac{(-1)^n(2n)!}{(1-2n)4^n(n!)^2}\sum_{\alpha\in\N^n\atop|\alpha|=m-n}\sum_{\beta\in\N^n\atop|\beta|=p-r}\frac1{\beta!}\nabla^{\beta_1}\nu_{k_0}^{\alpha_1}\otimes\ldots\otimes\nabla^{\beta_n}\nu_{k_0}^{\alpha_n}\bigg).
\end{multline*}
Note that the first terms in this series can be explicitly computed,
\begin{align*}
C_{0,0}=|k_0|,\qquad C_{1,0}=\frac{\nu^0_{k_0}}{2|k_0|},\qquad
C_{0,1}=\frac{k_0}{|k_0|}
,\qquad C_{0,2}=
\frac{|k_0|^2 \Id- k_0 \otimes k_0}{2|k_0|^3}.
\end{align*}
Injecting this expansion into~\eqref{eq:approx-V-last}, we are lead to
\begin{multline*}
R_{\lambda,\e,\pm}^{\ell;t}(x)\,=\,\int_{\R^d}e^{\pm it\sum_{m=0}^{\ell}\sum_{p=0}^P\lambda^{m}\e^{p}k^{\otimes p}\odot C_{m,p}(k_0)}\Big(\hat u^\circ (k)\mp\tfrac{i \hat v^\circ(k)}{|k_0|}\Big)\,e^{ik\cdot x}\,d^*k\\
+O_{k_0,\ell,P,K,u^\circ,v^\circ}\big(\lambda+\e+t(\lambda^{\ell+1}+\e^{P+1})\big),
\end{multline*}
where the approximation holds in $\Ld^2(\R^d)$.
In view of the scaling relation $\e=\lambda^\beta$, we naturally choose $P=\lfloor{\ell}/{\beta}\rfloor$. Injecting the explicit form of $C_{0,0},C_{1,0},C_{0,1}$ and combining with~\eqref{eq:approx-uU-ellfix-wave},
the conclusion follows.\qed


\addtocontents{toc}{\protect\setcounter{tocdepth}{0}}
\section*{Acknowledgements}
\addtocontents{toc}{\protect\setcounter{tocdepth}{1}}
The authors warmly thank
L\'aszl\'o Erd\H{o}s and
Tom Spencer for some stimulating discussions on this problem.
The work of MD was supported by F.R.S.-FNRS and by the CNRS-Momentum program.
Financial support is acknowledged from the European Research Council under the European Community's Seventh Framework Programme (FP7/2014-2019 Grant Agreement QUANTHOM 335410) and European Union’s Horizon 2020 research
and innovation programme under grant agreement No 864066.

\bibliographystyle{plain}

\def\cprime{$'$}\def\cprime{$'$} \def\cprime{$'$}

\end{document}